\documentclass[a4paper,11pt,fleqn]{article}
\usepackage[official]{eurosym}
\usepackage{amsmath}
\usepackage{amssymb}
\usepackage{theorem}
\usepackage[usenames,dvipsnames]{pstricks}
\usepackage{euscript}
\usepackage{graphicx}
\usepackage{charter}

\renewcommand{\baselinestretch}{1.1}

\newcommand{\email}[1]{\href{mailto:#1}{\nolinkurl{#1}}}
\definecolor{labelkey}{rgb}{0,0.08,0.45}
\definecolor{refkey}{rgb}{0,0.6,0.0}
\definecolor{Brown}{rgb}{0.45,0.0,0.05}
\definecolor{dgreen}{rgb}{0.00,0.49,0.00}
\definecolor{dblue}{rgb}{0,0.08,0.75}
\RequirePackage[dvips,colorlinks,hyperindex]{hyperref} 
\hypersetup{linktocpage=true,citecolor=dblue,linkcolor=dgreen}
\PassOptionsToPackage{normalem}{ulem}
\usepackage{ulem}
\renewcommand{\leq}{\ensuremath{\leqslant}}
\renewcommand{\geq}{\ensuremath{\geqslant}}
\renewcommand{\le}{\ensuremath{\leqslant}}
\renewcommand{\ge}{\ensuremath{\geqslant}}
\newcommand{\minimize}[2]{\ensuremath{\underset{\substack{{#1}}}%
{\text{\rm minimize}}\;\;#2 }}

\newcommand{\EC}[2]{{\mathsf E}(#1\! \mid\! #2)}

\newcommand{\scal}[2]{{\left\langle{{#1}\mid{#2}}\right\rangle}}

\newcommand{\menge}[2]{\big\{{#1}~\big |~{#2}\big\}} 
 
\newcommand{\HHH}{{\ensuremath{\boldsymbol{\mathsf H}}}}
\newcommand{\GGG}{{\ensuremath{\boldsymbol{\mathsf G}}}}
\newcommand{\KKK}{{\ensuremath{\boldsymbol{\mathsf K}}}}
\newcommand{\HH}{\ensuremath{{\mathsf H}}}
\newcommand{\GG}{\ensuremath{{\mathsf G}}}
\newcommand{\KK}{\ensuremath{{\mathsf K}}}
\newcommand{\FF}{\ensuremath{{\EuScript F}}}
\newcommand{\XX}{\ensuremath{\EuScript{X}}}
\newcommand{\ZZZ}{\ensuremath{\EuScript{Z}}}
\newcommand{\XXX}{\ensuremath{\boldsymbol{\EuScript{X}}}}
\newcommand{\Sum}{\ensuremath{\displaystyle\sum}}
\newcommand{\emp}{\ensuremath{{\varnothing}}}

\newcommand{\Id}{\ensuremath{\text{\rm Id}}\,}
\newcommand{\ID}{\ensuremath{\mathbf{Id}}\,}
\newcommand{\cart}{\ensuremath{\raisebox{-0.5mm}{\mbox{\LARGE{$\times$}}}}}

\newcommand{\RR}{\ensuremath{\mathbb{R}}}

\newcommand{\BL}{\ensuremath{\EuScript B}}
\newcommand{\RPP}{\ensuremath{\left]0,+\infty\right[}}

\newcommand{\RX}{\ensuremath{\left]-\infty,+\infty\right]}}
\newcommand{\RXX}{\ensuremath{\left[-\infty,+\infty\right]}}
\newcommand{\EE}{\ensuremath{\mathsf E}}
\newcommand{\PP}{\ensuremath{\mathsf P}}
\newcommand{\as}{\ensuremath{\text{\rm $\PP$-a.s.}}}

\newcommand{\NN}{\ensuremath{\mathbb N}}

\newcommand{\ran}{\ensuremath{\text{\rm ran}\,}}

\newcommand{\zer}{\ensuremath{\text{\rm zer}\,}}
\newcommand{\pinf}{\ensuremath{{+\infty}}}

\newcommand{\dom}{\ensuremath{\text{\rm dom}\,}}
\newcommand{\prox}{\ensuremath{\text{\rm prox}}}

\newcommand{\reli}{\ensuremath{\text{\rm ri}\,}}
\newcommand{\infconv}{\ensuremath{\mbox{\small$\,\square\,$}}}

\newcommand{\EEE}{\ensuremath{\boldsymbol{\EuScript{E}}}}

\newcommand{\argmind}[2]{\ensuremath{\underset{\substack{{#1}}}%
{\mathrm{argmin}}\;\;#2 }}

\newtheorem{theorem}{Theorem}[section]
\newtheorem{lemma}[theorem]{Lemma}

\newtheorem{proposition}[theorem]{Proposition}
\theoremstyle{plain}{\theorembodyfont{\rmfamily}%
}
\theoremstyle{plain}{\theorembodyfont{\rmfamily}%
\newtheorem{assumption}[theorem]{Assumption}}
\theoremstyle{plain}{\theorembodyfont{\rmfamily}%
}
\theoremstyle{plain}{\theorembodyfont{\rmfamily}%
\newtheorem{problem}[theorem]{Problem}}
\theoremstyle{plain}{\theorembodyfont{\rmfamily}%
}
\theoremstyle{plain}{\theorembodyfont{\rmfamily}%
\newtheorem{remark}[theorem]{Remark}}
\theoremstyle{plain}{\theorembodyfont{\rmfamily}%
}
\theoremstyle{plain}{\theorembodyfont{\rmfamily}%
}

\numberwithin{equation}{section}
\begin{document}

\title{\sffamily A Class of Randomized Primal-Dual Algorithms for Distributed Optimization}
\author{Jean-Christophe Pesquet and Audrey Repetti \footnote{This work was supported by the CNRS
MASTODONS project (grant 2013 MesureHD).}\\[5mm]
\small
\small Universit\'e Paris-Est\\
\small Laboratoire d'Informatique Gaspard Monge -- CNRS UMR 8049\\
\small 77454 Marne la Vall\'ee Cedex 2, France\\
\small \email{first.last@univ-paris-est.fr}
}
\date{}

\maketitle
\thispagestyle{empty}

\vskip 8mm

\begin{abstract}
Based on a preconditioned version of the randomized block-coordinate forward-backward algorithm recently proposed in \cite{Combettes_P_2014_stochastic_qfbc},
several variants of block-coordinate primal-dual algorithms are designed in order to solve a wide array of monotone inclusion problems.
These methods rely on a sweep of blocks of variables which are activated at each iteration according to a random rule, and they
allow stochastic errors in the evaluation of the involved operators. 
Then, this framework is employed to derive block-coordinate primal-dual proximal algorithms for solving composite convex variational problems. 
The resulting algorithm implementations may be useful for reducing computational complexity and memory requirements.
Furthermore, we show that the proposed approach can be used to develop novel asynchronous distributed primal-dual algorithms in a multi-agent context.  
\end{abstract}



\newpage
\setcounter{page}{1}

\section{Introduction}
There has been recently a growing interest in primal-dual approaches for
finding a zero of a sum of monotone operators or minimizing a sum of proper lower-semicontinuous
convex functions (see \cite{Komodakis_N_2014_playing_dor} and the references therein). When various linear operators are involved in the formulation of the problem under investigation, 
solving jointly its primal and dual forms allows the design
of strategies where none of the linear operators needs to be inverted. Avoiding such inversions may offer a significant advantage in terms
of computational complexity when dealing with large-scale problems (see e.g. \cite{Bec_S_2014_j-nonlinear-conv-anal_alg_sps,Couprie_C_j-siam-is_dual_ctvb,Harizanov_S_2013_Epigraphical_psls,Jezierska_A_2012_p-icassp_primal_dps,Pustelnik_N_2014_j-sp_empirical_mdr,Repetti_A_2012_p-eusipco_penalized_wlsardcsdn,Teuber_T_2014_j-inv-prob_Minimization_pes}).

Various classes of fixed-point primal-dual algorithms have been developed, in particular those based on the forward-backward iteration \cite{Chambolle_A_2010_first_opdacpai,Chen_P_2013_j-inv-prob_prim_dfp,Combettes_P_2014_p-icip_forward_bvo,Combettes_P_2010_j-svva_dualization_srp,Combettes_P_2014_j-optim_Variable_mfb,Condat_L_2013_j-ota-primal-dsm,Esser_E_2010_j-siam-is_gen_fcf,Goldstein_T_2013_adaptive_pdh,He_B_2012_j-siam-is_conv_apd,Loris_I_2011_generalization_ist,Pock_T_2008_p-iccv_diagonal_pffo,Vu_B_2013_j-acm_spl_adm},
on the forward-backward-forward iteration \cite{Bec_S_2014_j-nonlinear-conv-anal_alg_sps,Bot_R_2014_jmiv_conv_primal_apd,Briceno_L_2011_j-siam-opt_mon_ssm,Combettes_P_2013_siam-opt_sys_smi,Combettes_P_2012_j-svva_pri_dsa}, on the Douglas-Rachford iteration \cite{Bot_R_2013_siam-opt_Dou_rtp,Combettes_P_2014_stochastic_qfbc}, or those derived from other principles \cite{Alghamdi_M_2014_optim-letters_primal_dmpi,Alotaibi_A_2014_Solving_ccm,Chen_G_1994_j-mp_pro_bdm,Nedic_A_2009_j-ota_subgradient_msp}.
This work is focused on the first class of primal-dual algorithms. When searching for a zero of a sum of monotone operators, the most recent versions
of these methods can exploit the properties of each operator either in an implicit manner, through the use of its resolvent, or
in a direct manner when the operator is cocoercive. When a sum of convex functions is minimized, this brings the ability either  to make use of the proximity 
operator of each function or to employ its gradient if the function is Lipschitz differentiable. As discussed in \cite{Bach_F_2012_j-ftml_optimization_sip,Combettes_P_2010_inbook_proximal_smsp,Parikh_N_2013_j-found-tml_prox_algo}, the proximity operator of a function is a versatile tool
in convex optimization for tackling possibly nonsmooth problems, but it may be sometimes preferable, in particular for complexity reasons,
to compute the gradient of the function when it enjoys some smoothness property.

Most of the aforementioned primal-dual methods make it possible to split the original problem in a sum of simpler terms whose associated operators
can be addressed individually, in a parallel manner, at each iteration of the algorithm. Our objective in this paper is to add more flexibility to the existing primal-dual
methods by allowing only a restricted number of these operators to be activated at each iteration. In the line of the work in \cite{Combettes_P_2014_stochastic_qfbc}, our approach will be grounded
on the use of random sweeping techniques which are applicable to algorithms generating (quasi-)Fej\'er monotone sequences.
One additional benefit of the proposed randomized approach is that it leads to algorithms which can be proved to be tolerant of stochastic errors satisfying some summability condition.

In the following, we will investigate two variants of forward-backward based primal-dual algorithms and we will design block-coordinate versions of both algorithms.
These block-coordinate methods may be interesting for their own sake in order to reduce memory and computational loads, but it turns out that they are also instrumental in 
developing distributed strategies. More precisely, we will be interested in multi-agent problems where the performed updates can be limited to a neighborhood of a small
number of agents 
in an asynchronous way. We will show that the proposed random distributed schemes apply not only to convex optimization problems, but
also to general monotone inclusion ones. It is worth noting that, in the variational case, some distributed primal-dual algorithms have already been proposed 
implementing subgradient steps \cite{Chang_T_2014_j-ieee-tac_distributed_coc,Yuan_D_2011_j-ieee-smc_distributed_pdsm}
(see also \cite{Towfic_Z_2014_Stability_plapdn} for applications to data networks). As a general feature of (unaveraged) subgradient methods, 
their convergence requires the use of step-sizes converging to zero. Making use of proximity operators, which can be viewed as implicit subgradient descent steps, 
allows less restrictive step-size choices to be made. For example, convergence of the iterates can be established for constant step-size values.

The remainder of the paper is organized as follows. In Section \ref{se:nota} we provide some relevant background on monotone operator theory and convex analysis, and we introduce our notation. In Section \ref{se:fback}, a preconditioned random block-coordinate version of the forward-backward iteration is presented. Based on this algorithm, in Section \ref{se:monocoord}, we propose novel block-coordinate primal-dual methods for constructing iteratively a zero of a sum of monotone operators, and we study their convergence. In Section \ref{se:coordopt}, similar block-coordinate primal-dual algorithms are developed for solving composite convex optimization problems. Finally, in Section \ref{se:dist}, we show how the proposed random block-coordinate approaches are able to provide distributed iterative solutions to monotone inclusion and convex variational problems.

\section{Notation}\label{se:nota}
The reader is referred to \cite{Bauschke_H_2011_book_con_amo} 
for background on monotone operator theory and convex analysis, and to
\cite{Fortet_R_1995_book_fonctions_daeh} for background on probability in Hilbert spaces.
Throughout this work, $(\Omega,\FF,\PP)$ is the underlying probability space.
For simplicity, the same notation $\scal{\cdot}{\cdot}$ (resp. $\|\cdot\|$) is used for the inner products
(resp. norms) which equip all the Hilbert spaces considered in this paper.
Let $\HH$ be a separable real Hilbert space with Borel $\sigma$-algebra $\mathcal{B}$. 
A $\HH$-valued random variable is a measurable 
map $x\colon(\Omega,\FF)\to(\HH,\mathcal{B})$. The smallest
$\sigma$-algebra generated by a family $\Phi$ of 
random variables is denoted by $\sigma(\Phi)$. The expectation is denoted by $\EE(\cdot)$.

Let $\GG$ be a real Hilbert space.
We denote by $\BL(\HH,\GG)$ the space of bounded linear operators 
from $\HH$ to $\GG$, and we set $\BL(\HH)=\BL(\HH,\HH)$.
Let $\mathsf{L}\in \BL(\HH,\GG)$, its adjoint is denoted by $\mathsf{L}^*$.
$\mathsf{L}\in \BL(\HH)$ is a strongly positive self-adjoint operator if $\mathsf{L}^*=\mathsf{L}$
and there exists $\alpha \in \RPP$ such that $(\forall \mathsf{x}\in \HH)$
$\scal{\mathsf{x}}{\mathsf{L}\mathsf{x}} \ge \alpha \|\mathsf{x}\|^2$.
Then, $\mathsf{L}$ is an isomorphism and its inverse is a strongly positive self-adjoint operator
in $\BL(\HH)$. The square root of a strongly positive operator $\mathsf{L}$ is denoted by $\mathsf{L}^{1/2}$
and its inverse by $\mathsf{L}^{-1/2}$.
$\Id$ denotes the identity operator on $\HH$. 

The power set of $\HH$ is denoted by $2^\HH$. Let $\mathsf{A}\colon\HH\to 2^{\HH}$ be a set-valued operator.
If, for every $\mathsf{x}\in \HH$, $\mathsf{A}\mathsf{x}$ is a singleton, then $\mathsf{A}$ will be identified
with a mapping from $\HH$ to $\HH$.
We denote 
by $\zer \mathsf{A}=\menge{\mathsf{x}\in\HH}{\mathsf{0}\in \mathsf{A}\mathsf{x}}$ the set of zeros 
of $\mathsf{A}$ and by $\mathsf{A}^{-1}\colon\HH\mapsto 2^{\HH}\colon \mathsf{u}\mapsto 
\menge{\mathsf{x}\in\HH}{\mathsf{u}\in \mathsf{A}\mathsf{x}}$ 
the inverse of $\mathsf{A}$.
Operator $\mathsf{A}$ is monotone if 
$(\forall(\mathsf{x},\mathsf{y})\in\HH^2)$
$(\forall \mathsf{u}\in \mathsf{A}\mathsf{x})$ $(\forall \mathsf{v}\in \mathsf{A}\mathsf{y})$
$\scal{\mathsf{x}-\mathsf{y}}{\mathsf{u}-\mathsf{v}}\geq~0$.
Such an operator is maximally monotone if there exists no other 
monotone operator whose graph includes the graph of $\mathsf{A}$.
$\mathsf{A}$ is $\beta$-strongly monotone for some 
$\beta\in\RPP$ if 
$(\forall(\mathsf{x},\mathsf{y})\in\HH^2)$
$(\forall \mathsf{u}\in \mathsf{A}\mathsf{x})$ $(\forall \mathsf{v}\in \mathsf{A}\mathsf{y})$
$\scal{\mathsf{x}-\mathsf{y}}{\mathsf{u}-\mathsf{v}}\geq \beta \|\mathsf{x}-\mathsf{y}\|^2$.
Let $\mathsf{B}$ be a single-valued operator from $\HH$ to $\HH$.
$\mathsf{B}$ is $\beta$-cocoercive for some 
$\beta\in\RPP$ if
$(\forall(\mathsf{x},\mathsf{y})\in\HH^2)$
$\scal{\mathsf{x}-\mathsf{y}}{\mathsf{B}\mathsf{x}-\mathsf{B}\mathsf{y}}\geq\beta\|\mathsf{B}\mathsf{x}-\mathsf{B}\mathsf{y}\|^2$.
Therefore, $\mathsf{B}$ is $\beta$-cocoercive if and only if $\mathsf{B}^{-1}\colon \HH \to 2^\HH$ is $\beta$-strongly monotone.
$\mathsf{B}$ is $\alpha$-averaged with $\alpha\in ]0,1[$ if
$(\forall(\mathsf{x},\mathsf{y})\in\HH^2)$ 
$\|\mathsf{B}\mathsf{x}-\mathsf{B}\mathsf{y}\|^2\leq
\|\mathsf{x}-\mathsf{y}\|^2-\displaystyle{\frac{1-\alpha}{\alpha}}
\|(\Id-\mathsf{B})\mathsf{x}-(\Id-\mathsf{B})\mathsf{y}\|^2$.
As a consequence of Minty's theorem, an operator $\mathsf{A}\colon\HH\to 2^\HH$ is maximally monotone if and only if its resolvent 
$\mathsf{J}_\mathsf{A}=(\Id+\mathsf{A})^{-1}$ is a firmly nonexpansive 
(i.e. 1-cocoercive) operator from $\HH$ to $\HH$.  
As a generalization of Moreau's decomposition formula, if $\mathsf{A}\colon \HH\to 2^\HH$ is maximally monotone, $\mathsf{U}$ is a strongly positive self-adjoint
operator in $\BL(\HH)$, and $\gamma\in \RPP$, then $\mathsf{J}_{\gamma \mathsf{U}\mathsf{A}}\colon\HH\to \HH$ is such that
\begin{equation}\label{e:moreaugen}
\mathsf{J}_{\gamma \mathsf{U}\mathsf{A}} = \mathsf{U}^{1/2} \mathsf{J}_{\gamma \mathsf{U}^{1/2}\mathsf{A}\mathsf{U}^{1/2}} \mathsf{U}^{-1/2}
= \Id-\gamma \mathsf{U} \mathsf{J}_{\gamma^{-1} \mathsf{U}^{-1}\mathsf{A}^{-1}}(\gamma^{-1} \mathsf{U}^{-1}\cdot)
\end{equation}
(see \cite[Example 3.9]{Combettes_P_2014_j-optim_Variable_mfb}).
The parallel sum of $\mathsf{A}\colon\HH\to 2^{\HH}$ and $\mathsf{C}\colon\HH\to 2^{\HH}$ is 
$\mathsf{A}\infconv \mathsf{C}=(\mathsf{A}^{-1}+ \mathsf{C}^{-1})^{-1}$.

The domain of a function $\mathsf{f}\colon \HH \to \RX$ is 
$\dom \mathsf{f} = \menge{\mathsf{x}\in \HH}{\mathsf{f}(\mathsf{x}) < \pinf}$. 
A function with a nonempty domain is said to be proper.
The class of proper, convex, lower-semicontinuous functions from
$\HH$ to $\RX$ is denoted by $\Gamma_0(\HH)$. If $\mathsf{f}\in\Gamma_0(\HH)$,
then the Moreau subdifferential of $\mathsf{f}$ is the maximally monotone operator
\begin{equation}
\partial \mathsf{f}\colon\HH\to 2^{\HH}\colon \mathsf{x} \mapsto\menge{\mathsf{u}\in\HH}{(\forall \mathsf{y}\in\HH)\;
\scal{\mathsf{y}-\mathsf{x}}{\mathsf{u}}+\mathsf{f}(\mathsf{x})\leq \mathsf{f}(\mathsf{y})}.
\end{equation} 
If $\mathsf{f}$ is proper and $\beta$-strongly convex for some $\beta\in \RPP$, then $\partial \mathsf{f}$ is $\beta$-strongly monotone.
If $\mathsf{f}\in \Gamma_0(\HH)$ is G\^ateaux-differentiable at $\mathsf{x}\in \HH$, then $\partial \mathsf{f}(\mathsf{x}) = \{\nabla \mathsf{f}(\mathsf{x})\}$
where $\nabla \mathsf{f}(\mathsf{x})$ is the gradient of $\mathsf{f}$ at $\mathsf{x}$.
$\mathsf{f}\colon \HH\to \RR$ is $\beta^{-1}$-Lipschitz differentiable for some $\beta \in \RPP$ if it is G\^ateaux-differentiable on $\HH$ and
$(\forall (\mathsf{x},\mathsf{y})\in \HH^2)$ $\beta \|\nabla \mathsf{f}(\mathsf{x})-\nabla \mathsf{f}(\mathsf{y})\|\le \|\mathsf{x}-\mathsf{y}\|$. 
The Baillon-Haddad theorem asserts that a differentiable convex function $\mathsf{f}$
defined on $\HH$ is $\beta^{-1}$-Lipschitz differentiable if and only if
its gradient $\nabla \mathsf{f}$ is $\beta$-cocoercive.
If $\Lambda$ is a nonempty subset of $\HH$, 
the indicator function of $\Lambda$ is
$(\forall \mathsf{x} \in \HH)$ $\iota_\Lambda(\mathsf{x}) = 0$ if $\mathsf{x}\in \Lambda$, and $\pinf$ otherwise.
This function
belongs to $\Gamma_0(\HH)$ if and only if $\Lambda$ is a nonempty closed convex set.
Its subdifferential $\partial \iota_\Lambda$ is the normal cone to $\Lambda$, denoted by $\mathsf{N}_\Lambda$.
The identity element of the parallel sum is $\mathsf{N}_{\{0\}}$.
 The inf-convolution of two functions $\mathsf{f}\colon \HH \to \RX$
and $\mathsf{h}\colon \HH \to \RX$ is defined as 
$\mathsf{f}\infconv \mathsf{h}\colon\HH\to\RXX\colon \mathsf{x}\mapsto
\inf_{\mathsf{y}\in\HH}\big(\mathsf{f}(\mathsf{y})+\mathsf{h}(\mathsf{x}-\mathsf{y})\big)$.
The identity element of the inf-convolution is $\iota_{\{0\}}$.
The conjugate of a function $\mathsf{f}\in \Gamma_0(\HH)$ is $\mathsf{f}^*\in \Gamma_0(\HH)$ such that
$(\forall  \mathsf{v}\in \HH)$ $\mathsf{f}^*(\mathsf{v}) = \sup_{\mathsf{x}\in\HH}\big(\scal{\mathsf{x}}{\mathsf{v}}-\mathsf{f}(\mathsf{x})\big)$.
We have then $\partial \mathsf{f}^* = (\partial \mathsf{f})^{-1}$.
Let $\mathsf{U}$ be a strongly positive self-adjoint operator in $\BL(\HH)$. 
The proximity operator of $\mathsf{f}\in\Gamma_0(\HH)$ relative to 
the metric induced by $\mathsf{U}$ is \cite[Section~XV.4]{Hiriart_Urruty_1996_book_2_convex_amaIf} 
\begin{equation}
\prox^{\mathsf{U}}_ {\mathsf{f}}\colon\HH\to\HH\colon \mathsf{x}\to
\argmind{\mathsf{y}\in\HH}{\mathsf{f}(\mathsf{y})+\frac12 \scal{\mathsf{x}-\mathsf{y}}{\mathsf{U}(\mathsf{x}-\mathsf{y}})}.
\end{equation}
We have thus $\prox^{\mathsf{U}}_ {\mathsf{f}} = \mathsf{J}_{\mathsf{U}^{-1}\partial \mathsf{f}}$.
When $\mathsf{U}=\Id$, we retrieve the standard definition of 
the proximity operator originally introduced in \cite{Moreau_J_1965_bsmf_Proximite_eddueh}.
If $\Lambda$ is a nonempty closed convex subset of $\HH$, $\Pi_\Lambda = \prox^{\Id}_{\iota_{\Lambda}}$
is the projector onto $\Lambda$. In the following, 
the relative interior of a subset $\Lambda$ of $\HH$
is denoted by $\reli \Lambda$.

Let $(\GG_i)_{1\leq i\leq m}$ be real Hilbert spaces.
$\GGG=\GG_1\oplus\cdots\oplus\GG_m$ is their Hilbert direct sum, i.e., their 
product space endowed with the scalar product 
$(\boldsymbol{\mathsf{x}},\boldsymbol{\mathsf{y}})\mapsto
\sum_{i=1}^m\scal{\mathsf{x}_i}{\mathsf{y}_i}$, where a generic element in
$\GGG$ is denoted by
$\boldsymbol{\mathsf{x}}=(\mathsf{x}_i)_{1\leq i\leq m}$ with $\mathsf{x}_i\in \GG_i$, for every
$i\in \{1,\ldots,m\}$. In addition, $\mathbb{D}_m=\{0,1\}^m\smallsetminus\{\boldsymbol{0}\}$
denotes the set of nonzero binary strings of length $m$.
We will keep on using this notation throughout the paper.

\section{A preconditioned random block-coordinate forward-backward algorithm}
\label{se:fback}
In this section, $m$ is a positive integer, $\KK_1,\ldots,\KK_m$ are separable
real Hilbert spaces, and $\KKK=\KK_1\oplus\cdots\oplus\KK_m$ is 
their Hilbert direct sum. 

The algorithms in this paper are rooted in the forward-backward iteration \cite{Combettes_P_2005_j-siam-mms_signal_rpfb}
(see \cite{Attouch_H_2010_j-siam-control-optim_Parallel_spmcmi} for examples of problems which can be solved by this method).
A block-coordinate version
of the forward-backward method was recently proposed in \cite[Section 5.2]{Combettes_P_2014_stochastic_qfbc}. Stochastic versions of this algorithm
were also presented in \cite{Necoara_I_2014_j-comp-opt-appl_random_cda,Richtarik_P_2014_j-math-programm_iteration_crbc} in a variational framework.
Now, we show how a preconditioning operator can be included in the 
block-coordinate forward-backward algorithm through a metric change.
\begin{proposition}
\label{p:FBPreconf}
Let $\boldsymbol{\mathsf{Q}}\colon\KKK\to2^\KKK$ be 
a maximally monotone operator and let $\boldsymbol{\mathsf R}\colon\KKK\to\KKK$ 
be a cocoercive operator. Assume that $\boldsymbol{\mathsf{Z}} = \zer(\boldsymbol{\mathsf Q}+\boldsymbol{\mathsf R})$ 
is nonempty. Let $\boldsymbol{\mathsf{V}}$
be a strongly positive self-adjoint operator in $\BL(\KKK)$
such that $\boldsymbol{\mathsf V}^{1/2} \boldsymbol{\mathsf R} \boldsymbol{\mathsf V}^{1/2}$
is $\vartheta$-cocoercive with $\vartheta \in \RPP$.
Let $(\gamma_n)_{n\in\NN}$ be a sequence in 
$\RR$ such that
$\inf_{n\in\NN}\gamma_n>0$ and $\sup_{n\in\NN}\gamma_n<2\vartheta$,
and let $(\lambda_n)_{n\in\NN}$ be a sequence in $\left]0,1\right]$
such that $\inf_{n\in\NN}\lambda_n>0$. 
Let $\boldsymbol{z}_0$, $(\boldsymbol{s}_n)_{n\in\NN}$, 
and $(\boldsymbol{t}_n)_{n\in\NN}$ be $\KKK$-valued random 
variables, and let $(\boldsymbol{\varepsilon}_n)_{n\in\NN}$ be identically distributed
$\mathbb{D}_m$-valued random variables. For every $n\in \NN$, set
$\boldsymbol{\mathsf{J}}_{\gamma_n\boldsymbol{\mathsf{V}}\boldsymbol{\mathsf{Q}}}
\colon\boldsymbol{\mathsf{z}}\mapsto
(\mathsf{T}_{i,n}\boldsymbol{\mathsf{z}})_{1\leq i\leq m}$
where $(\forall i \in \{1,\ldots,m\})$ $\mathsf{T}_{i,n}\colon\KKK\to\KK_i$,
iterate
\begin{equation}
\label{e:FBPrecond}
\begin{array}{l}
\text{for}\;n=0,1,\ldots\\
\left\lfloor
\begin{array}{l}
\boldsymbol{r}_n = \boldsymbol{\mathsf{V}}\boldsymbol{\mathsf{R}}\boldsymbol{z}_n\\ 
\text{for}\;i=1,\ldots,m\\
\left\lfloor
\begin{array}{l}
z_{i,n+1}=z_{i,n}+\lambda_n\varepsilon_{i,n}
\big(\mathsf{T}_{i,n}(\boldsymbol{z}_n-\gamma_n \boldsymbol{r}_n+\boldsymbol{s}_n)+
t_{i,n}-z_{i,n}\big),
\end{array}
\right.
\end{array}
\right.\\
\end{array}
\end{equation}
and set $(\forall n\in\NN)$ $\EEE_n=\sigma(\boldsymbol{\varepsilon}_n)$
and $\boldsymbol{\ZZZ}_n=\sigma(\boldsymbol{z}_0,\ldots,\boldsymbol{z}_n)$.
In addition, assume that the following hold:
\begin{enumerate}
\item
\label{p:nyc2014-04-03ii}
$\sum_{n\in\NN}\sqrt{\EC{\|\boldsymbol{s}_n\|^2}{
\boldsymbol{\ZZZ}_n}}<\pinf$ and
$\sum_{n\in\NN}\sqrt{\EC{\|\boldsymbol{t}_n\|^2}{
\boldsymbol{\ZZZ}_n}}<\pinf$ $\as$
\item
\label{p:nyc2014-04-03iv}
For every $n\in\NN$, $\EEE_n$ and $\boldsymbol{\ZZZ}_n$ 
are independent and $(\forall i\in\{1,\ldots,m\})$ 
$\PP[\varepsilon_{i,0} = 1] > 0$.
\end{enumerate}
Then $(\boldsymbol{z}_n)_{n\in\NN}$ converges weakly $\as$ to a 
$\boldsymbol{\mathsf Z}$-valued random variable.
\end{proposition}

\begin{proof}
We have $\boldsymbol{\mathsf{Z}}=\zer(\boldsymbol{\mathsf V}\boldsymbol{\mathsf{Q}}+\boldsymbol{\mathsf{V}}\boldsymbol{\mathsf R})\neq \emp$.
Since $\boldsymbol{\mathsf{V}}$ is a strongly positive self-adjoint operator, 
we can renorm the space $\KKK$ with the norm:
\begin{equation}
(\forall \boldsymbol{\mathsf z}\in \KKK)\qquad 
\|\boldsymbol{\mathsf z}\|_{\boldsymbol{\mathsf{V}}^{-1}} = \sqrt{\scal{\boldsymbol{\mathsf{z}}}{\boldsymbol{\mathsf{V}}^{-1} \boldsymbol{\mathsf{z}}}}.
\end{equation}
Let $\scal{\cdot}{\cdot}_{\boldsymbol{\mathsf{V}}^{-1}}$ denote the associated inner product.
In this renormed space, $\boldsymbol{\mathsf{V}}\boldsymbol{\mathsf{Q}}$ is maximally monotone. In addition,  
\begin{align}
\;\big(\forall (\boldsymbol{\mathsf{z}},\boldsymbol{\mathsf{z}}')\in \KKK^2\big)\quad
\|\boldsymbol{\mathsf{V}}\boldsymbol{\mathsf R}\boldsymbol{\mathsf{z}}-\boldsymbol{\mathsf{V}}\boldsymbol{\mathsf R}\boldsymbol{\mathsf{z}}'\|_{\boldsymbol{\mathsf{V}}^{-1}}^2
&= \|\boldsymbol{\mathsf{V}}^{1/2}\boldsymbol{\mathsf R}\boldsymbol{\mathsf{z}}-\boldsymbol{\mathsf{V}}^{1/2}\boldsymbol{\mathsf R}\boldsymbol{\mathsf{z}}'\|^2\nonumber\\
&\le \vartheta^{-1} \scal{\boldsymbol{\mathsf{V}}^{-1/2}\boldsymbol{\mathsf{z}}-\boldsymbol{\mathsf{V}}^{-1/2}\boldsymbol{\mathsf{z}}'}
{\boldsymbol{\mathsf{V}}^{1/2}\boldsymbol{\mathsf R}\boldsymbol{\mathsf{z}}-\boldsymbol{\mathsf{V}}^{1/2}\boldsymbol{\mathsf R}\boldsymbol{\mathsf{z}}'}\nonumber\\
&= \vartheta^{-1} \scal{\boldsymbol{\mathsf{z}}-\boldsymbol{\mathsf{z}}'}
{\boldsymbol{\mathsf R}\boldsymbol{\mathsf{z}}-\boldsymbol{\mathsf R}\boldsymbol{\mathsf{z}}'}\nonumber\\
&= \vartheta^{-1} \scal{\boldsymbol{\mathsf{z}}-\boldsymbol{\mathsf{z}}'}
{\boldsymbol{\mathsf{V}}\boldsymbol{\mathsf R}\boldsymbol{\mathsf{z}}-\boldsymbol{\mathsf{V}}\boldsymbol{\mathsf R}\boldsymbol{\mathsf{z}}'}_{\boldsymbol{\mathsf{V}}^{-1}},
\end{align}
which shows that $\boldsymbol{\mathsf{V}}\boldsymbol{\mathsf R}$ is $\vartheta$-cocoercive in $(\KKK,\|\cdot\|_{\boldsymbol{\mathsf{V}}^{-1}})$.
A forward-backward iteration can thus be employed to find an element of $\boldsymbol{\mathsf{Z}}$ by composing
operators $\boldsymbol{\mathsf{J}}_{\gamma_n\boldsymbol{\mathsf{V}}\boldsymbol{\mathsf{Q}}}$ and 
$\ID-\gamma_n \boldsymbol{\mathsf{V}}\boldsymbol{\mathsf{R}}$. In $(\KKK,\|\cdot\|_{\boldsymbol{\mathsf{V}}^{-1}})$, the first operator is firmly nonexpansive 
(hence, $1/2$-averaged) and the second one is $\gamma_n/(2\vartheta)$-averaged \cite[Proposition 4.33]{Bauschke_H_2011_book_con_amo}.
The relaxed randomized algorithm given in \cite[Section 4]{Combettes_P_2014_stochastic_qfbc} then takes the form
\eqref{e:FBPrecond}. The convergence result follows from \cite[Theorem~4.1]{Combettes_P_2014_stochastic_qfbc} by noticing that Assumption \ref{p:nyc2014-04-03ii}
leads to
\begin{align}
&\sum_{n\in\NN}\sqrt{\EC{\|\boldsymbol{s}_n\|_{\boldsymbol{\mathsf{V}}^{-1}}^2}{
\boldsymbol{\ZZZ}_n}}\le \sqrt{\|\boldsymbol{\mathsf{V}}^{-1}\|} \sum_{n\in\NN}\sqrt{\EC{\|\boldsymbol{s}_n\|^2}{
\boldsymbol{\ZZZ}_n}} <\pinf\\
&\sum_{n\in\NN}\sqrt{\EC{\|\boldsymbol{t}_n\|_{\boldsymbol{\mathsf{V}}^{-1}}^2}{
\boldsymbol{\ZZZ}_n}} \le \sqrt{\|\boldsymbol{\mathsf{V}}^{-1}\|} \sum_{n\in\NN}\sqrt{\EC{\|\boldsymbol{t}_n\|^2}{
\boldsymbol{\ZZZ}_n}} <\pinf
\end{align}
and that  weak convergences in the sense of $\scal{\cdot}{\cdot}$ and $\scal{\cdot}{\cdot}_{\boldsymbol{\mathsf{V}}^{-1}}$ are equivalent.
\end{proof}
\begin{remark}\ 
\begin{enumerate}
\item If $\boldsymbol{\mathsf R} = \boldsymbol{\mathsf L}^* \widetilde{\boldsymbol{\mathsf R}} \boldsymbol{\mathsf L}$
where $\boldsymbol{\mathsf L}  \in \BL(\KKK,\widetilde{\KKK})$, $\widetilde{\boldsymbol{\mathsf R}}\colon \widetilde{\KKK}
\to \widetilde{\KKK}$ is $\widetilde{\vartheta}$-cocoercive with $\widetilde{\vartheta}\in \RPP$, and $\widetilde{\KKK}$ is a separable real  Hilbert space, then $\boldsymbol{\mathsf V}^{1/2} \boldsymbol{\mathsf R} \boldsymbol{\mathsf V}^{1/2}$
is $\vartheta$-cocoercive for every strongly positive self-adjoint operator $\boldsymbol{\mathsf{V}}\in \BL(\KKK)$
such that $\vartheta \|\boldsymbol{\mathsf L} \boldsymbol{\mathsf V} \boldsymbol{\mathsf L}^*\| = \widetilde{\vartheta}$.
\item At iteration $n\in \NN$, $\boldsymbol{s}_{n}$ and $\boldsymbol{t}_{n}$ can be viewed as error terms
when applying $\boldsymbol{\mathsf{R}}$ and $\boldsymbol{\mathsf{J}}_{\gamma_n\boldsymbol{\mathsf{V}}\boldsymbol{\mathsf{Q}}}$, respectively.
The ability to consider summable stochastic errors offers more freedom than the assumption of summable deterministic errors which is often adopted in the literature. Note however that   
relative error models are considered in \cite{Monteiro_Svaiter_2010_Optim_Online,Solodov_Svaiter_2001_Num_Funct_Anal_Optim,Svaiter_2014_jota}.
\item Let $n\in\NN^*$. In view of \eqref{e:FBPrecond}, $\EEE_n$ and $\boldsymbol{\ZZZ}_n$ are independent if
$\boldsymbol{\varepsilon}_n$ is independent of 
$\big(\boldsymbol{z}_0,(\boldsymbol{\varepsilon}_{n'},\boldsymbol{s}_{n'},\boldsymbol{t}_{n'})_{0\le n' < n}\big)$.
\end{enumerate}
\end{remark}

\section{Block-coordinate primal-dual algorithms for composite monotone inclusion problems}
\label{se:monocoord}
In the rest of this section, $p$ and $q$ are positive integers,
$(\HH_j)_{1\leq j\leq p}$ and $(\GG_k)_{1\leq k\leq q}$ are separable real Hilbert spaces.
In addition, $\HHH=\HH_1\oplus\cdots\oplus\HH_p$ and $\GGG=\GG_1\oplus\cdots\oplus\GG_q$
denote the Hilbert direct sums of $(\HH_j)_{1\leq j\leq p}$ and
 $(\GG_k)_{1\leq k\leq q}$, respectively. We will also consider the product space $\KKK=\HHH\oplus\GGG$. 

\subsection{Problem}
The following problem involving monotone operators
which has drawn much attention in the last years (see e.g. \cite{Bot_R_2013_siam-opt_Dou_rtp,Briceno_L_2012_nonlinear-anal_Douglas_rsms,Combettes_P_2013_siam-opt_sys_smi,Combettes_P_2012_j-svva_pri_dsa,Pesquet_J_2012_j-pjpjoo_par_ipo,Raguet_H_2013_j-siam-is_generalized_fbs}) will play a prominent role throughout this work.
\begin{problem}
\label{prob:main}
For every 
$j\in\{1,\ldots,p\}$, let 
$\mathsf{A}_j\colon\HH_j\to 2^{\HH_j}$ be maximally monotone,
let $\mathsf{C}_j\colon\HH_j\to \HH_j$ be cocoercive
and, for every $k\in\{1,\ldots,q\}$, let 
$\mathsf{B}_k\colon\GG_k\to 2^{\GG_k}$ be maximally monotone, 
let $\mathsf{D}_k\colon\GG_k\to 2^{\GG_k}$ be maximally and strongly monotone,
and let $\mathsf{L}_{k,j}\in \BL(\HH_j,\GG_k)$. 
It is assumed that
\begin{align}
&(\forall k\in \{1,\ldots,q\})\qquad
\mathbb{L}_k = \menge{j\in \{1,\ldots,p\}}{\mathsf{L}_{k,j}\neq 0}\neq \emp,\label{e:defLk}\\
&(\forall j\in \{1,\ldots,p\})\qquad \mathbb{L}_j^* = \menge{k\in \{1,\ldots,q\}}{\mathsf{L}_{k,j}\neq 0}\neq \emp,\label{e:defLjs}
\end{align}
and that the set $\boldsymbol{\mathsf{F}}$ of solutions to the problem:
\begin{multline}
\label{e:primmon}
\text{find}\;\;{\mathsf{x}_1\in\HH_1,\ldots,\mathsf{x}_p\in\HH_p}
\;\;\text{such that}\;\;\\(\forall j\in\{1,\ldots,p\})\quad 0\in
\mathsf{A}_j\mathsf{x}_j+\mathsf{C}_j\mathsf{x}_j+\sum_{k=1}^q\mathsf{L}_{k,j}^*(\mathsf{B}_k\infconv\mathsf{D}_k)
\bigg(\sum_{j'=1}^p\mathsf{L}_{k,j'}\mathsf{x}_{j'}\bigg)
\end{multline}
is nonempty. We also consider the set $\boldsymbol{\mathsf{F}}^*$ of 
solutions to the dual problem:
\begin{multline}
\label{e:dualmon}
\text{find}\;\;{\mathsf{v}_1\in\GG_1,\ldots,\mathsf{v}_q\in\GG_q}
\;\;\text{such that}\;\;\\(\forall k\in\{1,\ldots,q\})\quad 0\in
-\Sum_{j=1}^p\mathsf{L}_{k,j}(\mathsf{A}_j^{-1}\infconv \mathsf{C}_j^{-1}) \bigg(-\Sum_{k'=1}^q
\mathsf{L}_{k',j}^*\mathsf{v}_{k'}\bigg)+\mathsf{B}_k^{-1}\mathsf{v}_k+\mathsf{D}_k^{-1}\mathsf{v}_k.
\end{multline}
Our objective is to find a pair $(\widehat{\boldsymbol{x}},\widehat{\boldsymbol{v}})$ of random variables
such that $\widehat{\boldsymbol{x}}$ is $\boldsymbol{\mathsf{F}}$-valued and 
$\widehat{\boldsymbol{v}}$ is $\boldsymbol{\mathsf{F}}^*$-valued.
\end{problem}
The previous problem 
can be recast as a search for a zero of the sum of two maximally monotone operators in the product space $\KKK$
as indicated below \cite{Condat_L_2013_j-ota-primal-dsm,Vu_B_2013_j-acm_spl_adm}.
\begin{proposition}\label{p:probmainbis}
Let us define $\boldsymbol{\mathsf A}\colon\HHH\to 2^\HHH\colon
\boldsymbol{\mathsf x}\mapsto\cart_{\!j=1}^{\!p}\mathsf{A}_j
\mathsf{x}_j$, $\boldsymbol{\mathsf B}\colon\GGG\to 2^\GGG\colon
\boldsymbol{\mathsf v}\mapsto\cart_{\!k=1}^{\!q}\mathsf{B}_k
\mathsf{v}_k$, 
$\boldsymbol{\mathsf C}\colon\HHH\to\HHH\colon\boldsymbol{\mathsf x}
\mapsto(\mathsf{C}_j\mathsf{x}_j)_{1 \leq j\leq p}$,
$\boldsymbol{\mathsf D}\colon\GGG\to 2^\GGG\colon
\boldsymbol{\mathsf v}\mapsto\cart_{\!k=1}^{\!q}\mathsf{D}_k
\mathsf{v}_k$, and $\boldsymbol{\mathsf{L}}\colon\HHH\to\GGG\colon
\boldsymbol{\mathsf{x}}\mapsto\big(\sum_{j=1}^p\mathsf{L}_{k,j}
\mathsf{x}_j\big)_{1\leq k\leq q}$.
Let us now introduce the operators
\begin{align}
\label{e:maximal1}
\boldsymbol{\mathsf{Q}}\colon\qquad\;\;\KKK&\to 2^{\KKK}\nonumber\\
(\boldsymbol{\mathsf{x}},\boldsymbol{\mathsf{v}})&\mapsto (\boldsymbol{\mathsf{A}}\boldsymbol{\mathsf{x}}+\boldsymbol{\mathsf{L}}^*
\boldsymbol{\mathsf{v}}) \times(-\boldsymbol{\mathsf{L}} \boldsymbol{\mathsf{x}}+\boldsymbol{\mathsf{B}}^{-1}\boldsymbol{\mathsf{v}})
\end{align}
and
\begin{align}
\label{e:maximal21}
\boldsymbol{\mathsf R}\colon\qquad\;\;\KKK&\to\KKK\nonumber\\
(\boldsymbol{\mathsf{x}},\boldsymbol{\mathsf{v}})&\mapsto 
\big(\boldsymbol{\mathsf C} \boldsymbol{\mathsf{x}},\boldsymbol{\mathsf D}^{-1} \boldsymbol{\mathsf{v}}\big).
\end{align}
\newpage
Then, the following hold:
\begin{enumerate}
\item \label{p:probmainbisi} $\boldsymbol{\mathsf{Q}}$ is maximally monotone and $\boldsymbol{\mathsf R}$ is cocoercive.
\item \label{p:probmainbisii}  $\boldsymbol{\mathsf{Z}}=\zer(\boldsymbol{\mathsf{Q}}+\boldsymbol{\mathsf R})$ is nonempty.
\item \label{p:probmainbisiii} A pair $(\widehat{\boldsymbol{x}},\widehat{\boldsymbol{v}})$ of random variables
is a solution to Problem~\ref{prob:main} if and only if $(\widehat{\boldsymbol{x}},\widehat{\boldsymbol{v}})$
is $\boldsymbol{\mathsf{Z}}$-valued.
\end{enumerate}
\end{proposition}
The above properties suggest employing the block-coordinate forward-backward algorithm developed in Section~\ref{se:fback}
to solve numerically Problem~\ref{prob:main}. According to the choice of the involved preconditioning operator,
several algorithms can be devised.
Subsequently, $\boldsymbol{\mathsf{L}}\in\BL(\HHH,\GGG)$ is defined as in Proposition \ref{p:probmainbis}.

\subsection{First algorithm subclass} \label{se:firstalgomon}
We state two preliminary results which will be useful in the derivation of the algorithms proposed in this section.
\begin{lemma}\label{le:1}
Let $\boldsymbol{\mathsf W} \in \BL(\HHH)$ and $\boldsymbol{\mathsf U} \in \BL(\GGG)$
be two strongly positive self-adjoint operators such that $\|\boldsymbol{\mathsf U}^{1/2}\boldsymbol{\mathsf{L}}\boldsymbol{\mathsf W}^{1/2}\| <1$.
\begin{enumerate}
\item \label{le:1i} The operator defined by
\begin{align} \label{e:invV1}
\!\!\!\!\!\!\!\boldsymbol{\mathsf V}'\colon\qquad\;\;\KKK&\to\KKK\nonumber\\
(\boldsymbol{\mathsf{x}},\boldsymbol{\mathsf{v}})&\mapsto 
\big(\boldsymbol{\mathsf W}^{-1} \boldsymbol{\mathsf{x}}-\boldsymbol{\mathsf{L}}^*\boldsymbol{\mathsf{v}},
-\boldsymbol{\mathsf{L}}\boldsymbol{\mathsf{x}}+\boldsymbol{\mathsf U}^{-1} \boldsymbol{\mathsf{v}}\big)
\end{align}
is a strongly positive self-adjoint operator in $\BL(\KKK)$. Its inverse given by
\begin{align}\label{e:defV1}
\!\!\!\!\!\!\!\boldsymbol{\mathsf V}\colon\qquad\;\;\KKK&\to\KKK\nonumber\\
(\boldsymbol{\mathsf{x}},\boldsymbol{\mathsf{v}})&\mapsto 
\big((\boldsymbol{\mathsf W}^{-1}-\boldsymbol{\mathsf{L}}^*\boldsymbol{\mathsf U}\boldsymbol{\mathsf{L}})^{-1} \boldsymbol{\mathsf{x}}+\boldsymbol{\mathsf W}\boldsymbol{\mathsf{L}}^*(\boldsymbol{\mathsf U}^{-1}-\boldsymbol{\mathsf{L}}\boldsymbol{\mathsf W}\boldsymbol{\mathsf{L}}^*)^{-1}\boldsymbol{\mathsf{v}}, (\boldsymbol{\mathsf U}^{-1}-\boldsymbol{\mathsf{L}}\boldsymbol{\mathsf W}\boldsymbol{\mathsf{L}}^*)^{-1}(\boldsymbol{\mathsf{L}}\boldsymbol{\mathsf W}\boldsymbol{\mathsf{x}}+\boldsymbol{\mathsf{v}})\big)
\end{align}
is also a strongly positive self-adjoint operator in $\BL(\KKK)$.
\item \label{le:1ii} Let $\boldsymbol{\mathsf{C}}\colon \HHH \to \HHH$, $\boldsymbol{\mathsf{D}}\colon \GGG \to 2^\GGG$, and
$\boldsymbol{\mathsf{R}}\colon \KKK \to \KKK$ be the operators defined in Proposition \ref{p:probmainbis}.
If $\boldsymbol{\mathsf W}^{1/2} \boldsymbol{\mathsf{C}} \boldsymbol{\mathsf W}^{1/2}$ is $\mu$-cocoercive with $\mu \in \RPP$ and
$\boldsymbol{\mathsf U}^{1/2} \boldsymbol{\mathsf{D}}^{-1} \boldsymbol{\mathsf U}^{1/2}$ is $\nu$-cocoercive with $\nu \in \RPP$, then, for every
$\alpha \in \RPP$, $\boldsymbol{\mathsf V}^{1/2} \boldsymbol{\mathsf{R}} \boldsymbol{\mathsf V}^{1/2}$ is $\vartheta_\alpha$-cocoercive,
where
\begin{equation}\label{e:varthetalpha}
\vartheta_\alpha = (1-\|\boldsymbol{\mathsf U}^{1/2}\boldsymbol{\mathsf{L}}\boldsymbol{\mathsf W}^{1/2}\|^2)
\min\big\{\mu (1+\alpha \|\boldsymbol{\mathsf U}^{1/2}\boldsymbol{\mathsf{L}}\boldsymbol{\mathsf W}^{1/2}\|)^{-1},
\nu (1+\alpha^{-1} \|\boldsymbol{\mathsf U}^{1/2}\boldsymbol{\mathsf{L}}\boldsymbol{\mathsf W}^{1/2}\|)^{-1}\big\}.
\end{equation}
\end{enumerate}
\end{lemma}
\begin{proof}
\noindent\ref{le:1i}
The operators $\boldsymbol{\mathsf W}^{-1}$ and  $\boldsymbol{\mathsf U}^{-1}$ being linear bounded and self-adjoint,
$\boldsymbol{\mathsf V}'$ is linear bounded and self-adjoint. In addition, for every $(\boldsymbol{\mathsf{x}},\boldsymbol{\mathsf{v}})\in\KKK$,
\begin{align}\label{e:strongposWinv1}
\scal{\boldsymbol{\mathsf{x}}}{(\boldsymbol{\mathsf W}^{-1}-\boldsymbol{\mathsf{L}}^*\boldsymbol{\mathsf U}\boldsymbol{\mathsf{L}})\boldsymbol{\mathsf{x}}}
&= \scal{\boldsymbol{\mathsf W}^{-1/2}\boldsymbol{\mathsf{x}}}{(\ID-\boldsymbol{\mathsf W}^{1/2}\boldsymbol{\mathsf{L}}^*\boldsymbol{\mathsf U}\boldsymbol{\mathsf{L}}\boldsymbol{\mathsf W}^{1/2})\boldsymbol{\mathsf W}^{-1/2}\boldsymbol{\mathsf{x}}}\nonumber\\
&= \scal{\boldsymbol{\mathsf{x}}}{\boldsymbol{\mathsf W}^{-1}\boldsymbol{\mathsf{x}}}-\scal{\boldsymbol{\mathsf W}^{-1/2}\boldsymbol{\mathsf{x}}}{\boldsymbol{\mathsf W}^{1/2}\boldsymbol{\mathsf{L}}^*\boldsymbol{\mathsf U}\boldsymbol{\mathsf{L}}\boldsymbol{\mathsf W}^{1/2}\boldsymbol{\mathsf W}^{-1/2}\boldsymbol{\mathsf{x}}}\nonumber\\
&\ge 
(1-\|\boldsymbol{\mathsf W}^{1/2}\boldsymbol{\mathsf{L}}^*\boldsymbol{\mathsf U}\boldsymbol{\mathsf{L}}\boldsymbol{\mathsf W}^{1/2}\|) 
\scal{\boldsymbol{\mathsf{x}}}{\boldsymbol{\mathsf W}^{-1}\boldsymbol{\mathsf{x}}}\nonumber\\
&\ge (1-\|\boldsymbol{\mathsf U}^{1/2}\boldsymbol{\mathsf{L}}\boldsymbol{\mathsf W}^{1/2}\|^2) \|\boldsymbol{\mathsf W}\|^{-1} \|\boldsymbol{\mathsf{x}}\|^2
\end{align}
and
\begin{align}\label{e:strongposUinv1}
\scal{\boldsymbol{\mathsf{v}}}{(\boldsymbol{\mathsf U}^{-1}-\boldsymbol{\mathsf{L}}\boldsymbol{\mathsf W}\boldsymbol{\mathsf{L}}^*)\boldsymbol{\mathsf{v}}}
&\ge (1-\|\boldsymbol{\mathsf U}^{1/2}\boldsymbol{\mathsf{L}}\boldsymbol{\mathsf W}\boldsymbol{\mathsf{L}}^*\boldsymbol{\mathsf U}^{1/2}\|) 
\scal{\boldsymbol{\mathsf{v}}}{\boldsymbol{\mathsf U}^{-1}\boldsymbol{\mathsf{v}}}\nonumber\\
&\ge (1-\|\boldsymbol{\mathsf U}^{1/2}\boldsymbol{\mathsf{L}}\boldsymbol{\mathsf W}^{1/2}\|^2) \|\boldsymbol{\mathsf U}\|^{-1} \|\boldsymbol{\mathsf{v}}\|^2.
\end{align}
We can deduce that
\begin{align}
\scal{(\boldsymbol{\mathsf{x}},\boldsymbol{\mathsf{v}})}{\boldsymbol{\mathsf V}'(\boldsymbol{\mathsf{x}},\boldsymbol{\mathsf{v}})}
&= \scal{\boldsymbol{\mathsf{x}}-\boldsymbol{\mathsf W}\boldsymbol{\mathsf{L}}^*\boldsymbol{\mathsf{v}}}{\boldsymbol{\mathsf W}^{-1}(\boldsymbol{\mathsf{x}}-\boldsymbol{\mathsf W}\boldsymbol{\mathsf{L}}^*\boldsymbol{\mathsf{v}})}+\scal{\boldsymbol{\mathsf{v}}}{(\boldsymbol{\mathsf U}^{-1}-\boldsymbol{\mathsf{L}}\boldsymbol{\mathsf W}\boldsymbol{\mathsf{L}}^*)\boldsymbol{\mathsf{v}}}\nonumber\\
&\ge (1-\|\boldsymbol{\mathsf U}^{1/2}\boldsymbol{\mathsf{L}}\boldsymbol{\mathsf W}^{1/2}\|^2) \|\boldsymbol{\mathsf U}\|^{-1} \|\boldsymbol{\mathsf{v}}\|^2
\end{align}
and similarly,
\begin{equation}
\scal{(\boldsymbol{\mathsf{x}},\boldsymbol{\mathsf{v}})}{\boldsymbol{\mathsf V}'(\boldsymbol{\mathsf{x}},\boldsymbol{\mathsf{v}})}
\ge (1-\|\boldsymbol{\mathsf U}^{1/2}\boldsymbol{\mathsf{L}}\boldsymbol{\mathsf W}^{1/2}\|^2) \|\boldsymbol{\mathsf W}\|^{-1} \|\boldsymbol{\mathsf{x}}\|^2.
\end{equation}
The latter two inequalities yield
\begin{align}
\scal{(\boldsymbol{\mathsf{x}},\boldsymbol{\mathsf{v}})}{\boldsymbol{\mathsf V}'(\boldsymbol{\mathsf{x}},\boldsymbol{\mathsf{v}})}
&\ge (1-\|\boldsymbol{\mathsf U}^{1/2}\boldsymbol{\mathsf{L}}\boldsymbol{\mathsf W}^{1/2}\|^2) 
\min\{\|\boldsymbol{\mathsf W}\|^{-1},\|\boldsymbol{\mathsf U}\|^{-1}\} \max\{\|\boldsymbol{\mathsf{x}}\|^2,\|\boldsymbol{\mathsf{v}}\|^2\}\nonumber\\
&\ge  \frac12 (1-\|\boldsymbol{\mathsf U}^{1/2}\boldsymbol{\mathsf{L}}\boldsymbol{\mathsf W}^{1/2}\|^2) 
\min\{\|\boldsymbol{\mathsf W}\|^{-1},\|\boldsymbol{\mathsf U}\|^{-1}\} (\|\boldsymbol{\mathsf{x}}\|^2+\|\boldsymbol{\mathsf{v}}\|^2).
\end{align}
This shows that $\boldsymbol{\mathsf V}'$ is a strongly positive operator. It is thus an isomorphism and its inverse is a strongly positive self-adjoint operator in $\BL(\KKK)$.

Furthermore,  \eqref{e:strongposWinv1} (resp. \eqref{e:strongposUinv1}) shows that $\boldsymbol{\mathsf W}^{-1}-\boldsymbol{\mathsf{L}}^*\boldsymbol{\mathsf U}\boldsymbol{\mathsf{L}}$
(resp. $\boldsymbol{\mathsf U}^{-1}-\boldsymbol{\mathsf{L}}\boldsymbol{\mathsf W}\boldsymbol{\mathsf{L}}^*$) is an isomorphism since it is a strongly positive self-adjoint operator in $\BL(\HHH)$ (resp. $\BL(\GGG)$). The expression of the inverse of $\boldsymbol{\mathsf V}'$ can be checked by direct calculations.

\noindent\ref{le:1ii}
Let $\alpha \in \RPP$.
 Showing that $\boldsymbol{\mathsf V}^{1/2}\boldsymbol{\mathsf R}\boldsymbol{\mathsf V}^{1/2}$ is $\vartheta_\alpha$-cocoercive is tantamount to establishing that
\begin{align}
&\;\big(\forall (\boldsymbol{\mathsf{z}},\boldsymbol{\mathsf{z}}')\in \KKK^2\big)\;\;
\scal{\boldsymbol{\mathsf{z}}-\boldsymbol{\mathsf{z}}'}{
\boldsymbol{\mathsf V}^{1/2}\boldsymbol{\mathsf R}\boldsymbol{\mathsf V}^{1/2} \boldsymbol{\mathsf{z}}-\boldsymbol{\mathsf V}^{1/2}\boldsymbol{\mathsf R}\boldsymbol{\mathsf V}^{1/2}\boldsymbol{\mathsf{z}}'}
\ge \vartheta_\alpha \|\boldsymbol{\mathsf V}^{1/2}\boldsymbol{\mathsf R}\boldsymbol{\mathsf V}^{1/2} \boldsymbol{\mathsf{z}}-\boldsymbol{\mathsf V}^{1/2}\boldsymbol{\mathsf R}\boldsymbol{\mathsf V}^{1/2}\boldsymbol{\mathsf{z}}'\|^2\nonumber\\
\Leftrightarrow &\; \big(\forall (\boldsymbol{\mathsf{z}},\boldsymbol{\mathsf{z}}')\in \KKK^2\big)\;\;\,
\scal{\boldsymbol{\mathsf{z}}-\boldsymbol{\mathsf{z}}'}{\boldsymbol{\mathsf R}\boldsymbol{\mathsf{z}}-\boldsymbol{\mathsf R}\boldsymbol{\mathsf{z}}'}
\ge \vartheta_\alpha \|\boldsymbol{\mathsf R} \boldsymbol{\mathsf{z}}-\boldsymbol{\mathsf R}\boldsymbol{\mathsf{z}}'\|_{\boldsymbol{\mathsf V}}^2.
\end{align}
Let $\boldsymbol{\mathsf{z}} = (\boldsymbol{\mathsf{x}},\boldsymbol{\mathsf{v}})\in \KKK$ and $\boldsymbol{\mathsf{z}}' = (\boldsymbol{\mathsf{x}}',\boldsymbol{\mathsf{v}}')\in \KKK$.
We have
\begin{align}\label{e:cocomaj1}
\|\boldsymbol{\mathsf R} \boldsymbol{\mathsf{z}}-\boldsymbol{\mathsf R}\boldsymbol{\mathsf{z}}'\|_{\boldsymbol{\mathsf V}}^2
= \;&\scal{\boldsymbol{\mathsf C}\boldsymbol{\mathsf{x}} - \boldsymbol{\mathsf C}\boldsymbol{\mathsf{x}}'}{(\boldsymbol{\mathsf W}^{-1}-\boldsymbol{\mathsf{L}}^*\boldsymbol{\mathsf U}\boldsymbol{\mathsf{L}})^{-1}(\boldsymbol{\mathsf C}\boldsymbol{\mathsf{x}} - \boldsymbol{\mathsf C}\boldsymbol{\mathsf{x}}')}\nonumber\\
&+ \scal{\boldsymbol{\mathsf D}^{-1}\boldsymbol{\mathsf{v}} - \boldsymbol{\mathsf D}^{-1}\boldsymbol{\mathsf{v}}'}{(\boldsymbol{\mathsf U}^{-1}-\boldsymbol{\mathsf{L}}\boldsymbol{\mathsf W}\boldsymbol{\mathsf{L}}^*)^{-1}(\boldsymbol{\mathsf D}^{-1}\boldsymbol{\mathsf{v}} - \boldsymbol{\mathsf D}^{-1}\boldsymbol{\mathsf{v}}')}\nonumber\\
&+ 2\scal{\boldsymbol{\mathsf C}\boldsymbol{\mathsf{x}} - \boldsymbol{\mathsf C}\boldsymbol{\mathsf{x}}'}{\boldsymbol{\mathsf W}\boldsymbol{\mathsf{L}}^*(\boldsymbol{\mathsf U}^{-1}-\boldsymbol{\mathsf{L}}\boldsymbol{\mathsf W}\boldsymbol{\mathsf{L}}^*)^{-1}(\boldsymbol{\mathsf D}^{-1}\boldsymbol{\mathsf{v}} - \boldsymbol{\mathsf D}^{-1}\boldsymbol{\mathsf{v}}')}.
\end{align}
On the other hand,
\begin{align}
\label{e:cocomaj2}
&\scal{\boldsymbol{\mathsf C}\boldsymbol{\mathsf{x}} - \boldsymbol{\mathsf C}\boldsymbol{\mathsf{x}}'}{(\boldsymbol{\mathsf W}^{-1}-\boldsymbol{\mathsf{L}}^*\boldsymbol{\mathsf U}\boldsymbol{\mathsf{L}})^{-1}(\boldsymbol{\mathsf C}\boldsymbol{\mathsf{x}} - \boldsymbol{\mathsf C}\boldsymbol{\mathsf{x}}')}\nonumber\\
&\;\;= \scal{\boldsymbol{\mathsf W}^{1/2}(\boldsymbol{\mathsf C}\boldsymbol{\mathsf{x}} - \boldsymbol{\mathsf C}\boldsymbol{\mathsf{x}}')}{(\ID-\boldsymbol{\mathsf W}^{1/2}\boldsymbol{\mathsf{L}}^*\boldsymbol{\mathsf U}\boldsymbol{\mathsf{L}}\boldsymbol{\mathsf W}^{1/2})^{-1}\boldsymbol{\mathsf W}^{1/2}(\boldsymbol{\mathsf C}\boldsymbol{\mathsf{x}} - \boldsymbol{\mathsf C}\boldsymbol{\mathsf{x}}')}\nonumber\\
&\;\;\le \|(\ID-\boldsymbol{\mathsf W}^{1/2}\boldsymbol{\mathsf{L}}^*\boldsymbol{\mathsf U}\boldsymbol{\mathsf{L}}\boldsymbol{\mathsf W}^{1/2})^{-1}\|\,
\|\boldsymbol{\mathsf C}\boldsymbol{\mathsf{x}} - \boldsymbol{\mathsf C}\boldsymbol{\mathsf{x}}'\|_{\boldsymbol{\mathsf W}}^2\nonumber\\
&\;\;= (1-\|\boldsymbol{\mathsf U}^{1/2}\boldsymbol{\mathsf{L}}\boldsymbol{\mathsf W}^{1/2}\|^2)^{-1} 
\|\boldsymbol{\mathsf C}\boldsymbol{\mathsf{x}} - \boldsymbol{\mathsf C}\boldsymbol{\mathsf{x}}'\|_{\boldsymbol{\mathsf W}}^2\,, 
\end{align}
\begin{align}
\label{e:cocomaj3}
&\scal{\boldsymbol{\mathsf D}^{-1}\boldsymbol{\mathsf{v}} - \boldsymbol{\mathsf D}^{-1}\boldsymbol{\mathsf{v}}'}{(\boldsymbol{\mathsf U}^{-1}-\boldsymbol{\mathsf{L}}\boldsymbol{\mathsf W}\boldsymbol{\mathsf{L}}^*)^{-1})(\boldsymbol{\mathsf D}^{-1}\boldsymbol{\mathsf{v}} - \boldsymbol{\mathsf D}^{-1}\boldsymbol{\mathsf{v}}')}\nonumber\\
&\;\;\le (1-\|\boldsymbol{\mathsf U}^{1/2}\boldsymbol{\mathsf{L}}\boldsymbol{\mathsf W}^{1/2}\|^2)^{-1}
\|\boldsymbol{\mathsf D}^{-1}\boldsymbol{\mathsf{v}} - \boldsymbol{\mathsf D}^{-1}\boldsymbol{\mathsf{v}}'\|_{\boldsymbol{\mathsf U}}^2\,,\\
\label{e:cocomaj4}
&\scal{\boldsymbol{\mathsf C}\boldsymbol{\mathsf{x}} - \boldsymbol{\mathsf C}\boldsymbol{\mathsf{x}}'}{\boldsymbol{\mathsf W}\boldsymbol{\mathsf{L}}^*(\boldsymbol{\mathsf U}^{-1}-\boldsymbol{\mathsf{L}}\boldsymbol{\mathsf W}\boldsymbol{\mathsf{L}}^*)^{-1})(\boldsymbol{\mathsf D}^{-1}\boldsymbol{\mathsf{v}} - \boldsymbol{\mathsf D}^{-1}\boldsymbol{\mathsf{v}}')}\nonumber\\
&\;\;\le \|\boldsymbol{\mathsf W}^{1/2}(\boldsymbol{\mathsf C}\boldsymbol{\mathsf{x}} - \boldsymbol{\mathsf C}\boldsymbol{\mathsf{x}}')\|
\|\boldsymbol{\mathsf W}^{1/2}\boldsymbol{\mathsf{L}}^*\boldsymbol{\mathsf U}^{1/2}(\ID-\boldsymbol{\mathsf U}^{1/2}\boldsymbol{\mathsf{L}}
\boldsymbol{\mathsf W}\boldsymbol{\mathsf{L}}^*\boldsymbol{\mathsf U}^{1/2})^{-1}\| \|\boldsymbol{\mathsf U}^{1/2}(\boldsymbol{\mathsf D}^{-1}\boldsymbol{\mathsf{v}} - \boldsymbol{\mathsf D}^{-1}\boldsymbol{\mathsf{v}}')\|\nonumber\\
&\;\;\le \|\boldsymbol{\mathsf U}^{1/2}\boldsymbol{\mathsf{L}}\boldsymbol{\mathsf W}^{1/2}\| (1-\|\boldsymbol{\mathsf U}^{1/2}\boldsymbol{\mathsf{L}}
\boldsymbol{\mathsf W}^{1/2}\|^2)^{-1} \|\boldsymbol{\mathsf C}\boldsymbol{\mathsf{x}} - \boldsymbol{\mathsf C}\boldsymbol{\mathsf{x}}'\|_{\boldsymbol{\mathsf W}}
\|\boldsymbol{\mathsf D}^{-1}\boldsymbol{\mathsf{v}} - \boldsymbol{\mathsf D}^{-1}\boldsymbol{\mathsf{v}}'\|_{\boldsymbol{\mathsf U}}\nonumber\\
&\;\;\le \frac12\|\boldsymbol{\mathsf U}^{1/2}\boldsymbol{\mathsf{L}}\boldsymbol{\mathsf W}^{1/2}\| (1-\|\boldsymbol{\mathsf U}^{1/2}\boldsymbol{\mathsf{L}}
\boldsymbol{\mathsf W}^{1/2}\|^2)^{-1} (\alpha \|\boldsymbol{\mathsf C}\boldsymbol{\mathsf{x}} - \boldsymbol{\mathsf C}\boldsymbol{\mathsf{x}}'\|_{\boldsymbol{\mathsf W}}^2
+\alpha^{-1}\|\boldsymbol{\mathsf D}^{-1}\boldsymbol{\mathsf{v}} - \boldsymbol{\mathsf D}^{-1}\boldsymbol{\mathsf{v}}'\|_{\boldsymbol{\mathsf U}}^2).
\end{align}
 Altogether, \eqref{e:cocomaj1}-\eqref{e:cocomaj4} and the cocoercivity assumptions on $\boldsymbol{\mathsf W}^{1/2} \boldsymbol{\mathsf C} \boldsymbol{\mathsf W}^{1/2}$
and $\boldsymbol{\mathsf U}^{1/2} \boldsymbol{\mathsf D}^{-1} \boldsymbol{\mathsf U}^{1/2}$ lead to the inequalities
\begin{align}
\|\boldsymbol{\mathsf R} \boldsymbol{\mathsf{z}}-\boldsymbol{\mathsf R}\boldsymbol{\mathsf{z}}'\|_{\boldsymbol{\mathsf V}}^2
&\le (1-\|\boldsymbol{\mathsf U}^{1/2}\boldsymbol{\mathsf{L}}\boldsymbol{\mathsf W}^{1/2}\|^2)^{-1} \big(
(1+\alpha \|\boldsymbol{\mathsf U}^{1/2}\boldsymbol{\mathsf{L}}\boldsymbol{\mathsf W}^{1/2}\|) 
\|\boldsymbol{\mathsf C}\boldsymbol{\mathsf{x}} - \boldsymbol{\mathsf C}\boldsymbol{\mathsf{x}}'\|_{\boldsymbol{\mathsf W}}^2\nonumber\\
&\quad + (1+\alpha^{-1} \|\boldsymbol{\mathsf U}^{1/2}\boldsymbol{\mathsf{L}}\boldsymbol{\mathsf W}^{1/2}\|)
\|\boldsymbol{\mathsf D}^{-1}\boldsymbol{\mathsf{v}} - \boldsymbol{\mathsf D}^{-1}\boldsymbol{\mathsf{v}}'\|_{\boldsymbol{\mathsf U}}^2\big)\nonumber\\
&\le (1-\|\boldsymbol{\mathsf U}^{1/2}\boldsymbol{\mathsf{L}}\boldsymbol{\mathsf W}^{1/2}\|^2)^{-1} \big(
\mu^{-1} (1+\alpha \|\boldsymbol{\mathsf U}^{1/2}\boldsymbol{\mathsf{L}}\boldsymbol{\mathsf W}^{1/2}\|) 
\scal{\boldsymbol{\mathsf{x}}-\boldsymbol{\mathsf{x}}'}{\boldsymbol{\mathsf C}\boldsymbol{\mathsf{x}} - \boldsymbol{\mathsf C}\boldsymbol{\mathsf{x}}'}\nonumber\\
&\quad + \nu^{-1} (1+\alpha^{-1} \|\boldsymbol{\mathsf U}^{1/2}\boldsymbol{\mathsf{L}}\boldsymbol{\mathsf W}^{1/2}\|)
 \scal{\boldsymbol{\mathsf{v}}-\boldsymbol{\mathsf{v}}'}{\boldsymbol{\mathsf D}^{-1}\boldsymbol{\mathsf{v}} - \boldsymbol{\mathsf D}^{-1}\boldsymbol{\mathsf{v}}'}\big)\nonumber\\
& \le \vartheta_\alpha^{-1}\scal{\boldsymbol{\mathsf{z}}-\boldsymbol{\mathsf{z}}'}{\boldsymbol{\mathsf R}\boldsymbol{\mathsf{z}}-\boldsymbol{\mathsf R}\boldsymbol{\mathsf{z}}'}.
\end{align}
\end{proof}

\begin{remark}\ \label{re:condFB1}
\begin{enumerate}
\item \label{re:condFB1i} In \eqref{e:varthetalpha}, we can simply choose $\alpha = 1$, yielding the cocoercivity constant
\begin{equation}
\vartheta_{1} = (1-\|\boldsymbol{\mathsf U}^{1/2}\boldsymbol{\mathsf{L}}\boldsymbol{\mathsf W}^{1/2}\|)\min\{\mu,\nu\}.
\end{equation}
A tighter value of this constant is $\vartheta_{\widehat{\alpha}}$ where $\widehat{\alpha}$ is the maximizer of 
$\alpha \mapsto \vartheta_\alpha$ on $\RPP$. It can be readily shown that
\begin{equation}
\widehat{\alpha} = \frac{\mu-\nu+\sqrt{(\mu-\nu)^2+4 \mu \nu \|\boldsymbol{\mathsf U}^{1/2}\boldsymbol{\mathsf{L}}\boldsymbol{\mathsf W}^{1/2}\|^2}}{2  \nu\|\boldsymbol{\mathsf U}^{1/2}\boldsymbol{\mathsf{L}}\boldsymbol{\mathsf W}^{1/2}\|}.
\end{equation}
\item \label{re:condFB1ii} When $\boldsymbol{\mathsf{D}}^{-1} = \boldsymbol{\mathsf{0}}$, the positive constant $\nu$ can be chosen arbitrarily large. A cocoercivity constant of
$\boldsymbol{\mathsf V}^{1/2}\boldsymbol{\mathsf R}\boldsymbol{\mathsf V}^{1/2}$ is then equal to
\begin{equation}
\lim_{\substack{\alpha \to 0\\\alpha > 0}} \vartheta_{\alpha} = (1-\|\boldsymbol{\mathsf U}^{1/2}\boldsymbol{\mathsf{L}}\boldsymbol{\mathsf W}^{1/2}\|^2)\mu.
\end{equation}
\end{enumerate}
\end{remark}

\begin{lemma}\label{lem:resespprod1}
Let $\boldsymbol{\mathsf{A}}\colon \HHH \to 2^\HHH$, $\boldsymbol{\mathsf{B}}\colon \GGG \to 2^\GGG$, 
$\boldsymbol{\mathsf{C}}\colon \HHH \to \HHH$, $\boldsymbol{\mathsf{D}}\colon \GGG \to 2^\GGG$, $\boldsymbol{\mathsf{Q}}\colon \KKK \to 2^\KKK$, and $\boldsymbol{\mathsf{R}}\colon \KKK \to \KKK$.
Let $\boldsymbol{\mathsf W} \in \BL(\HHH)$ and $\boldsymbol{\mathsf U} \in \BL(\GGG)$
be two strongly positive self-adjoint operators such that $\|\boldsymbol{\mathsf U}^{1/2}\boldsymbol{\mathsf{L}}\boldsymbol{\mathsf W}^{1/2}\| <1$.
Let $\boldsymbol{\mathsf{V}} \in \BL(\KKK)$ be defined by \eqref{e:defV1}.
For every $\boldsymbol{\mathsf{z}}= (\boldsymbol{\mathsf{x}},\boldsymbol{\mathsf{v}})\in \KKK$ and 
$(\boldsymbol{\mathsf{c}},\boldsymbol{\mathsf{e}})\in \KKK$, let
\begin{equation}\label{e:expresprod1}
\begin{cases}
\boldsymbol{\mathsf{y}} =
\boldsymbol{\mathsf{J}}_{\boldsymbol{\mathsf{W}}\boldsymbol{\mathsf{A}}}(\boldsymbol{\mathsf{x}}-\boldsymbol{\mathsf{W}}(\boldsymbol{\mathsf{L}}^*\boldsymbol{\mathsf{v}}
+ \boldsymbol{\mathsf{C}}\boldsymbol{\mathsf{x}}+\boldsymbol{\mathsf{c}}))\\
\boldsymbol{\mathsf{u}} = \boldsymbol{\mathsf{J}}_{\boldsymbol{\mathsf{U}}\boldsymbol{\mathsf{B}}^{-1}}(\boldsymbol{\mathsf{v}}+\boldsymbol{\mathsf{U}}(\boldsymbol{\mathsf{L}}(2\boldsymbol{\mathsf{y}}-\boldsymbol{\mathsf{x}})- \boldsymbol{\mathsf{D}}^{-1}\boldsymbol{\mathsf{v}}+\boldsymbol{\mathsf{e}})).
\end{cases}
\end{equation}
Then,
$(\boldsymbol{\mathsf{y}},\boldsymbol{\mathsf{u}}) =\boldsymbol{\mathsf{J}}_{\boldsymbol{\mathsf{V}}\boldsymbol{\mathsf{Q}}}(\boldsymbol{\mathsf{z}}-\boldsymbol{\mathsf{V}}\boldsymbol{\mathsf{R}}\boldsymbol{\mathsf{z}}+\boldsymbol{\mathsf{s}})$
where 
\begin{equation}
\boldsymbol{\mathsf{s}} = \big((\boldsymbol{\mathsf{W}}^{-1}-\boldsymbol{\mathsf{L}}^*\boldsymbol{\mathsf{U}}\boldsymbol{\mathsf{L}})^{-1}
(\boldsymbol{\mathsf{L}}^*\boldsymbol{\mathsf{U}}\boldsymbol{\mathsf{e}}-\boldsymbol{\mathsf{c}}),
(\boldsymbol{\mathsf{U}}^{-1}-\boldsymbol{\mathsf{L}}\boldsymbol{\mathsf{W}}\boldsymbol{\mathsf{L}}^*)^{-1}
(\boldsymbol{\mathsf{e}}-\boldsymbol{\mathsf{L}}\boldsymbol{\mathsf{W}}\boldsymbol{\mathsf{c}})\big).
\end{equation}
\end{lemma}
\begin{proof}
Let $\boldsymbol{\mathsf{z}}= (\boldsymbol{\mathsf{x}},\boldsymbol{\mathsf{v}})\in \KKK$ and let $\boldsymbol{\mathsf{s}}=(\boldsymbol{\mathsf{c}}',\boldsymbol{\mathsf{e}}')\in \KKK$.
We have the following equivalences:
\begin{align}
&(\boldsymbol{\mathsf{y}},\boldsymbol{\mathsf{u}}) = \boldsymbol{\mathsf{J}}_{\boldsymbol{\mathsf{V}}\boldsymbol{\mathsf{Q}}}(\boldsymbol{\mathsf{z}}-\boldsymbol{\mathsf{V}}\boldsymbol{\mathsf{R}}\boldsymbol{\mathsf{z}}+\boldsymbol{\mathsf{s}})\nonumber\\
\Leftrightarrow \quad & \boldsymbol{\mathsf{z}}-\boldsymbol{\mathsf{V}}\boldsymbol{\mathsf{R}}\boldsymbol{\mathsf{z}}+\boldsymbol{\mathsf{s}} \in (\ID+\boldsymbol{\mathsf{V}}\boldsymbol{\mathsf{Q}}) (\boldsymbol{\mathsf{y}},\boldsymbol{\mathsf{u}})\nonumber\\
\Leftrightarrow \quad & 
\boldsymbol{\mathsf{V}}^{-1}(\boldsymbol{\mathsf{z}}+\boldsymbol{\mathsf{s}}-(\boldsymbol{\mathsf{y}},\boldsymbol{\mathsf{u}}))
-\boldsymbol{\mathsf{R}}\boldsymbol{\mathsf{z}} \in \boldsymbol{\mathsf{Q}} (\boldsymbol{\mathsf{y}},\boldsymbol{\mathsf{u}})\nonumber\\
\Leftrightarrow \quad & 
\begin{cases}
\boldsymbol{\mathsf{W}}^{-1}(\boldsymbol{\mathsf{x}}-\boldsymbol{\mathsf{y}}+\boldsymbol{\mathsf{c}}')-\boldsymbol{\mathsf{L}}^*(\boldsymbol{\mathsf{v}}+\boldsymbol{\mathsf{e}}')
- \boldsymbol{\mathsf{C}}\boldsymbol{\mathsf{x}} \in \boldsymbol{\mathsf{A}} \boldsymbol{\mathsf{y}}\\
\boldsymbol{\mathsf{U}}^{-1}(\boldsymbol{\mathsf{v}}-\boldsymbol{\mathsf{u}}+\boldsymbol{\mathsf{e}}')+\boldsymbol{\mathsf{L}}(2\boldsymbol{\mathsf{y}}-\boldsymbol{\mathsf{x}}-\boldsymbol{\mathsf{c}}')
- \boldsymbol{\mathsf{D}}^{-1}\boldsymbol{\mathsf{v}} \in \boldsymbol{\mathsf{B}}^{-1} \boldsymbol{\mathsf{u}}
\end{cases}\label{e:tempexpres1}\\
\Leftrightarrow \quad & 
\begin{cases}
\boldsymbol{\mathsf{x}}+\boldsymbol{\mathsf{c}}'-\boldsymbol{\mathsf{W}}(\boldsymbol{\mathsf{L}}^*(\boldsymbol{\mathsf{v}}+\boldsymbol{\mathsf{e}}')
+ \boldsymbol{\mathsf{C}}\boldsymbol{\mathsf{x}}) \in (\ID+\boldsymbol{\mathsf{W}}\boldsymbol{\mathsf{A}}) \boldsymbol{\mathsf{y}}\\
\boldsymbol{\mathsf{v}}+\boldsymbol{\mathsf{e}}'+\boldsymbol{\mathsf{U}}(\boldsymbol{\mathsf{L}}(2\boldsymbol{\mathsf{y}}-\boldsymbol{\mathsf{x}}-\boldsymbol{\mathsf{c}}')
- \boldsymbol{\mathsf{D}}^{-1}\boldsymbol{\mathsf{v}}) \in (\ID+\boldsymbol{\mathsf{U}}\boldsymbol{\mathsf{B}}^{-1}) \boldsymbol{\mathsf{u}}
\end{cases}\nonumber\\
\Leftrightarrow \quad & 
\begin{cases}
\boldsymbol{\mathsf{y}} =
\boldsymbol{\mathsf{J}}_{\boldsymbol{\mathsf{W}}\boldsymbol{\mathsf{A}}}\big(\boldsymbol{\mathsf{x}}+\boldsymbol{\mathsf{c}}'
-\boldsymbol{\mathsf{W}}(\boldsymbol{\mathsf{L}}^*(\boldsymbol{\mathsf{v}}+\boldsymbol{\mathsf{e}}')
+ \boldsymbol{\mathsf{C}}\boldsymbol{\mathsf{x}})\big)\\
\boldsymbol{\mathsf{u}} = \boldsymbol{\mathsf{J}}_{\boldsymbol{\mathsf{U}}\boldsymbol{\mathsf{B}}^{-1}}\big(\boldsymbol{\mathsf{v}}+\boldsymbol{\mathsf{e}}'+\boldsymbol{\mathsf{U}}(\boldsymbol{\mathsf{L}}(2\boldsymbol{\mathsf{y}}-\boldsymbol{\mathsf{x}}-\boldsymbol{\mathsf{c}}')- \boldsymbol{\mathsf{D}}^{-1}\boldsymbol{\mathsf{v}})\big),
\end{cases}
\end{align} 
where, in \eqref{e:tempexpres1}, we have used the expression of $\boldsymbol{\mathsf{Q}}$ in \eqref{e:maximal1} and
the expression of the inverse of $\boldsymbol{\mathsf{V}}$ given by \eqref{e:invV1}.

In order to conclude, let us note that, since $\|\boldsymbol{\mathsf U}^{1/2}\boldsymbol{\mathsf{L}}\boldsymbol{\mathsf W}^{1/2}\| <1$,
it has already been observed in the proof of Lemma \ref{le:1}\ref{le:1i} that $\boldsymbol{\mathsf{W}}^{-1}-\boldsymbol{\mathsf{L}}^*\boldsymbol{\mathsf{U}}\boldsymbol{\mathsf{L}}$
and $\boldsymbol{\mathsf{U}}^{-1}-\boldsymbol{\mathsf{L}}\boldsymbol{\mathsf{W}}\boldsymbol{\mathsf{L}}^*$
are isomorphisms (as a result of \eqref{e:strongposWinv1} and \eqref{e:strongposUinv1}). Thus, for every $(\boldsymbol{\mathsf{c}},\boldsymbol{\mathsf{e}})\in \KKK$,
\begin{equation}
\begin{cases}
\boldsymbol{\mathsf{W}}\boldsymbol{\mathsf{c}} = \boldsymbol{\mathsf{W}}\boldsymbol{\mathsf{L}}^*\boldsymbol{\mathsf{e}}'-\boldsymbol{\mathsf{c}}'\\
\boldsymbol{\mathsf{U}}\boldsymbol{\mathsf{e}} = \boldsymbol{\mathsf{e}}'-\boldsymbol{\mathsf{U}}\boldsymbol{\mathsf{L}}\boldsymbol{\mathsf{c}}'
\end{cases}
\quad\Leftrightarrow \quad
\begin{cases}
\boldsymbol{\mathsf{c}}' = (\boldsymbol{\mathsf{W}}^{-1}-\boldsymbol{\mathsf{L}}^*\boldsymbol{\mathsf{U}}\boldsymbol{\mathsf{L}})^{-1}
(\boldsymbol{\mathsf{L}}^*\boldsymbol{\mathsf{U}}\boldsymbol{\mathsf{e}}-\boldsymbol{\mathsf{c}})\\
\boldsymbol{\mathsf{e}}' = (\boldsymbol{\mathsf{U}}^{-1}-\boldsymbol{\mathsf{L}}\boldsymbol{\mathsf{W}}\boldsymbol{\mathsf{L}}^*)^{-1}
(\boldsymbol{\mathsf{e}}-\boldsymbol{\mathsf{L}}\boldsymbol{\mathsf{W}}\boldsymbol{\mathsf{c}}).
\end{cases}
\end{equation}
\end{proof}

The above two lemmas allow us to obtain a first block-coordinate primal-dual algorithm to generate a solution to Problem \ref{prob:main}.
\begin{proposition}\label{p:algopdmon1}
Let 
\begin{equation}
\boldsymbol{\mathsf{W}}\colon \HHH \to \HHH\colon \boldsymbol{\mathsf{x}}\mapsto (\mathsf{W}_1 \mathsf{x}_1,\ldots,\mathsf{W}_p \mathsf{x}_p)
\quad\text{and}\quad  \boldsymbol{\mathsf{U}}\colon \GGG \to \GGG\colon \boldsymbol{\mathsf{v}}\mapsto (\mathsf{U}_1 \mathsf{v}_1,\ldots,\mathsf{U}_q \mathsf{v}_q)
\end{equation}
where, for every $j\in \{1,\ldots,p\}$, $\mathsf{W}_j$
is a strongly positive self-adjoint operator in $\BL(\HH_j)$
such that $\mathsf{W}_j^{1/2} \mathsf{C}_j \mathsf{W}_j^{1/2}$
is $\mu_j$-cocoercive with $\mu_j \in \RPP$ and, for every $k\in \{1,\ldots,q\}$, $\mathsf{U}_k$
is a strongly positive self-adjoint operator in $\BL(\GG_k)$
such that $\mathsf{U}_k^{1/2} \mathsf{D}_k^{-1} \mathsf{U}_k^{1/2}$
is $\nu_k$-cocoercive with $\nu_k \in \RPP$.
Suppose that
\begin{align}
&(\exists \alpha\in \RPP)\qquad  2\vartheta_\alpha > 1 \label{e:condvarthetaalpha}
\end{align}
where $\vartheta_\alpha$ is defined by \eqref{e:varthetalpha}
with $\mu = \min\{\mu_1,\ldots,\mu_p\}$ and $\nu=\min\{\nu_1,\ldots,\nu_q\}$. Let $(\lambda_n)_{n\in\NN}$ 
be a sequence in $\left]0,1\right]$ such that 
$\inf_{n\in\NN}\lambda_n>0$,
let $\boldsymbol{x}_0$, 
$(\boldsymbol{a}_n)_{n\in\NN}$, and 
$(\boldsymbol{c}_n)_{n\in\NN}$ be $\HHH$-valued random variables,
let $\boldsymbol{v}_0$, 
$(\boldsymbol{b}_n)_{n\in\NN}$, and 
$(\boldsymbol{d}_n)_{n\in\NN}$ be $\GGG$-valued random 
variables, and let $(\boldsymbol{\varepsilon}_n)_{n\in\NN}$ be identically distributed 
$\mathbb{D}_{p+q}$-valued random variables. 
Iterate
\begin{equation}\label{e:PDcoord1}
\begin{array}{l}
\text{for}\;n=0,1,\ldots\\
\left\lfloor
\begin{array}{l}
\text{for}\;j=1,\ldots,p\\
\left\lfloor
\begin{array}{l}
\displaystyle y_{j,n} =
\varepsilon_{j,n}\Big(\mathsf{J}_{\mathsf{W}_j\mathsf{A}_j}\big(x_{j,n}-\mathsf{W}_j(\sum_{k\in \mathbb{L}_j^*} {\mathsf{L}^*_{k,j} v_{k,n}}
+ \mathsf{C}_j x_{j,n} +c_{j,n})\big)+a_{j,n}\Big)\\
x_{j,n+1} = x_{j,n}+\lambda_n \varepsilon_{j,n}  (y_{j,n}- x_{j,n})
\end{array}
\right.\\
\text{for}\;k=1,\ldots,q\\
\left\lfloor
\begin{array}{l}
\displaystyle u_{k,n} = \varepsilon_{p+k,n}\Big(\mathsf{J}_{\mathsf{U}_k\mathsf{B}_k^{-1}}\big(v_{k,n}+\mathsf{U}_k(\sum_{j \in \mathbb{L}_k}\mathsf{L}_{k,j} (2y_{j,n}-x_{j,n})- \mathsf{D}_k^{-1}v_{k,n}+d_{k,n})\big)+b_{k,n}\Big)\\
v_{k,n+1} = v_{k,n}+\lambda_n \varepsilon_{p+k,n} (u_{k,n}-v_{k,n}),
\end{array}
\right.
\end{array}
\right.\\
\end{array}
\end{equation}
and set $(\forall n\in\NN)$ $\EEE_n=\sigma(\boldsymbol{\varepsilon}_n)$ and $\boldsymbol{\XX}_n=
\sigma(\boldsymbol{x}_{n'},\boldsymbol{v}_{n'})_{0\leq n'\leq n}$.
In addition, assume that the following hold:
\begin{enumerate}
\item
\label{c:PDcoord1i}
$\sum_{n\in\NN}\sqrt{\EC{\|\boldsymbol{a}_n\|^2}
{\boldsymbol{\XX}_n}}<\pinf$,
$\sum_{n\in\NN}\sqrt{\EC{\|\boldsymbol{b}_n\|^2}
{\boldsymbol{\XX}_n}}<\pinf$,
$\sum_{n\in\NN}\sqrt{\EC{\|\boldsymbol{c}_n\|^2}
{\boldsymbol{\XX}_n}}<\pinf$, and
$\sum_{n\in\NN}\sqrt{\EC{\|\boldsymbol{d}_n\|^2}
{\boldsymbol{\XX}_n}}<\pinf$ $\as$
\item
\label{c:PDcoord1ii}
For every $n\in\NN$, $\EEE_n$ and $\XXX_n$ are independent, and 
$(\forall k\in\{1,\ldots,q\})$ $\PP[\varepsilon_{p+k,0}=1]>0$.
\item \label{c:PDcoord1iii} For every $j\in\{1,\ldots,p\}$ and $n\in \NN$,
$\displaystyle \bigcup_{k\in \mathbb{L}_j^*} \menge{\omega\in \Omega}{\varepsilon_{p+k,n}(\omega) = 1}
\subset \menge{\omega\in \Omega}{\varepsilon_{j,n}(\omega) = 1}$.
\end{enumerate}
Then, $(\boldsymbol{x}_n)_{n\in\NN}$ converges weakly $\as$ to an 
$\boldsymbol{\mathsf{F}}$-valued random variable, and
$(\boldsymbol{v}_n)_{n\in\NN}$ 
converges weakly $\as$ to an
$\boldsymbol{\mathsf{F}}^*$-valued random variable.
\end{proposition}
  \begin{proof}
	In view of \ref{c:PDcoord1iii},  for every $j \in \{1,\ldots,p\}$, $ \max \big\{ \varepsilon_{j,n} , ( \varepsilon_{p+k,n})_{k\in \mathbb{L}_j^*} \big\} = \varepsilon_{j,n}$.
	Moreover, for every  $k\in \{1,\ldots,q\}$, $j\in \mathbb{L}_k \Leftrightarrow k \in \mathbb{L}_j^*$.
Iterations \eqref{e:PDcoord1} are thus equivalent to 
  \begin{equation}\label{e:PDcoord1prim}
\begin{array}{l}
\text{for}\;n=0,1,\ldots\\
\left\lfloor
\begin{array}{l}
\text{for}\;j=1,\ldots,p\\
\left\lfloor
\begin{array}{l}
\eta_{j,n} = \max \big\{ \varepsilon_{j,n} , ( \varepsilon_{p+k,n})_{k\in \mathbb{L}_j^*} \big\}\\
\displaystyle y_{j,n} =
\eta_{j,n}\Big(\mathsf{J}_{\mathsf{W}_j\mathsf{A}_j}\big(x_{j,n}-\mathsf{W}_j(\sum_{k=1}^q {\mathsf{L}^*_{k,j} v_{k,n}}
+ \mathsf{C}_j x_{j,n} +c_{j,n})\big)+a_{j,n}\Big)\\
x_{j,n+1} = x_{j,n}+\lambda_n \varepsilon_{j,n}  (y_{j,n}- x_{j,n})
\end{array}
\right.\\
\text{for}\;k=1,\ldots,q\\
\left\lfloor
\begin{array}{l}
\displaystyle u_{k,n} = \varepsilon_{p+k,n}\Big(\mathsf{J}_{\mathsf{U}_k\mathsf{B}_k^{-1}}\big(v_{k,n}+\mathsf{U}_k(\sum_{j \in \mathbb{L}_k}\mathsf{L}_{k,j} (2y_{j,n}-x_{j,n})- \mathsf{D}_k^{-1}v_{k,n}+d_{k,n})\big)+b_{k,n}\Big)\\
v_{k,n+1} = v_{k,n}+\lambda_n \varepsilon_{p+k,n} (u_{k,n}-v_{k,n}).
\end{array}
\right.
\end{array}
\right.\\
\end{array}
\end{equation}
On the other hand, according to Proposition \ref{p:probmainbis}\ref{p:probmainbisi}-\ref{p:probmainbisii}, 
$\boldsymbol{\mathsf{Q}}$ is maximally monotone, $\boldsymbol{\mathsf R}$ is cocoercive,
and $\boldsymbol{\mathsf{Z}}=\zer(\boldsymbol{\mathsf{Q}}+\boldsymbol{\mathsf R})\neq \emp$.
It can be noticed that \eqref{e:varthetalpha} and \eqref{e:condvarthetaalpha} imply that
$\|\boldsymbol{\mathsf U}^{1/2}\boldsymbol{\mathsf{L}}\boldsymbol{\mathsf W}^{1/2}\| <1$.
Thus, by virtue of Lemma~\ref{lem:resespprod1}, Algorithm \eqref{e:PDcoord1prim}
can be rewritten under the form of Algorithm \eqref{e:FBPrecond}, where $m=p+q$, $\boldsymbol{\mathsf{V}}$ is defined by \eqref{e:defV1}
and,  for every $n\in \NN$,
\begin{align}
& \boldsymbol{z}_n = (\boldsymbol{x}_n,\boldsymbol{v}_n), \label{e:equivFB1}\\
&\gamma_n = 1,\\
& \boldsymbol{\mathsf{J}}_{\boldsymbol{\mathsf{V}}\boldsymbol{\mathsf{Q}}}
\colon\boldsymbol{\mathsf{z}}\mapsto
(\mathsf{T}_{i,n}\boldsymbol{\mathsf{z}})_{1\leq i\leq m},\\
&(\forall j\in\{1,\ldots,p\})\qquad \mathsf{T}_{j,n}\colon \KKK \to \HH_j,\\
&(\forall k\in\{1,\ldots,q\})\qquad \mathsf{T}_{p+k,n}\colon \KKK \to \GG_k, \label{e:equivFB2}\\
& \boldsymbol{t}_n = (\boldsymbol{a}_n,\boldsymbol{b}_n),\label{e:deftn1}\\
&\boldsymbol{s}_n = \big((\boldsymbol{\mathsf{W}}^{-1}-\boldsymbol{\mathsf{L}}^*\boldsymbol{\mathsf{U}}\boldsymbol{\mathsf{L}})^{-1}
(\boldsymbol{\mathsf{L}}^*\boldsymbol{\mathsf{U}}\boldsymbol{e}_n-\boldsymbol{c}_n),
(\boldsymbol{\mathsf{U}}^{-1}-\boldsymbol{\mathsf{L}}\boldsymbol{\mathsf{W}}\boldsymbol{\mathsf{L}}^*)^{-1}
(\boldsymbol{e}_n-\boldsymbol{\mathsf{L}}\boldsymbol{\mathsf{W}}\boldsymbol{c}_n)\big),\\
& \boldsymbol{e}_n = 2 \boldsymbol{\mathsf{L}} \boldsymbol{a}_n + \boldsymbol{d}_n \label{e:defen1}.
\end{align}
Since $\boldsymbol{\mathsf W}$ and $\boldsymbol{\mathsf U}$ are
two strongly positive self-adjoint operators such that $\|\boldsymbol{\mathsf U}^{1/2}\boldsymbol{\mathsf{L}}\boldsymbol{\mathsf W}^{1/2}\| <1$,
Lemma~\ref{le:1}\ref{le:1i} allows us to claim that $\boldsymbol{\mathsf{V}}$ is a strongly positive self-adjoint operator in $\BL(\KKK)$.
In addition, for every $(\boldsymbol{\mathsf{x}},\boldsymbol{\mathsf{x}}')\in \HHH^2$,
\begin{align}
\scal{\boldsymbol{\mathsf{x}}-\boldsymbol{\mathsf{x}}'}
{\boldsymbol{\mathsf W}^{1/2} \boldsymbol{\mathsf{C}} \boldsymbol{\mathsf W}^{1/2}\boldsymbol{\mathsf{x}}-
\boldsymbol{\mathsf W}^{1/2} \boldsymbol{\mathsf{C}} \boldsymbol{\mathsf W}^{1/2}\boldsymbol{\mathsf{x}}'}
= &\sum_{j=1}^p \scal{\mathsf{x}_j-\mathsf{x}'_j}
{\mathsf{W}^{1/2}_j \mathsf{C}_j \mathsf{W}^{1/2}_j \mathsf{x}_j-
{\mathsf{W}^{1/2}_j \mathsf{C}_j \mathsf{W}^{1/2}_j \mathsf{x}_j'}}\nonumber\\
\ge & \sum_{j=1}^p \mu_j \|\mathsf{W}^{1/2}_j \mathsf{C}_j \mathsf{W}^{1/2}_j \mathsf{x}_j-
{\mathsf{W}^{1/2}_j \mathsf{C}_j \mathsf{W}^{1/2}_j \mathsf{x}_j'}\|^2\nonumber\\
\ge & \mu \|\boldsymbol{\mathsf W}^{1/2} \boldsymbol{\mathsf{C}} \boldsymbol{\mathsf W}^{1/2}\boldsymbol{\mathsf{x}}-
\boldsymbol{\mathsf W}^{1/2} \boldsymbol{\mathsf{C}} \boldsymbol{\mathsf W}^{1/2}\boldsymbol{\mathsf{x}}'\|^2.
\end{align}
Thus $\boldsymbol{\mathsf W}^{1/2} \boldsymbol{\mathsf{C}} \boldsymbol{\mathsf W}^{1/2}$ is $\mu$-cocoercive, and similarly,
$\boldsymbol{\mathsf U}^{1/2} \boldsymbol{\mathsf{D}}^{-1} \boldsymbol{\mathsf U}^{1/2}$ is $\nu$-cocoercive.
It follows from Lemma~\ref{le:1}\ref{le:1ii} that $\boldsymbol{\mathsf V}^{1/2} \boldsymbol{\mathsf R} \boldsymbol{\mathsf V}^{1/2}$
is $\vartheta_\alpha$-cocoercive, and our assumptions guarantee that $1 = \sup_{n\in\NN}\gamma_n<2\vartheta_\alpha$.
Moreover, it can be deduced from Condition \ref{c:PDcoord1i} and \eqref{e:deftn1}-\eqref{e:defen1} that
\begin{align}
& \sum_{n\in\NN}\sqrt{\EC{\|\boldsymbol{t}_n\|^2}{\boldsymbol{\XX}_n}} \le 
\sum_{n\in\NN}\sqrt{\EC{\|\boldsymbol{a}_n\|^2}{\boldsymbol{\XX}_n}}+\sum_{n\in\NN}\sqrt{\EC{\|\boldsymbol{b}_n\|^2}{\boldsymbol{\XX}_n}} < \pinf,\\
& \sum_{n\in\NN}\sqrt{\EC{\|\boldsymbol{s}_n\|^2}{\boldsymbol{\XX}_n}}\nonumber\\
&\;\;\le
\|(\boldsymbol{\mathsf{W}}^{-1}-\boldsymbol{\mathsf{L}}^*\boldsymbol{\mathsf{U}}\boldsymbol{\mathsf{L}})^{-1}\|
\Big(\sum_{n\in\NN}\sqrt{\EC{\|\boldsymbol{c}_n\|^2}{\boldsymbol{\XX}_n}}+2\|\boldsymbol{\mathsf{L}}^*\boldsymbol{\mathsf{U}}\boldsymbol{\mathsf{L}}\| 
\sum_{n\in\NN}\sqrt{\EC{\|\boldsymbol{a}_n\|^2}{\boldsymbol{\XX}_n}}\nonumber\\
&\;\;+ \|\boldsymbol{\mathsf{L}}^*\boldsymbol{\mathsf{U}}\|
\sum_{n\in\NN}\sqrt{\EC{\|\boldsymbol{d}_n\|^2}{\boldsymbol{\XX}_n}}\Big)
+\|(\boldsymbol{\mathsf{U}}^{-1}-\boldsymbol{\mathsf{L}}\boldsymbol{\mathsf{W}}\boldsymbol{\mathsf{L}}^*)^{-1}\|
\Big(2 \|\boldsymbol{\mathsf{L}}\| \sum_{n\in\NN}\sqrt{\EC{\|\boldsymbol{a}_n\|^2}{\boldsymbol{\XX}_n}}\nonumber\\
&\;\;+\sum_{n\in\NN}\sqrt{\EC{\|\boldsymbol{d}_n\|^2}{\boldsymbol{\XX}_n}}+\|\boldsymbol{\mathsf{L}}\boldsymbol{\mathsf{W}}\|
\sum_{n\in\NN}\sqrt{\EC{\|\boldsymbol{c}_n\|^2}{\boldsymbol{\XX}_n}}\Big) < \pinf.
\end{align}
In addition, since we have assumed that, for every $j \in \{1,\ldots,p\}$, $\mathbb{L}_j^* \neq \emp$,
\ref{c:PDcoord1ii} and \ref{c:PDcoord1iii} guarantee that Condition \ref{p:nyc2014-04-03iv} in Proposition \ref{p:FBPreconf}
is also fulfilled.
All the assumptions of Proposition \ref{p:FBPreconf} are then satisfied, which allows us to establish the almost sure convergence of $(\boldsymbol{x}_n,\boldsymbol{v}_n)_{n\in \NN}$
to a $\boldsymbol{\mathsf{Z}}$-valued random variable. Finally, Proposition \ref{p:probmainbis}\ref{p:probmainbisiii} ensures that the limit is an 
$\boldsymbol{\mathsf{F}}\times \boldsymbol{\mathsf{F}}^*$-valued random variable.
\end{proof}

\noindent A number of observations can be made on Proposition \ref{p:algopdmon1}.
\begin{remark}\ \label{re:mon1}
\begin{enumerate}
\item The Boolean random variables $(\varepsilon_{i,n})_{1\le i \le p+q}$ signal the variables $(x_{j,n})_{1\le j \le p}$ and $(v_{k,n})_{1\le k \le q}$
that are activated at each iteration $n$. From a computational standpoint, when some of them are zero-valued, no update of the associated variables
must be performed. Note that, in accordance with Condition \ref{c:PDcoord1iii},
for every $j\in \{1,\ldots,p\}$, $y_{j,n}$ needs to be computed not only when $x_{j,n}$ is activated,
but also when there exists $k\in \{1,\ldots,q\}$ such that $v_{k,n}$ is activated and $\mathsf{L}_{k,j} \neq 0$. 
\item  For every $n\in \NN$, $j\in \{1,\ldots,p\}$, and $k\in \{1,\ldots,q\}$, $a_{j,n}$, $b_{k,n}$, $c_{j,n}$,
and $d_{k,n}$ model stochastic errors possibly arising at iteration $n$, when applying $\mathsf{J}_{\mathsf{W}_j\mathsf{A}_j}$,
$\mathsf{J}_{\mathsf{U}_k\mathsf{B}_k^{-1}}$, $\mathsf{C}_j$, and $\mathsf{D}_k^{-1}$, respectively.
\item \label{re:mon1iii} Using the triangle and Cauchy-Schwarz inequalities yields
\begin{align}
(\boldsymbol{\mathsf x} \in \HHH)\qquad
\|\boldsymbol{\mathsf U}^{1/2} \boldsymbol{\mathsf{L}}\boldsymbol{\mathsf W}^{1/2}\boldsymbol{\mathsf x} \|^2 
&= \sum_{k=1}^q \Big\| \sum_{j=1}^p \mathsf{U}_k^{1/2}  \mathsf{L}_{k,j} \mathsf{W}_j^{1/2} \mathsf{x}_j\Big\|^2\nonumber\\
&\le \sum_{k=1}^q \Big(\sum_{j=1}^p \|\mathsf{U}_k^{1/2} \mathsf{L}_{k,j} \mathsf{W}_j^{1/2}\| \|\mathsf{x}_j\|\Big)^2\nonumber\\
&\le \sum_{k=1}^q 
\Big(\sum_{j=1}^p \|\mathsf{U}_k^{1/2} \mathsf{L}_{k,j} \mathsf{W}_j^{1/2} \|^2\Big) \Big(\sum_{j=1}^p \|\mathsf{x}_j \|^2\Big),
\end{align} 
which shows that
\begin{equation}\label{e:boundnormULW}
\|\boldsymbol{\mathsf U}^{1/2}\boldsymbol{\mathsf{L}}\boldsymbol{\mathsf W}^{1/2}\| 
\le \Big(\sum_{j=1}^p \sum_{k=1}^q \|\mathsf{U}_k^{1/2} \mathsf{L}_{k,j} \mathsf{W}_j^{1/2} \|^2\Big)^{1/2}.
\end{equation} 
For every $j\in \{1,\ldots,p\}$, let a cocoercivity constant of $\mathsf{C}_j$ be denoted by $\widetilde{\mu}_j\in \RPP$ and, for every $k\in \{1,\ldots,q\}$,
let a strong monotonicity constant of $\mathsf{D}_k$ be denoted by $\widetilde{\nu}_k\in \RPP$. Then, one can choose
\begin{equation}
\mu = \min\{(\|\mathsf{W}_j\|^{-1}\widetilde{\mu}_j)_{1\le j \le p}\},\qquad 
\nu = \min\{(\|\mathsf{U}_k\|^{-1}\widetilde{\nu}_k)_{1\le k \le q}\}.\label{e:mumutj}
\end{equation}
Therefore, by using Remark \ref{re:condFB1}\ref{re:condFB1i}, a sufficient condition for \eqref{e:condvarthetaalpha} to be satisfied with $\alpha=1$ is
\begin{align}
&\left(1-\Big(\sum_{j=1}^p \sum_{k=1}^q \|\mathsf{U}_k^{1/2} \mathsf{L}_{k,j} \mathsf{W}_j^{1/2} \|^2\Big)^{1/2}\right)
\min\{(\|\mathsf{W}_j\|^{-1}\widetilde{\mu}_j)_{1\le j \le p},(\|\mathsf{U}_k\|^{-1}\widetilde{\nu}_k)_{1\le k \le q}\}> \frac12. \label{e:condvarthetaalphabis}
\end{align}
When $\boldsymbol{\mathsf{D}}^{-1}=\boldsymbol{\mathsf{0}}$, in accordance with Remark~\ref{re:condFB1}\ref{re:condFB1ii}, this condition
can be replaced by
\begin{equation}
\left(1-\sum_{j=1}^p \sum_{k=1}^q \|\mathsf{U}_k^{1/2} \mathsf{L}_{k,j} \mathsf{W}_j^{1/2} \|^2\right)
\min\{(\|\mathsf{W}_j\|^{-1}\widetilde{\mu}_j)_{1\le j \le p}\}> \frac12. \label{e:condvarthetaalphater}
\end{equation}
\item The above algorithm extends a number of results existing in a deterministic setting, when $p=1$ and no random sweeping is performed.
In most of these works, $\mathsf{W}_1 = \tau \Id$ and, for every $k\in \{1,\ldots,q\}$, $\mathsf{U}_k = \rho_k \Id$ where $(\tau,\rho_1,\ldots,\rho_q) \in \RPP^{q+1}$.
In particular, in \cite{Vu_B_2013_j-acm_spl_adm}, a sufficient condition for \eqref{e:condvarthetaalphabis} to be satisfied is employed, while in \cite{Condat_L_2013_j-ota-primal-dsm} 
it is assumed that $\boldsymbol{\mathsf{D}}^{-1}=\boldsymbol{\mathsf{0}}$ and a condition similar to \eqref{e:condvarthetaalphater} is used.
The proposed block-coordinate algorithm also extends the results in \cite[Section 6]{Combettes_P_2014_j-optim_Variable_mfb} when a constant metric is considered.
\end{enumerate}
\end{remark}
Due to the symmetry existing between the primal and the dual problems, we can swap the roles of these two problems,
so leading to a symmetric form of Algorithm \eqref{e:PDcoord1}:
\begin{proposition}\label{p:algopdmon1sym}
Let $\boldsymbol{\mathsf{W}}$, $\boldsymbol{\mathsf{U}}$, $\mu$, and $\nu$
be defined as in Proposition \ref{p:algopdmon1}.
Suppose that \eqref{e:condvarthetaalpha} holds
where $\vartheta_\alpha$ is defined by \eqref{e:varthetalpha}.
Let $(\lambda_n)_{n\in\NN}$ be a sequence in $\left]0,1\right]$ such that 
$\inf_{n\in\NN}\lambda_n>0$,
let $\boldsymbol{x}_0$, 
$(\boldsymbol{a}_n)_{n\in\NN}$, and 
$(\boldsymbol{c}_n)_{n\in\NN}$ be $\HHH$-valued random variables,
let $\boldsymbol{v}_0$, 
$(\boldsymbol{b}_n)_{n\in\NN}$, and 
$(\boldsymbol{d}_n)_{n\in\NN}$ be $\GGG$-valued random 
variables, and let $(\boldsymbol{\varepsilon}_n)_{n\in\NN}$ be identically distributed 
$\mathbb{D}_{p+q}$-valued random variables. 
Iterate
\begin{equation}\label{e:PDcoord1sym}
\begin{array}{l}
\text{for}\;n=0,1,\ldots\\
\left\lfloor
\begin{array}{l}
\text{for}\;k=1,\ldots,q\\
\left\lfloor
\begin{array}{l}
\displaystyle u_{k,n} = \varepsilon_{p+k,n}\Big(\mathsf{J}_{\mathsf{U}_k\mathsf{B}_k^{-1}}\big(v_{k,n}+\mathsf{U}_k(\sum_{j\in\mathbb{L}_k}\mathsf{L}_{k,j} x_{j,n}- \mathsf{D}_k^{-1}v_{k,n}+d_{k,n})\big)+b_{k,n}\Big)\\
v_{k,n+1} = v_{k,n}+\lambda_n \varepsilon_{p+k,n} (u_{k,n}-v_{k,n})
\end{array}
\right.\\
\text{for}\;j=1,\ldots,p\\
\left\lfloor
\begin{array}{l}
\displaystyle y_{j,n} =
\varepsilon_{j,n}\Big(\mathsf{J}_{\mathsf{W}_j\mathsf{A}_j}\big(x_{j,n}-\mathsf{W}_j(\sum_{k\in \mathbb{L}_j^*} {\mathsf{L}^*_{k,j} (2u_{k,n}-v_{k,n})}
+ \mathsf{C}_j x_{j,n} +c_{j,n})\big)+a_{j,n}\Big)\\
x_{j,n+1} = x_{j,n}+\lambda_n \varepsilon_{j,n}  (y_{j,n}- x_{j,n}).
\end{array}
\right.\\
\end{array}
\right.
\end{array}
\end{equation}
In addition, assume that Condition~\ref{c:PDcoord1i} 
in Proposition \ref{p:algopdmon1} is satisfied where
$(\forall n\in\NN)$ $\EEE_n=\sigma(\boldsymbol{\varepsilon}_n)$ and $\boldsymbol{\XX}_n=
\sigma(\boldsymbol{x}_{n'},\boldsymbol{v}_{n'})_{0\leq n'\leq n}\,$, and that
the following hold:
\begin{enumerate}
\setcounter{enumi}{1}
\item
\label{c:PDcoord1iisym}
For every $n\in\NN$, $\EEE_n$ and $\XXX_n$ are independent, and 
$(\forall j\in\{1,\ldots,p\})$ $\PP[\varepsilon_{j,0}=1]>0$.
\item \label{c:PDcoord1iiisym} For every $k \in \{1,\ldots,q\}$ and $n\in \NN$,
$\displaystyle\bigcup_{j\in \mathbb{L}_k} \menge{\omega\in \Omega}{\varepsilon_{j,n}(\omega) = 1}
\subset \menge{\omega\in \Omega}{\varepsilon_{p+k,n}(\omega) = 1}$.
\end{enumerate}
Then, $(\boldsymbol{x}_n)_{n\in\NN}$ converges weakly $\as$ to an 
$\boldsymbol{\mathsf{F}}$-valued random variable, and
$(\boldsymbol{v}_n)_{n\in\NN}$ 
converges weakly $\as$ to an
$\boldsymbol{\mathsf{F}}^*$-valued random variable.
\end{proposition}

\subsection{Second algorithm subclass}
We now consider a diagonal form of the operator $\boldsymbol{\mathsf V}$, for which we 
proceed similarly to the approach followed in Section \ref{se:firstalgomon}.
\begin{lemma}\label{le:2}
Let $\boldsymbol{\mathsf W} \in \BL(\HHH)$ and $\boldsymbol{\mathsf U} \in \BL(\GGG)$
be two strongly positive self-adjoint operators such that $\|\boldsymbol{\mathsf U}^{1/2}\boldsymbol{\mathsf{L}}\boldsymbol{\mathsf W}^{1/2}\| <1$.
\begin{enumerate}
\item \label{le:2i} The operator defined by
\begin{align} \label{e:defV2}
\boldsymbol{\mathsf V}\colon\qquad\;\;\KKK&\to\KKK\nonumber\\
(\boldsymbol{\mathsf{x}},\boldsymbol{\mathsf{v}})&\mapsto 
\big(\boldsymbol{\mathsf W} \boldsymbol{\mathsf{x}},
(\boldsymbol{\mathsf U}^{-1}-\boldsymbol{\mathsf{L}}\boldsymbol{\mathsf W}\boldsymbol{\mathsf{L}}^*)^{-1}\boldsymbol{\mathsf{v}}\big)
\end{align}
is a strongly positive self-adjoint operator in $\BL(\KKK)$. 
\item \label{le:2ii} 
 Let $\boldsymbol{\mathsf{C}}\colon \HHH \to \HHH$, $\boldsymbol{\mathsf{D}}\colon \GGG \to 2^\GGG$, and
$\boldsymbol{\mathsf{R}}\colon \KKK \to \KKK$ be the operators defined in Proposition \ref{p:probmainbis}.
If $\boldsymbol{\mathsf W}^{1/2} \boldsymbol{\mathsf{C}} \boldsymbol{\mathsf W}^{1/2}$ is $\mu$-cocoercive with $\mu \in \RPP$ and
$\boldsymbol{\mathsf U}^{1/2} \boldsymbol{\mathsf{D}}^{-1} \boldsymbol{\mathsf U}^{1/2}$ is $\nu$-cocoercive with $\nu \in \RPP$, then
$\boldsymbol{\mathsf V}^{1/2} \boldsymbol{\mathsf{R}} \boldsymbol{\mathsf V}^{1/2}$ is $\vartheta$-cocoercive,
where
\begin{equation}\label{e:vartheta2}
\vartheta = \min\big\{\mu,\nu(1-\|\boldsymbol{\mathsf U}^{1/2}\boldsymbol{\mathsf{L}}\boldsymbol{\mathsf W}^{1/2}\|^2)\big\}.
\end{equation}
\end{enumerate} 
\end{lemma}

\begin{proof}
\ref{le:2i} First note that, as we have already shown in \eqref{e:strongposUinv1}, if $\|\boldsymbol{\mathsf U}^{1/2}\boldsymbol{\mathsf{L}}\boldsymbol{\mathsf W}^{1/2}\| <1$,
then $\boldsymbol{\mathsf U}^{-1}-\boldsymbol{\mathsf{L}}\boldsymbol{\mathsf W}\boldsymbol{\mathsf{L}}^*$ is a strongly positive operator in $\BL(\GGG)$ and it is thus 
an isomorphism. We have then, for every $(\boldsymbol{\mathsf{x}},\boldsymbol{\mathsf{v}}) \in \KKK$,
\begin{align}
\scal{(\boldsymbol{\mathsf{x}},\boldsymbol{\mathsf{v}})}{\boldsymbol{\mathsf V}(\boldsymbol{\mathsf{x}},\boldsymbol{\mathsf{v}})}
&=\scal{\boldsymbol{\mathsf{x}}}{\boldsymbol{\mathsf W}\boldsymbol{\mathsf{x}}}
+\scal{\boldsymbol{\mathsf{v}}}{(\boldsymbol{\mathsf U}^{-1}-\boldsymbol{\mathsf{L}}\boldsymbol{\mathsf W}\boldsymbol{\mathsf{L}}^*)^{-1}\boldsymbol{\mathsf{v}}}\nonumber\\
&\ge \|\boldsymbol{\mathsf W}^{-1}\|^{-1}\|\boldsymbol{\mathsf{x}}\|^2+\|\boldsymbol{\mathsf{U}}^{-1}-\boldsymbol{\mathsf{L}}\boldsymbol{\mathsf W}\boldsymbol{\mathsf{L}}^*\|^{-1} \|\boldsymbol{\mathsf{v}}\|^2\nonumber\\
&\ge \min\big\{\|\boldsymbol{\mathsf W}^{-1}\|^{-1},\|\boldsymbol{\mathsf{U}}^{-1}-\boldsymbol{\mathsf{L}}\boldsymbol{\mathsf W}\boldsymbol{\mathsf{L}}^*\|^{-1} \big\}(\|\boldsymbol{\mathsf{x}}\|^2+\|\boldsymbol{\mathsf{v}}\|^2).
\end{align}
Hence, $\boldsymbol{\mathsf V}$ is a strongly positive self-adjoint operator.

\noindent\ref{le:2ii} Let $\boldsymbol{\mathsf{z}} = (\boldsymbol{\mathsf{x}},\boldsymbol{\mathsf{v}})\in \KKK$ and $\boldsymbol{\mathsf{z}}' = (\boldsymbol{\mathsf{x}}',\boldsymbol{\mathsf{v}}')\in \KKK$.
We have
\begin{align}
\|\boldsymbol{\mathsf R} \boldsymbol{\mathsf{z}}&-\boldsymbol{\mathsf R}\boldsymbol{\mathsf{z}}'\|_{\boldsymbol{\mathsf V}}^2\nonumber\\
= \; & \scal{\boldsymbol{\mathsf C}\boldsymbol{\mathsf{x}} - \boldsymbol{\mathsf C}\boldsymbol{\mathsf{x}}'}{\boldsymbol{\mathsf W}(\boldsymbol{\mathsf C}\boldsymbol{\mathsf{x}} - \boldsymbol{\mathsf C}\boldsymbol{\mathsf{x}}')}+ \scal{\boldsymbol{\mathsf D}^{-1}\boldsymbol{\mathsf{v}} - \boldsymbol{\mathsf D}^{-1}\boldsymbol{\mathsf{v}}'}{(\boldsymbol{\mathsf U}^{-1}-\boldsymbol{\mathsf{L}}\boldsymbol{\mathsf W}\boldsymbol{\mathsf{L}}^*)^{-1}(\boldsymbol{\mathsf D}^{-1}\boldsymbol{\mathsf{v}} - \boldsymbol{\mathsf D}^{-1}\boldsymbol{\mathsf{v}}')}\nonumber\\
\le \; & \|\boldsymbol{\mathsf C}\boldsymbol{\mathsf{x}} - \boldsymbol{\mathsf C}\boldsymbol{\mathsf{x}}'\|_{\boldsymbol{\mathsf W}}^2
+\|(\ID-\boldsymbol{\mathsf U}^{1/2}\boldsymbol{\mathsf{L}}\boldsymbol{\mathsf W}\boldsymbol{\mathsf{L}}^*\boldsymbol{\mathsf U}^{1/2})^{-1}\| \|\boldsymbol{\mathsf D}^{-1}\boldsymbol{\mathsf{v}} - \boldsymbol{\mathsf D}^{-1}\boldsymbol{\mathsf{v}}'\|_{\boldsymbol{\mathsf U}}^2\nonumber\\
\le \; &  \mu^{-1} \scal{\boldsymbol{\mathsf{x}} - \boldsymbol{\mathsf{x}}'}{\boldsymbol{\mathsf C}\boldsymbol{\mathsf{x}} - \boldsymbol{\mathsf C}\boldsymbol{\mathsf{x}}'}
+\nu^{-1}(1-\|\boldsymbol{\mathsf U}^{1/2}\boldsymbol{\mathsf{L}}\boldsymbol{\mathsf W}^{1/2}\|^2)^{-1} 
\scal{\boldsymbol{\mathsf{v}} - \boldsymbol{\mathsf{v}}'}{\boldsymbol{\mathsf D}^{-1}\boldsymbol{\mathsf{v}} - \boldsymbol{\mathsf D}^{-1}\boldsymbol{\mathsf{v}}'}\nonumber\\
\le \; &  \max\big\{\mu^{-1},\nu^{-1}(1-\|\boldsymbol{\mathsf U}^{1/2}\boldsymbol{\mathsf{L}}\boldsymbol{\mathsf W}^{1/2}\|^2)^{-1}\big\} 
\scal{\boldsymbol{\mathsf{z}} - \boldsymbol{\mathsf{z}}'}{\boldsymbol{\mathsf R}\boldsymbol{\mathsf{z}} - \boldsymbol{\mathsf R}\boldsymbol{\mathsf{z}}'},
\end{align}
which, in view of the remark made at the beginning of the proof of Lemma~\ref{le:1}\ref{le:1ii}, shows that $\boldsymbol{\mathsf V}^{1/2}\boldsymbol{\mathsf R}\boldsymbol{\mathsf V}^{1/2}$
is $\vartheta$-cocoercive.
\end{proof}

\begin{lemma}\label{lem:resespprod2}
Let $\boldsymbol{\mathsf{B}}\colon \GGG \to 2^\GGG$, 
$\boldsymbol{\mathsf{C}}\colon \HHH \to \HHH$, $\boldsymbol{\mathsf{D}}\colon \GGG \to 2^\GGG$, $\boldsymbol{\mathsf{Q}}\colon \KKK \to 2^\KKK$, and $\boldsymbol{\mathsf{R}}\colon \KKK \to \KKK$.
Assume that the operator $\boldsymbol{\mathsf{A}}$ defined in Proposition \ref{p:probmainbis} is zero.
Let $\boldsymbol{\mathsf W} \in \BL(\HHH)$ and $\boldsymbol{\mathsf U} \in \BL(\GGG)$
be two strongly positive self-adjoint operators such that $\|\boldsymbol{\mathsf U}^{1/2}\boldsymbol{\mathsf{L}}\boldsymbol{\mathsf W}^{1/2}\| <1$.
Let $\boldsymbol{\mathsf{V}} \in \BL(\KKK)$ be defined by \eqref{e:defV2}.
For every $\boldsymbol{\mathsf{z}}= (\boldsymbol{\mathsf{x}},\boldsymbol{\mathsf{v}})\in \KKK$ and 
$(\boldsymbol{\mathsf{e}}_1,\boldsymbol{\mathsf{e}}_2)\in \KKK$, let
\begin{equation}\label{e:expresprod2}
\begin{cases}
\boldsymbol{\mathsf{u}} = \boldsymbol{\mathsf{J}}_{\boldsymbol{\mathsf{U}}\boldsymbol{\mathsf{B}}^{-1}}\left(\boldsymbol{\mathsf{v}}+\boldsymbol{\mathsf{U}}\big(\boldsymbol{\mathsf{L}}(
\boldsymbol{\mathsf{x}}-\boldsymbol{\mathsf{W}}(\boldsymbol{\mathsf{C}}\boldsymbol{\mathsf{x}}+\boldsymbol{\mathsf{L}}^*\boldsymbol{\mathsf{v}}))
- \boldsymbol{\mathsf{D}}^{-1}\boldsymbol{\mathsf{v}}+\boldsymbol{\mathsf{e}}_2\big)\right)\\
\boldsymbol{\mathsf{y}} = \boldsymbol{\mathsf{x}}-\boldsymbol{\mathsf{W}}(\boldsymbol{\mathsf{C}}\boldsymbol{\mathsf{x}}+\boldsymbol{\mathsf{L}}^*\boldsymbol{\mathsf{u}}
+\boldsymbol{\mathsf{e}}_1).
\end{cases}
\end{equation}
Then,
$(\boldsymbol{\mathsf{y}},\boldsymbol{\mathsf{u}}) =\boldsymbol{\mathsf{J}}_{\boldsymbol{\mathsf{V}}\boldsymbol{\mathsf{Q}}}(\boldsymbol{\mathsf{z}}-\boldsymbol{\mathsf{V}}\boldsymbol{\mathsf{R}}\boldsymbol{\mathsf{z}}+\boldsymbol{\mathsf{s}})$
where 
\begin{equation}\label{e:exps2}
\boldsymbol{\mathsf{s}} = \big(-\boldsymbol{\mathsf{W}}\boldsymbol{\mathsf{e}}_1,
(\boldsymbol{\mathsf{U}}^{-1}-\boldsymbol{\mathsf{L}}\boldsymbol{\mathsf{W}}\boldsymbol{\mathsf{L}}^*)^{-1}
(\boldsymbol{\mathsf{e}}_2+\boldsymbol{\mathsf{L}}\boldsymbol{\mathsf{W}}\boldsymbol{\mathsf{e}}_1)\big).
\end{equation}
\end{lemma}
\begin{proof}
Let $\boldsymbol{\mathsf{z}}= (\boldsymbol{\mathsf{x}},\boldsymbol{\mathsf{v}})\in \KKK$ and let $\boldsymbol{\mathsf{s}}=(\boldsymbol{\mathsf{e}}'_1,\boldsymbol{\mathsf{e}}'_2)\in \KKK$.
The following equivalences are obtained:
\begin{align}
&(\boldsymbol{\mathsf{y}},\boldsymbol{\mathsf{u}}) = \boldsymbol{\mathsf{J}}_{\boldsymbol{\mathsf{V}}\boldsymbol{\mathsf{Q}}}(\boldsymbol{\mathsf{z}}-
\boldsymbol{\mathsf{V}}\boldsymbol{\mathsf{R}}\boldsymbol{\mathsf{z}}+\boldsymbol{\mathsf{s}})\nonumber\\
\Leftrightarrow \quad & 
\boldsymbol{\mathsf{V}}^{-1}(\boldsymbol{\mathsf{z}}+\boldsymbol{\mathsf{s}}-(\boldsymbol{\mathsf{y}},\boldsymbol{\mathsf{u}}))
-\boldsymbol{\mathsf{R}}\boldsymbol{\mathsf{z}} \in \boldsymbol{\mathsf{Q}} (\boldsymbol{\mathsf{y}},\boldsymbol{\mathsf{u}})\nonumber\\
\Leftrightarrow \quad & 
\begin{cases}
\boldsymbol{\mathsf{W}}^{-1}(\boldsymbol{\mathsf{x}}-\boldsymbol{\mathsf{y}}+\boldsymbol{\mathsf{e}}'_1)-\boldsymbol{\mathsf{L}}^*\boldsymbol{\mathsf{u}}
- \boldsymbol{\mathsf{C}}\boldsymbol{\mathsf{x}} = \boldsymbol{0} \\
(\boldsymbol{\mathsf U}^{-1}-\boldsymbol{\mathsf{L}}\boldsymbol{\mathsf W}\boldsymbol{\mathsf{L}}^*)(\boldsymbol{\mathsf{v}}-\boldsymbol{\mathsf{u}}+\boldsymbol{\mathsf{e}}'_2)+\boldsymbol{\mathsf{L}}\boldsymbol{\mathsf{y}}
- \boldsymbol{\mathsf{D}}^{-1}\boldsymbol{\mathsf{v}} \in \boldsymbol{\mathsf{B}}^{-1} \boldsymbol{\mathsf{u}}
\end{cases}\nonumber\\
\Leftrightarrow \quad & 
\begin{cases}
\boldsymbol{\mathsf{y}} = 
\boldsymbol{\mathsf{x}}-\boldsymbol{\mathsf{W}}(\boldsymbol{\mathsf{C}}\boldsymbol{\mathsf{x}}+\boldsymbol{\mathsf{L}}^*\boldsymbol{\mathsf{u}})+\boldsymbol{\mathsf{e}}'_1\\
\boldsymbol{\mathsf{v}}+\boldsymbol{\mathsf{e}}'_2+\boldsymbol{\mathsf U}
\left(
\boldsymbol{\mathsf{L}}\big(\boldsymbol{\mathsf{x}}
-\boldsymbol{\mathsf{W}}(\boldsymbol{\mathsf{C}}\boldsymbol{\mathsf{x}}+\boldsymbol{\mathsf{L}}^*\boldsymbol{\mathsf{v}}+\boldsymbol{\mathsf{L}}^*\boldsymbol{\mathsf{e}}'_2)+\boldsymbol{\mathsf{e}}'_1\big)
- \boldsymbol{\mathsf{D}}^{-1}\boldsymbol{\mathsf{v}}
\right)
\in (\ID+\boldsymbol{\mathsf{U}}\boldsymbol{\mathsf{B}}^{-1})\boldsymbol{\mathsf{u}}
\end{cases}
\nonumber\\
\Leftrightarrow \quad & 
\begin{cases}
\boldsymbol{\mathsf{u}} = \boldsymbol{\mathsf{J}}_{\boldsymbol{\mathsf{U}}\boldsymbol{\mathsf{B}}^{-1}}\left(
\boldsymbol{\mathsf{v}}+\boldsymbol{\mathsf{e}}'_2+\boldsymbol{\mathsf U}
\left(
\boldsymbol{\mathsf{L}}\big(\boldsymbol{\mathsf{x}}
-\boldsymbol{\mathsf{W}}(\boldsymbol{\mathsf{C}}\boldsymbol{\mathsf{x}}+\boldsymbol{\mathsf{L}}^*\boldsymbol{\mathsf{v}}+\boldsymbol{\mathsf{L}}^*
\boldsymbol{\mathsf{e}}'_2)+\boldsymbol{\mathsf{e}}'_1\big)- \boldsymbol{\mathsf{D}}^{-1}\boldsymbol{\mathsf{v}}
\right)
\right)\\
\boldsymbol{\mathsf{y}} = 
\boldsymbol{\mathsf{x}}-\boldsymbol{\mathsf{W}}(\boldsymbol{\mathsf{C}}\boldsymbol{\mathsf{x}}+\boldsymbol{\mathsf{L}}^*\boldsymbol{\mathsf{u}})+\boldsymbol{\mathsf{e}}'_1,
\end{cases}
\end{align} 
which lead to \eqref{e:expresprod2} provided that
\begin{equation}
\begin{cases}
-\boldsymbol{\mathsf{W}}\boldsymbol{\mathsf{e}}_1 = \boldsymbol{\mathsf{e}}'_1\\
\boldsymbol{\mathsf{U}}\boldsymbol{\mathsf{e}}_2 = \boldsymbol{\mathsf U}
\boldsymbol{\mathsf{L}}(\boldsymbol{\mathsf{e}}'_1-\boldsymbol{\mathsf{W}}\boldsymbol{\mathsf{L}}^*
\boldsymbol{\mathsf{e}}'_2)+\boldsymbol{\mathsf{e}}'_2.
\end{cases}
\end{equation}
Since $\boldsymbol{\mathsf U}^{-1}-\boldsymbol{\mathsf{L}}\boldsymbol{\mathsf W}\boldsymbol{\mathsf{L}}^*$ is an isomophism,
the latter equalities are equivalent to \eqref{e:exps2}.
\end{proof}

From the above two lemmas, a second type of block-coordinate primal-dual algorithm can be deduced to solve Problem~\ref{prob:main}
in the case when $\boldsymbol{\mathsf{A}} = \boldsymbol{\mathsf{0}}$.
\begin{proposition}\label{p:algopdmon2}
Let $\boldsymbol{\mathsf{W}}$, $\boldsymbol{\mathsf{U}}$, $\mu$, and $\nu$
be defined as in Proposition \ref{p:algopdmon1}.
Suppose that 
\begin{equation}\label{e:condvartheta2}
\min\big\{\mu,\nu(1-\|\boldsymbol{\mathsf U}^{1/2}\boldsymbol{\mathsf{L}}\boldsymbol{\mathsf W}^{1/2}\|^2)\big\} > \frac12.
\end{equation}
Let $(\lambda_n)_{n\in\NN}$  be a sequence in $\left]0,1\right]$ such that 
$\inf_{n\in\NN}\lambda_n>0$,
let $\boldsymbol{x}_0$ and 
$(\boldsymbol{c}_n)_{n\in\NN}$ be $\HHH$-valued random variables,
let $\boldsymbol{v}_0$, 
$(\boldsymbol{b}_n)_{n\in\NN}$, and 
$(\boldsymbol{d}_n)_{n\in\NN}$ be $\GGG$-valued random 
variables, and let $(\boldsymbol{\varepsilon}_n)_{n\in\NN}$ be identically distributed 
$\mathbb{D}_{p+q}$-valued random variables. 
Iterate
\begin{equation}\label{e:PDcoord2}
\begin{array}{l}
\text{for}\;n=0,1,\ldots\\
\left\lfloor
\begin{array}{l}
\text{for}\;j=1,\ldots,p\\
\left\lfloor
\begin{array}{l}
\eta_{j,n} = \max\menge{\varepsilon_{p+k,n}}{k\in \mathbb{L}_j^*}\\
\displaystyle w_{j,n} =
\eta_{j,n}\big(x_{j,n}-\mathsf{W}_j(\mathsf{C}_j x_{j,n} +c_{j,n})\big)
\end{array}
\right.\\
\text{for}\;k=1,\ldots,q\\
\left\lfloor
\begin{array}{l}
\displaystyle u_{k,n} = \varepsilon_{p+k,n}\Big(\mathsf{J}_{\mathsf{U}_k\mathsf{B}_k^{-1}}\big(v_{k,n}+\mathsf{U}_k(\sum_{j \in \mathbb{L}_k}\mathsf{L}_{k,j} (w_{j,n}-
\mathsf{W}_j\sum_{k'\in \mathbb{L}_j^*} \mathsf{L}_{k',j}^* v_{k',n})- \mathsf{D}_k^{-1}v_{k,n}+d_{k,n})\big)+b_{k,n}\Big)\\
v_{k,n+1} = v_{k,n}+\lambda_n \varepsilon_{p+k,n} (u_{k,n}-v_{k,n})
\end{array}
\right.\\
\text{for}\;j=1,\ldots,p\\
\left\lfloor
\begin{array}{l}
\displaystyle  x_{j,n+1} = x_{j,n}+\lambda_n \varepsilon_{j,n}  \Big(w_{j,n}-\mathsf{W}_j \sum_{k\in \mathbb{L}_j^*} \mathsf{L}^*_{k,j} u_{k,n}- x_{j,n}\Big),
\end{array}
\right.\vspace*{0.1cm}
\end{array}
\right.
\end{array}
\end{equation}
and set $(\forall n\in\NN)$ $\EEE_n=\sigma(\boldsymbol{\varepsilon}_n)$ and $\boldsymbol{\XX}_n=
\sigma(\boldsymbol{x}_{n'},\boldsymbol{v}_{n'})_{0\leq n'\leq n}$.
In addition, assume that
\begin{enumerate}
\item
\label{c:PDcoord2i}
$\sum_{n\in\NN}\sqrt{\EC{\|\boldsymbol{b}_n\|^2}
{\boldsymbol{\XX}_n}}<\pinf$,
$\sum_{n\in\NN}\sqrt{\EC{\|\boldsymbol{c}_n\|^2}
{\boldsymbol{\XX}_n}}<\pinf$, and
$\sum_{n\in\NN}\sqrt{\EC{\|\boldsymbol{d}_n\|^2}
{\boldsymbol{\XX}_n}}<\pinf$ $\as$
\end{enumerate}
and  Conditions \ref{c:PDcoord1iisym}-\ref{c:PDcoord1iiisym} in Proposition \ref{p:algopdmon1sym} hold.\\
If, in Problem~\ref{prob:main}, $(\forall j \in \{1,\ldots,p\})$ $\mathsf{A}_j = 0$,
then $(\boldsymbol{x}_n)_{n\in\NN}$ converges weakly $\as$ to an 
$\boldsymbol{\mathsf{F}}$-valued random variable, and
$(\boldsymbol{v}_n)_{n\in\NN}$ 
converges weakly $\as$ to an
$\boldsymbol{\mathsf{F}}^*$-valued random variable.
\end{proposition}
\begin{proof} First note that, in view of Condition \ref{c:PDcoord1iiisym} in Proposition \ref{p:algopdmon1sym} (since $(\forall j \in \{1,\ldots,p\})$ $\mathbb{L}_j^*\neq \emp$), Iterations \eqref{e:PDcoord2} are equivalent to
\begin{equation}\label{e:PDcoord2prim}
\begin{array}{l}
\text{for}\;n=0,1,\ldots\\
\left\lfloor
\begin{array}{l}
\text{for}\;k=1,\ldots,q\\
\left\lfloor
\begin{array}{l}
\zeta_{k,n} = \max\big\{\varepsilon_{p+k,n},(\varepsilon_{j,n})_{j\in \mathbb{L}_k}\big\}
\end{array}
\right.\\
\text{for}\;j=1,\ldots,p\\
\left\lfloor
\begin{array}{l}
\eta_{j,n} = \max\big\{\varepsilon_{j,n}, (\zeta_{k,n})_{k \in \mathbb{L}_j^*}\big\}\\
\displaystyle w_{j,n} =
\eta_{j,n}\big(x_{j,n}-\mathsf{W}_j(\mathsf{C}_j x_{j,n} +c_{j,n})\big)
\end{array}
\right.\\
\text{for}\;k=1,\ldots,q\\
\left\lfloor
\begin{array}{l}
\displaystyle u_{k,n} = \zeta_{k,n}\Big(\mathsf{J}_{\mathsf{U}_k\mathsf{B}_k^{-1}}\big(v_{k,n}+\mathsf{U}_k(\sum_{j \in \mathbb{L}_k}\mathsf{L}_{k,j} (w_{j,n}-
\mathsf{W}_j\sum_{k'\in \mathbb{L}_j^*} \mathsf{L}_{k',j}^* v_{k',n})- \mathsf{D}_k^{-1}v_{k,n}+d_{k,n})\big)+b_{k,n}\Big)\\
v_{k,n+1} = v_{k,n}+\lambda_n \varepsilon_{p+k,n} (u_{k,n}-v_{k,n})
\end{array}
\right.\\
\text{for}\;j=1,\ldots,p\\
\left\lfloor
\begin{array}{l}
\displaystyle  x_{j,n+1} = x_{j,n}+\lambda_n \varepsilon_{j,n}  \Big(w_{j,n}-\mathsf{W}_j \sum_{k\in \mathbb{L}_j^*} \mathsf{L}^*_{k,j} u_{k,n}- x_{j,n}\Big).
\end{array}
\right.\vspace*{0.1cm}
\end{array}
\right.
\end{array}
\end{equation}
Furthermore, Condition \eqref{e:condvartheta2} implies that $\|\boldsymbol{\mathsf U}^{1/2}\boldsymbol{\mathsf{L}}\boldsymbol{\mathsf W}^{1/2}\|<1$.
Hence, Lemma \ref{lem:resespprod2} allows us to show the equivalence between
Algorithms \eqref{e:PDcoord2prim} and \eqref{e:FBPrecond} when $\boldsymbol{\mathsf{V}}$
is given by \eqref{e:defV2}, $\boldsymbol{\mathsf{Q}}$ is given by \eqref{e:maximal1} (with $\boldsymbol{\mathsf{A}}=\boldsymbol{\mathsf{0}}$), 
and $\boldsymbol{\mathsf{R}}$ is given by \eqref{e:maximal21},
provided that, for every $n\in \NN$, \eqref{e:equivFB1}-\eqref{e:equivFB2}
hold and
\begin{align}
& \boldsymbol{t}_n = (\boldsymbol{0},\boldsymbol{b}_n),\\
&\boldsymbol{s}_n = \big(-\boldsymbol{\mathsf{W}}\boldsymbol{e}_{1,n},
(\boldsymbol{\mathsf{U}}^{-1}-\boldsymbol{\mathsf{L}}\boldsymbol{\mathsf{W}}\boldsymbol{\mathsf{L}}^*)^{-1}
(\boldsymbol{\mathsf{L}}\boldsymbol{\mathsf{W}}\boldsymbol{\mathsf{L}}^*\boldsymbol{b}_n+\boldsymbol{d}_n)\big),\\
& \boldsymbol{e}_{1,n} = \boldsymbol{\mathsf{L}}^*\boldsymbol{b}_n+\boldsymbol{c}_n.
\end{align}
In the proof of Proposition \ref{p:algopdmon1}, we have seen that $\boldsymbol{\mathsf W}^{1/2} \boldsymbol{\mathsf{C}} \boldsymbol{\mathsf W}^{1/2}$ is $\mu$-cocoercive and
$\boldsymbol{\mathsf U}^{1/2} \boldsymbol{\mathsf{D}}^{-1} \boldsymbol{\mathsf U}^{1/2}$ is $\nu$-cocoercive.
According to Lemma~\ref{le:2}\ref{le:2ii}, $\boldsymbol{\mathsf V}^{1/2} \boldsymbol{\mathsf R} \boldsymbol{\mathsf V}^{1/2}$
is thus $\vartheta$-cocoercive where $\vartheta$ is given by \eqref{e:vartheta2}, and \eqref{e:condvartheta2} means that $1 = \sup_{n\in\NN}\gamma_n<2\vartheta$.
In addition,
\begin{align}
& \sum_{n\in\NN}\sqrt{\EC{\|\boldsymbol{t}_n\|^2}{\boldsymbol{\XX}_n}} =\sum_{n\in\NN}\sqrt{\EC{\|\boldsymbol{b}_n\|^2}{\boldsymbol{\XX}_n}} < \pinf,\\
& \sum_{n\in\NN}\sqrt{\EC{\|\boldsymbol{s}_n\|^2}{\boldsymbol{\XX}_n}}\nonumber\\
&\;\;\le \|\boldsymbol{\mathsf{W}}\boldsymbol{\mathsf{L}}^*\| \sum_{n\in\NN}\sqrt{\EC{\|\boldsymbol{b}_n\|^2}{\boldsymbol{\XX}_n}}
+\|\boldsymbol{\mathsf{W}}\| \sum_{n\in\NN}\sqrt{\EC{\|\boldsymbol{c}_n\|^2}{\boldsymbol{\XX}_n}}\nonumber\\
&\;\;+ \|(\boldsymbol{\mathsf{U}}^{-1}-\boldsymbol{\mathsf{L}}\boldsymbol{\mathsf{W}}\boldsymbol{\mathsf{L}}^*)^{-1}\|
\Big(\|\boldsymbol{\mathsf{L}}\boldsymbol{\mathsf{W}}\boldsymbol{\mathsf{L}}^*\| \sum_{n\in\NN}\sqrt{\EC{\|\boldsymbol{b}_n\|^2}{\boldsymbol{\XX}_n}}+
\sum_{n\in\NN}\sqrt{\EC{\|\boldsymbol{d}_n\|^2}{\boldsymbol{\XX}_n}}\Big) < \pinf.
\end{align}
Since we have assumed that, for every $k \in \{1,\ldots,q\}$, $\mathbb{L}_k \neq \emp$,
Conditions \ref{c:PDcoord1iisym}-\ref{c:PDcoord1iiisym} in Proposition~\ref{p:algopdmon1sym}
guarantee that Condition \ref{p:nyc2014-04-03iv} in Proposition \ref{p:FBPreconf}
is satisfied. The convergence result then follows from this proposition.
\end{proof}
\begin{remark}\  
For every $j\in \{1,\ldots,p\}$, let a cocoercivity constant of $\mathsf{C}_j$ be denoted by $\widetilde{\mu}_j\in \RPP$ and, for every $k\in \{1,\ldots,q\}$,
let a strong monotonicity constant of $\mathsf{D}_k$ be denoted by $\widetilde{\nu}_k\in \RPP$. 
Using \eqref{e:boundnormULW}-\eqref{e:mumutj}, a necessary condition for \eqref{e:condvartheta2} to be satisfied is
\begin{align}
&\min\Big\{(\|\mathsf{W}_j\|^{-1}\widetilde{\mu}_j)_{1\le j \le p},\Big(1-\sum_{j=1}^p \sum_{k=1}^q \|\mathsf{U}_k^{1/2} \mathsf{L}_{k,j} \mathsf{W}_j^{1/2} \|^2\Big)(\|\mathsf{U}_k\|^{-1}\widetilde{\nu}_k)_{1\le k \le q}\Big\}> \frac12. \label{e:condvarthetaalpha4}
\end{align}
In the case when, for every $k\in \{1,\ldots,q\}$, $\mathsf{D}_k^{-1} = \mathsf{0}$, the constants $(\widetilde{\nu}_k)_{1\le k \le q}$ can be chosen arbitrarily large and
the above condition reduces to
\begin{equation}
\left\{
\begin{array}{l}
\displaystyle\sum_{j=1}^p \sum_{k=1}^q \|\mathsf{U}_k^{1/2} \mathsf{L}_{k,j} \mathsf{W}_j^{1/2} \|^2 < 1\\
\displaystyle\min\big\{(\|\mathsf{W}_j\|^{-1}\widetilde{\mu}_j)_{1\le j \le p}\big\} > \frac12.
\end{array}
\right.
\end{equation}
This condition appears to be less restrictive than \eqref{e:condvarthetaalphater}.
\end{remark}

\section{Block-coordinate primal-dual proximal algorithms for convex optimization problems}\label{se:coordopt}
As we will show next, the results obtained in the previous section allow us to deduce a couple of novel primal-dual proximal splitting
algorithms for solving a variety of (possibly nonsmooth) convex optimization problems.
More precisely, we will turn our attention
to the following class of optimization problems, the notation of the previous section being still in force:
\begin{problem}
\label{prob:mainfunc}
For every 
$j\in\{1,\ldots,p\}$, let 
$\mathsf{f}_j\in \Gamma_0(\HH_j)$,
let $\mathsf{h}_j\in \Gamma_0(\HH_j)$ be Lipschitz-differentiable,
and, for every $k\in\{1,\ldots,q\}$, let 
$\mathsf{g}_k\in \Gamma_0(\GG_k)$, 
let $\mathsf{l}_k\in \Gamma_0(\GG_k)$ be strongly convex,
and let $\mathsf{L}_{k,j}\in \BL(\HH_j,\GG_k)$. 
Suppose that \eqref{e:defLk} and \eqref{e:defLjs} hold,
and that there exists $(\overline{\mathsf{x}}_1,\ldots,\overline{\mathsf{x}}_p)\in \HH_1\oplus\cdots\oplus \HH_p$
such that 
\begin{equation}\label{e:qualif}
(\forall j\in\{1,\ldots,p\})\quad 0\in
\partial\mathsf{f}_j(\overline{\mathsf{x}}_j)+\nabla \mathsf{h}_j(\overline{\mathsf{x}}_j)
+\sum_{k=1}^q\mathsf{L}_{k,j}^*(\partial\mathsf{g}_k\infconv\partial\mathsf{l}_k)\bigg(\sum_{j'=1}^p\mathsf{L}_{k,j'}\overline{\mathsf{x}}_{j'}\bigg).
\end{equation}
Let $\widetilde{\boldsymbol{\mathsf{F}}}$ be the set of solutions to the
problem
\begin{equation}\label{e:primopt}
\minimize{\mathsf{x}_1\in\HH_1,\ldots,\mathsf{x}_p\in\HH_p}
{\sum_{j=1}^p\big(\mathsf{f}_j(\mathsf{x}_j)+\mathsf{h}_j(\mathsf{x}_j)\big)+
\sum_{k=1}^q (\mathsf{g}_k\infconv\mathsf{l}_k)
\bigg(\sum_{j=1}^p\mathsf{L}_{k,j}\mathsf{x}_{j}\bigg)}
\end{equation}
and let $\widetilde{\boldsymbol{\mathsf{F}}}^*$ be the set of 
solutions to the dual problem
\begin{equation}
\minimize{\mathsf{v}_1\in\GG_1,\ldots,\mathsf{v}_q\in\GG_q}
{\sum_{j=1}^p (\mathsf{f}_j^*\infconv \mathsf{h}_j^*) \bigg(-\Sum_{k=1}^q
\mathsf{L}_{k,j}^*\mathsf{v}_{k}\bigg)+\sum_{k=1}^q \big(\mathsf{g}_k^*(\mathsf{v}_k)+\mathsf{l}_k^*(\mathsf{v}_k)\big)}.
\end{equation}
Our objective is to find a pair $(\widehat{\boldsymbol{x}},\widehat{\boldsymbol{v}})$ of random variables
such that $\widehat{\boldsymbol{x}}$ is $\widetilde{\boldsymbol{\mathsf{F}}}$-valued and 
$\widehat{\boldsymbol{v}}$ is $\widetilde{\boldsymbol{\mathsf{F}}}^*$-valued.
\end{problem}
Note that the inclusion condition in Problem \ref{prob:mainfunc} is satisfied under a number of relatively weak assumptions:
\begin{proposition} {\rm \cite[Proposition~5.3]{Combettes_P_2013_siam-opt_sys_smi}} 
Consider the setting of Problem \ref{prob:mainfunc}. Suppose that \eqref{e:primopt} has a solution.
Then, the existence of $(\overline{\mathsf{x}}_1,\ldots,\overline{\mathsf{x}}_p)\in \HH_1\oplus\cdots\oplus \HH_p$ satisfying \eqref{e:qualif} is guaranteed in each of the following cases:
\begin{enumerate}
\item For every $j\in \{1,\ldots,p\}$, $\mathsf{f}_j$ is real-valued and, for every $k \in \{1,\ldots,q\}$,
$(\mathsf{x}_j)_{1\le j \le p} \mapsto \sum_{j=1}^p\mathsf{L}_{k,j}\mathsf{x}_{j}$ is surjective.
\item For every $k\in \{1,\ldots,q\}$, $\mathsf{g}_k$ or $\mathsf{l}_k$ is real-valued.
\item $(\HH_j)_{1\le j \le p}$ and $(\GG_k)_{1\le k \le q}$ are finite-dimensional, and
$(\forall j \in \{1,\ldots,p\})$ $(\exists \mathsf{x}_j \in \reli \dom \mathsf{f}_j)$ such that
$(\forall k \in \{1,\ldots,q\})$ $\sum_{j=1}^p\mathsf{L}_{k,j}\mathsf{x}_{j} \in \reli \dom \mathsf{g}_k
+ \reli \dom \mathsf{l}_k$.
\end{enumerate}
\end{proposition}

The following result can be deduced from Proposition \ref{p:algopdmon1}:
\begin{proposition}\label{p:algopdopt1}
Let $\boldsymbol{\mathsf{W}}$ and $\boldsymbol{\mathsf{U}}$ be defined as in Proposition \ref{p:algopdmon1}.
For every $j\in \{1,\ldots,p\}$, let $\mu_j^{-1}\in \RPP$ be a Lipschitz constant of the gradient of
$\mathsf{h}_j \circ \mathsf{W}_j^{1/2}$ and, for every $k\in \{1,\ldots,q\}$,
let $\nu_k^{-1}\in \RPP$ be a Lipschitz constant of the gradient of $\mathsf{l}_k^* \circ \mathsf{U}_k^{1/2}$.
Suppose that \eqref{e:condvarthetaalpha} holds
where $\vartheta_\alpha$ is defined by \eqref{e:varthetalpha},
$\mu = \min\{\mu_1,\ldots,\mu_p\}$, and $\nu=\min\{\nu_1,\ldots,\nu_q\}$. Let $(\lambda_n)_{n\in\NN}$ 
be a sequence in $\left]0,1\right]$ such that 
$\inf_{n\in\NN}\lambda_n>0$,
let $\boldsymbol{x}_0$, 
$(\boldsymbol{a}_n)_{n\in\NN}$, and 
$(\boldsymbol{c}_n)_{n\in\NN}$ be $\HHH$-valued random variables,
let $\boldsymbol{v}_0$, 
$(\boldsymbol{b}_n)_{n\in\NN}$, and 
$(\boldsymbol{d}_n)_{n\in\NN}$ be $\GGG$-valued random 
variables, and let $(\boldsymbol{\varepsilon}_n)_{n\in\NN}$ be identically distributed 
$\mathbb{D}_{p+q}$-valued random variables. 
Iterate
\begin{equation}\label{e:PDcoordopt1}
\begin{array}{l}
\text{for}\;n=0,1,\ldots\\
\left\lfloor
\begin{array}{l}
\text{for}\;j=1,\ldots,p\\
\left\lfloor
\begin{array}{l}
\displaystyle y_{j,n} =
\varepsilon_{j,n}\Big(\prox_{\mathsf{f}_j}^{\mathsf{W}_j^{-1}}\big(x_{j,n}-\mathsf{W}_j(\sum_{k\in \mathbb{L}_j^*} {\mathsf{L}^*_{k,j} v_{k,n}}
+ \nabla \mathsf{h}_j(x_{j,n}) +c_{j,n})\big)+a_{j,n}\Big)\\
x_{j,n+1} = x_{j,n}+\lambda_n \varepsilon_{j,n}  (y_{j,n}- x_{j,n})
\end{array}
\right.\\
\text{for}\;k=1,\ldots,q\\
\left\lfloor
\begin{array}{l}
\displaystyle u_{k,n} = \varepsilon_{p+k,n}\Big(\prox_{\mathsf{g}_k^*}^{\mathsf{U}_k^{-1}}\big(v_{k,n}+\mathsf{U}_k(\sum_{j \in \mathbb{L}_k}\mathsf{L}_{k,j} (2y_{j,n}-x_{j,n})- 
\nabla\mathsf{l}_k^*(v_{k,n})+d_{k,n})\big)+b_{k,n}\Big)\\
v_{k,n+1} = v_{k,n}+\lambda_n \varepsilon_{p+k,n} (u_{k,n}-v_{k,n}).
\end{array}
\right.
\end{array}
\right.\\
\end{array}
\end{equation}
In addition, assume that Conditions~\ref{c:PDcoord1i}-\ref{c:PDcoord1iii} in Proposition \ref{p:algopdmon1} hold, where
$(\forall n\in\NN)$ $\EEE_n=\sigma(\boldsymbol{\varepsilon}_n)$ and $\boldsymbol{\XX}_n=
\sigma(\boldsymbol{x}_{n'},\boldsymbol{v}_{n'})_{0\leq n'\leq n}$.\\
Then, $(\boldsymbol{x}_n)_{n\in\NN}$ converges weakly $\as$ to a 
$\widetilde{\boldsymbol{\mathsf{F}}}$-valued random variable, and
$(\boldsymbol{v}_n)_{n\in\NN}$ 
converges weakly $\as$ to a
$\widetilde{\boldsymbol{\mathsf{F}}}^*$-valued random variable.
\end{proposition}
\begin{proof} Let us set, for every $j\in \{1,\ldots,p\}$, $\mathsf{A}_j =\partial {\mathsf f}_j$, $\mathsf{C}_j = \nabla \mathsf{h}_j$
and, for every $k\in \{1,\ldots,q\}$, $\mathsf{B}_k = \partial \mathsf{g}_k$, and $\mathsf{D}_k^{-1} = \nabla \mathsf{l}_k^*$. Then, it can be noticed that, 
for every $j\in \{1,\ldots,p\}$ and $k\in \{1,\ldots,q\}$, $\mathsf{J}_{\mathsf{W}_j\mathsf{A}_j} = \prox_{\mathsf{f}_j}^{\mathsf{W}_j^{-1}}$,
$\mathsf{J}_{\mathsf{U}_k\mathsf{B}_k^{-1}} = \prox_{\mathsf{g}_k^*}^{\mathsf{U}_k^{-1}}$, and that the Lipschitz-differentiability assumptions
made on $\mathsf{h}_j$ and $\mathsf{l}_k^*$ are equivalent to the fact that $\mathsf{W}_j^{1/2} \mathsf{C}_j \mathsf{W}_j^{1/2}$ is $\mu_j$-cocoercive
and $\mathsf{U}_k^{1/2} \mathsf{D}_k^{-1} \mathsf{U}_k^{1/2}$ is $\nu_k$-cocoercive \cite[Corollaries~16.42 \& 18.16]{Bauschke_H_2011_book_con_amo}. Proposition \ref{p:algopdmon1} thus allows us to assert that
$(\boldsymbol{x}_n)_{n\in\NN}$ converges weakly $\as$ to an 
$\boldsymbol{\mathsf{F}}$-valued random variable, and
$(\boldsymbol{v}_n)_{n\in\NN}$ converges weakly $\as$ to an
$\boldsymbol{\mathsf{F}}^*$-valued random variable, where $\boldsymbol{\mathsf{F}}$ and $\boldsymbol{\mathsf{F}}^*$ have been defined in Problem \ref{prob:main}.
Let us now show that the first limit is a $\widetilde{\boldsymbol{\mathsf{F}}}$-valued random variable,
and the second one is a $\widetilde{\boldsymbol{\mathsf{F}}}^*$-valued random variable.
Define the separable functions $\boldsymbol{\mathsf{f}}\in \Gamma_0(\HHH)$, $\boldsymbol{\mathsf{h}}\in \Gamma_0(\HHH)$, 
$\boldsymbol{\mathsf{g}}\in \Gamma_0(\GGG)$, and $\boldsymbol{\mathsf{l}}\in \Gamma_0(\GGG)$ as
\begin{align}
& \boldsymbol{\mathsf{f}}\colon \boldsymbol{\mathsf{x}}\mapsto \sum_{j=1}^p \mathsf{f}_j(\mathsf{x}_j), \qquad
\boldsymbol{\mathsf{h}}\colon \boldsymbol{\mathsf{x}}\mapsto \sum_{j=1}^p \mathsf{h}_j(\mathsf{x}_j),\\
& \boldsymbol{\mathsf{g}}\colon \boldsymbol{\mathsf{v}}\mapsto \sum_{k=1}^q \mathsf{g}_k(\mathsf{v}_k), \qquad 
\boldsymbol{\mathsf{l}}\colon \boldsymbol{\mathsf{v}}\mapsto \sum_{k=1}^q \mathsf{l}_k(\mathsf{v}_k).
\end{align}
According to \cite[Proposition 16.8]{Bauschke_H_2011_book_con_amo},
\eqref{e:qualif} can be reexpressed more concisely as
\begin{equation}\label{e:qualifbis} 
\boldsymbol{\mathsf{0}}\in \partial \boldsymbol{\mathsf{f}}(\overline{\boldsymbol{\mathsf{x}}})+\nabla \boldsymbol{\mathsf{h}}(\overline{\boldsymbol{\mathsf{x}}})
+ \boldsymbol{\mathsf{L}}^* (\partial \boldsymbol{\mathsf{g}}\infconv \partial \boldsymbol{\mathsf{l}})(\boldsymbol{\mathsf{L}}\overline{\boldsymbol{\mathsf{x}}}).
\end{equation}
Since $\dom \boldsymbol{\mathsf{h}} = \HHH$, $\partial\boldsymbol{\mathsf{f}}+\nabla \boldsymbol{\mathsf{h}}
= \partial(\boldsymbol{\mathsf{f}}+\boldsymbol{\mathsf{h}})$ \cite[Propositions~16.38 \& 17.26]{Bauschke_H_2011_book_con_amo}
and since $\dom \boldsymbol{\mathsf{l}}^* = \GGG$, $\partial \boldsymbol{\mathsf{g}}\infconv \partial \boldsymbol{\mathsf{l}}
= \partial(\boldsymbol{\mathsf{g}}\infconv \boldsymbol{\mathsf{l}})$ \cite[Proposition~24.27]{Bauschke_H_2011_book_con_amo}.
Equation \eqref{e:qualifbis}  implies that $\boldsymbol{\mathsf{L}}\big(\dom(\boldsymbol{\mathsf{f}}+\boldsymbol{\mathsf{h}})\big) \cap \dom(\boldsymbol{\mathsf{g}}\infconv \boldsymbol{\mathsf{l}})\neq~\emp$
\cite[Proposition 16.3(i)]{Bauschke_H_2011_book_con_amo} and it follows from \cite[Proposition 16.5]{Bauschke_H_2011_book_con_amo} that
\begin{equation}
(\forall \boldsymbol{\mathsf{x}}\in \HHH)\qquad
\partial \boldsymbol{\mathsf{f}}(\boldsymbol{\mathsf{x}})+\nabla \boldsymbol{\mathsf{h}}(\boldsymbol{\mathsf{x}})
+ \boldsymbol{\mathsf{L}}^* (\partial \boldsymbol{\mathsf{g}}\infconv \partial \boldsymbol{\mathsf{l}})(\boldsymbol{\mathsf{L}}\boldsymbol{\mathsf{x}})
\subset \partial\big(\boldsymbol{\mathsf{f}}+\boldsymbol{\mathsf{h}}+(\boldsymbol{\mathsf{g}}\infconv \boldsymbol{\mathsf{l}})\circ \boldsymbol{\mathsf{L}}\big)(\boldsymbol{\mathsf{x}}).
\end{equation}
As a consequence of \eqref{e:primmon} and Fermat's rule \cite[Theorem 16.2]{Bauschke_H_2011_book_con_amo}, this allows us to conclude that 
\begin{equation}
\boldsymbol{\mathsf F} = \zer\big(\partial \boldsymbol{\mathsf{f}}+\nabla \boldsymbol{\mathsf{h}}
+ \boldsymbol{\mathsf{L}}^* (\partial \boldsymbol{\mathsf{g}}\infconv \partial \boldsymbol{\mathsf{l}})\boldsymbol{\mathsf{L}}\big)
\subset \zer(\partial\big(\boldsymbol{\mathsf{f}}+\boldsymbol{\mathsf{h}}+(\boldsymbol{\mathsf{g}}\infconv \boldsymbol{\mathsf{l}})\circ \boldsymbol{\mathsf{L}}\big)) =
\widetilde{\boldsymbol{\mathsf F}}.
\end{equation}
By a similar argument, the fact that $\boldsymbol{\mathsf F}^*=\zer\big(-\boldsymbol{\mathsf{L}}(\partial \boldsymbol{\mathsf{f}}^*\infconv \partial\boldsymbol{\mathsf{h}}^*)(-\boldsymbol{\mathsf{L}}^*)
+ \partial \boldsymbol{\mathsf{g}}^*+ \nabla \boldsymbol{\mathsf{l}}^*\big)\neq \emp$ allows us to show that $\boldsymbol{\mathsf F}^* \subset \widetilde{\boldsymbol{\mathsf F}}^*$.
\end{proof}

In a quite similar way, Proposition \ref{p:algopdmon2} leads to the following result.
\begin{proposition}
Let $\boldsymbol{\mathsf{W}}$ and $\boldsymbol{\mathsf{U}}$ be defined as in Proposition \ref{p:algopdmon1}.
Let $\mu$ and $\nu$ be defined as in Proposition \ref{p:algopdopt1}.
Suppose that Condition \eqref{e:condvartheta2} holds. 
Let $(\lambda_n)_{n\in\NN}$ be a sequence in $\left]0,1\right]$ such that 
$\inf_{n\in\NN}\lambda_n>0$,
let $\boldsymbol{x}_0$ and 
$(\boldsymbol{c}_n)_{n\in\NN}$ be $\HHH$-valued random variables,
let $\boldsymbol{v}_0$, 
$(\boldsymbol{b}_n)_{n\in\NN}$, and 
$(\boldsymbol{d}_n)_{n\in\NN}$ be $\GGG$-valued random 
variables, and let $(\boldsymbol{\varepsilon}_n)_{n\in\NN}$ be identically distributed 
$\mathbb{D}_{p+q}$-valued random variables. 
Iterate
\begin{equation}\label{e:PDcoordopt2}
\begin{array}{l}
\text{for}\;n=0,1,\ldots\\
\left\lfloor
\begin{array}{l}
\text{for}\;j=1,\ldots,p\\
\left\lfloor
\begin{array}{l}
\eta_{j,n} = \max\menge{\varepsilon_{p+k,n}}{k\in \mathbb{L}_j^*}\\
\displaystyle w_{j,n} =
\eta_{j,n}\big(x_{j,n}-\mathsf{W}_j(\nabla\mathsf{h}_j(x_{j,n}) +c_{j,n})\big)
\end{array}
\right.\\
\text{for}\;k=1,\ldots,q\\
\left\lfloor
\begin{array}{l}
\displaystyle u_{k,n} = \varepsilon_{p+k,n}\Big(\prox_{\mathsf{g}_k^*}^{\mathsf{U}_k^{-1}}\big(v_{k,n}+\mathsf{U}_k(\sum_{j \in \mathbb{L}_k}\mathsf{L}_{k,j} (w_{j,n}-
\mathsf{W}_j\sum_{k'\in \mathbb{L}_j^*} \mathsf{L}_{k',j}^* v_{k',n})- \nabla\mathsf{l}_k^*(v_{k,n})+d_{k,n})\big)+b_{k,n}\Big)\\
v_{k,n+1} = v_{k,n}+\lambda_n \varepsilon_{p+k,n} (u_{k,n}-v_{k,n})
\end{array}
\right.\\
\text{for}\;j=1,\ldots,p\\
\left\lfloor
\begin{array}{l}
\displaystyle  x_{j,n+1} = x_{j,n}+\lambda_n \varepsilon_{j,n}  \Big(w_{j,n}-\mathsf{W}_j \sum_{k\in \mathbb{L}_j^*} \mathsf{L}^*_{k,j} u_{k,n}- x_{j,n}\Big).
\end{array}
\right.\vspace*{0.1cm}
\end{array}
\right.
\end{array}
\end{equation}
In addition, assume that Conditions \ref{c:PDcoord2i} in Proposition \ref{p:algopdmon2}, and
\ref{c:PDcoord1iisym}-\ref{c:PDcoord1iiisym} in Proposition \ref{p:algopdmon1sym} hold, where
$(\forall n\in\NN)$ $\EEE_n=\sigma(\boldsymbol{\varepsilon}_n)$ and $\boldsymbol{\XX}_n=
\sigma(\boldsymbol{x}_{n'},\boldsymbol{v}_{n'})_{0\leq n'\leq n}$.\\
If, in Problem~\ref{prob:mainfunc}, $(\forall j \in \{1,\ldots,p\})$ $\mathsf{f}_j = 0$,
then $(\boldsymbol{x}_n)_{n\in\NN}$ converges weakly $\as$ to a
$\widetilde{\boldsymbol{\mathsf{F}}}$-valued random variable, and
$(\boldsymbol{v}_n)_{n\in\NN}$ 
converges weakly $\as$ to a
$\widetilde{\boldsymbol{\mathsf{F}}}^*$-valued random variable.
\end{proposition}
At this point, it may appear interesting to examine the connections existing between the two proposed block-coordinate proximal
algorithms and published works.
\newpage
\begin{remark}\ 
\begin{enumerate}
\item In practice, one may be interested in problems of the form
\begin{equation}
\minimize{\mathsf{x}_1\in\HH_1,\ldots,\mathsf{x}_p\in\HH_p}
{\sum_{j=1}^p\big(\mathsf{f}_j(\mathsf{x}_j)+\mathsf{h}_j(\mathsf{x}_j)\big)+
\sum_{k=1}^q \mathsf{g}_k\bigg(\sum_{j=1}^p\mathsf{L}_{k,j}\mathsf{x}_{j}\bigg)}.
\end{equation}
These are special cases of  \eqref{e:primopt} where $(\forall k \in\{1,\ldots,q\})$ $\mathsf{l}_k = \iota_{\{0\}}$, i.e. $\mathsf{l}_k^* = 0$.
\item Algorithm \eqref{e:PDcoordopt1} extends the deterministic approaches in 
\cite{Chambolle_A_2010_first_opdacpai,Condat_L_2013_j-ota-primal-dsm,Esser_E_2010_j-siam-is_gen_fcf,He_B_2012_j-siam-is_conv_apd,Vu_B_2013_j-acm_spl_adm}, which deal with the case when $p=1$,
by introducing some random sweeping of the coordinates and by allowing the use of stochastic errors.
Similarly, Algorithm \eqref{e:PDcoordopt2} extends the algorithms in \cite{Chen_P_2013_j-inv-prob_prim_dfp,Loris_I_2011_generalization_ist} which were developed in a deterministic setting
in the absence of errors, in the scenario where $p = q = 1$, 
$\HH_1$ and $\GG_1$ are finite dimensional spaces,
$\mathsf{l}_1 = \iota_{\{0\}}$, $\mathsf{W}_1 = \tau \Id$ with $\tau \in \RPP$,
$\mathsf{U}_1 = \rho \Id$ with $\rho \in \RPP$, and
no relaxation ($\lambda_n \equiv 1$) or a constant one ($\lambda_n \equiv \lambda_0 < 1$) is performed.
Recently, these works have been generalized to possibly infinite-dimensional Hilbert spaces when $p=1$ and $q > 1$, arbitrary preconditioning operators
are employed, and deterministic summable errors are allowed \cite{Combettes_P_2014_p-icip_forward_bvo}.
The practical interest of introducing preconditioning operators for accelerating the convergence of primal-dual
proximal methods was emphasized in \cite{Combettes_P_2014_p-icip_forward_bvo,Pock_T_2008_p-iccv_diagonal_pffo,Repetti_A_2012_p-eusipco_penalized_wlsardcsdn}.
\item In \cite[Corollary 5.5]{Combettes_P_2014_stochastic_qfbc}, another random block-coordinate primal-dual algorithm was proposed to solve an instance of Problem \ref{prob:mainfunc}
obtained when $(\forall j \in \{1,\ldots,p\})$ $\mathsf{h}_j = \mathsf{0}$ and $(\forall k \in \{1,\ldots,q\})$
$\mathsf{l}_k = \iota_{\{\mathsf{0}\}}$. This algorithm is based on the Douglas-Rachford iteration which is also at the origin of
the randomized Alternating Direction Method of Multipliers (ADMM) developed in finite dimensional spaces in \cite{Iutzeler_F_2013_cdc_asynchronous_dora}.
Note however that the algorithm in \cite[Corollary 5.5]{Combettes_P_2014_stochastic_qfbc} requires to invert $\ID+\boldsymbol{\mathsf{L}}\boldsymbol{\mathsf{L}}^*$
or $\ID+\boldsymbol{\mathsf{L}}^*\boldsymbol{\mathsf{L}}$ (see \cite[Remark 5.4]{Combettes_P_2014_stochastic_qfbc}). By contrast, Algorithms \eqref{e:PDcoordopt1}
and \eqref{e:PDcoordopt2} do not make it necessary to perform any linear operator inversion.
\end{enumerate}
\end{remark}

\section{Asynchronous distributed algorithms}\label{se:dist}
In this part, $\HH$, $\GG_1,\ldots,\GG_m$ are separable real Hilbert spaces, $\GGG = \GG_1\oplus \cdots \oplus \GG_m$,
and the following problem is addressed:
\begin{problem}\label{pr:distmon}
For every  $i\in\{1,\ldots,m\}$, let 
$\mathsf{A}_i\colon\HH\to 2^{\HH}$ be maximally monotone,
let $\mathsf{C}_i\colon\HH\to \HH$ be cocoercive, let 
$\mathsf{B}_i\colon\GG_i\to 2^{\GG_i}$ be maximally monotone, 
let $\mathsf{D}_i\colon\GG_i\to 2^{\GG_i}$ be maximally monotone and strongly monotone,
and let $\mathsf{M}_{i}$ be a nonzero operator in $\BL(\HH,\GG_i)$. 
We assume that the set $\widehat{\mathsf{F}}$ of solutions to the problem:
\newpage
\begin{equation}
\label{e:prdist}
\text{find}\;\;\mathsf{x}\in\HH
\;\;\text{such that}\;\;0\in
\sum_{i=1}^m \mathsf{A}_i\mathsf{x}+\mathsf{C}_i\mathsf{x}+\mathsf{M}_{i}^*(\mathsf{B}_i\infconv\mathsf{D}_i)
(\mathsf{M}_{i}\mathsf{x})
\end{equation}
is nonempty. 
Our objective is to find a $\widehat{\mathsf{F}}$-valued random variable $\widehat{x}$.
\end{problem}
Problem \eqref{e:prdist} can be reformulated in the product space $\HH^m$ 
as
\begin{equation}\label{e:formparallel}
\mbox{find $(\mathsf{x}_1,\ldots,\mathsf{x}_m)\in \Lambda_m$ such that}\;\; 0 \in 
\sum_{i=1}^m \mathsf{A}_i\mathsf{x}_i+\mathsf{C}_i\mathsf{x}_i+\mathsf{M}_{i}^*(\mathsf{B}_i\infconv\mathsf{D}_i)
(\mathsf{M}_{i}\mathsf{x}_i)
\end{equation}
where
\begin{equation}
\Lambda_m = \menge{(\mathsf{x}_1,\ldots,\mathsf{x}_m)\in \HH^m}{\mathsf{x}_1 = \ldots = \mathsf{x}_m}.
\end{equation}
This kind of reformulation was employed in \cite{Combettes_PL_2008_j-ip_proximal_apdmfscvip,Pesquet_J_2012_j-pjpjoo_par_ipo} to obtain parallel algorithms
for finding a zero of a sum of maximal operators and it is also popular in consensus problems \cite{Boyd_S_2011_j-found-tml_distributed_osl_admm,Nedic_A_2010_inbook_cooperative_mao}.
To devise distributed algorithms, the involved linear constraint is further split in a set of
similar constraints, each of them involving a reduced subset of variables. In this context indeed, each index
$i\in \{1,\ldots,m\}$ corresponds to a given agent and a modeling of the topological relationships existing between the different
agents is needed. To do so, we define nonempty
subsets $(\mathbb{V}_\ell)_{1\le \ell \le r}$ of $\{1,\ldots,m\}$, with cardinalities 
$(\kappa_\ell)_{1 \le \ell \le r}$, which are such that:
\begin{assumption}\label{a:distrib}
For every $\boldsymbol{\mathsf{x}}=(\mathsf{x}_i)_{1\le i \le m} \in \HH^m$,
\begin{equation}
\boldsymbol{\mathsf{x}} \in \Lambda_m \qquad \Leftrightarrow \qquad (\forall \ell \in \{1,\ldots,r\})\quad (\mathsf{x}_i)_{i \in \mathbb{V}_\ell} \in \Lambda_{\kappa_{\ell}}.
\end{equation}
\end{assumption}
This assumption is obviously satisfied if $r = 1$ and $\mathbb{V}_1 = \{1,\ldots,m\}$, or
if $r=m-1$ and $(\forall \ell \in \{1,\ldots,m-1\})$ $\mathbb{V}_\ell = \{\ell,\ell+1\}$.
More generally if the sets $(\mathbb{V}_\ell)_{1\le \ell \le r}$ correspond to
the hyperedges of a hypergraph with vertices $\{1,\ldots,m\}$, then the assumption is equivalent to the fact
that the hypergraph is connected.

In the following, we will need to introduce the notation: 
\begin{align}
&\HHH = \HH^{\kappa_1} \oplus \cdots \oplus \HH^{\kappa_r}, \qquad\qquad\qquad\;\; \boldsymbol{\Lambda} = \Lambda_{\kappa_1}\oplus \cdots \oplus \Lambda_{\kappa_r},\\
&\boldsymbol{\mathsf A}=\cart_{\!i=1}^{\!m}\mathsf{A}_i, \qquad\qquad\qquad\qquad\qquad
\boldsymbol{\mathsf C}=\cart_{\!i=1}^{\!m}\mathsf{C}_i,\\
&\boldsymbol{\mathsf B}=\cart_{\!i=1}^{\!m}\mathsf{B}_i,\qquad\qquad\qquad\qquad\qquad
\boldsymbol{\mathsf D}=\cart_{\!i=1}^{\!m}\mathsf{D}_i,\\
&\boldsymbol{\mathsf{S}}\colon \HH^m\to \HHH\colon \boldsymbol{\mathsf{x}}\mapsto (\mathsf{S}_\ell \boldsymbol{\mathsf{x}})_{1\le \ell \le r},\qquad\quad
\boldsymbol{\mathsf{M}}\colon \HH^m \to \GGG\colon \boldsymbol{\mathsf{x}}\mapsto (\mathsf{M}_i\mathsf{x}_i)_{1\le i \le m},
\end{align}
where, for every $\ell \in \{1,\ldots,r\}$,
\begin{equation}\label{e:defSell}
\mathsf{S}_\ell \colon \HH^m \to \HH^{\kappa_\ell}\colon \boldsymbol{\mathsf{x}}\mapsto (\mathsf{x}_i)_{i\in \mathbb{V}_\ell}
=  (\mathsf{x}_{\mathsf{i}(\ell,j)})_{1\le j \le \kappa_\ell}
\end{equation}
and $\mathsf{i}(\ell,1),\ldots,\mathsf{i}(\ell,\kappa_\ell)$ denote the elements of $\mathbb{V}_\ell$ ordered in an increasing manner.
Note that, for every $\ell \in \{1,\ldots,r\}$, the adjoint of $\mathsf{S}_\ell$ is
\begin{equation}\label{e:Ladjdis}
\mathsf{S}_\ell^* \colon \HH^{\kappa_\ell}\to \HH^m\colon \mathsf{z}_\ell = (\mathsf{z}_{\ell,j})_{1\le j \le \kappa_\ell} \mapsto (\mathsf{x}_i)_{1\le i \le m}
\end{equation}
where
\begin{equation}\label{e:Ladjdisb}
(\forall i \in \{1,\ldots,m\})\qquad
\mathsf{x}_i = 
\begin{cases}
\mathsf{z}_{\ell,j} & \mbox{if $i = \mathsf{i}(\ell,j)$ with $j\in \{1,\ldots,\kappa_\ell\}$}\\
0 & \mbox{otherwise.}
\end{cases}
\end{equation}
The adjoint of $\boldsymbol{\mathsf{S}}$ is thus given by
\begin{equation}\label{e:defSs1}
\boldsymbol{\mathsf{S}}^*\colon \HHH\to \HH^m\colon (\mathsf{z}_\ell)_{1\le \ell \le r}\mapsto \sum_{\ell=1}^r \mathsf{S}_\ell^* \mathsf{z}_\ell = (\mathsf{x}_i)_{1\le i \le m}
\end{equation}
where, for every $i \in \{1,\ldots,m\}$,
\begin{equation}\label{e:defSs2}
\mathsf{x}_i = \sum_{(\ell,j) \in \mathbb{V}_i^*} \mathsf{z}_{\ell,j}
\end{equation}
with 
\begin{equation}
\mathbb{V}_i^* = \menge{(\ell,j)}{\ell\in \{1,\ldots,r\}, j\in \{1,\ldots,\kappa_\ell\}, \text{and}\;\,\mathsf{i}(\ell,j) = i}.
\end{equation}
As a consequence of Assumption~\ref{a:distrib}, the cardinality of $\mathbb{V}_i^*$ (i.e. the number of sets  $(\mathbb{V}_\ell)_{1\le \ell\le r}$ containing index $i$)
is nonzero.

The link between Problems \ref{pr:distmon} and \ref{prob:main} is now 
enlightened by the next result:
\begin{proposition}\label{co:formparallelvectdis}
Under Assumption \ref{a:distrib}, Problem \eqref{e:formparallel} is equivalent to
\begin{equation}\label{e:formparallelvect}
\text{find}\;\;\boldsymbol{\mathsf{x}}\in\HH^m
\;\;\text{such that}\;\;\boldsymbol{\mathsf{0}} \in  \boldsymbol{\mathsf{A}}\boldsymbol{\mathsf{x}}+\boldsymbol{\mathsf{C}}\boldsymbol{\mathsf{x}}+\boldsymbol{\mathsf{M}}^*(\boldsymbol{\mathsf{B}}\infconv\boldsymbol{\mathsf{D}})(\boldsymbol{\mathsf{M}}\boldsymbol{\mathsf{x}})+
\boldsymbol{\mathsf{S}}^*\boldsymbol{\mathsf{N}}_{\boldsymbol{\Lambda}}(\boldsymbol{\mathsf{S}}\boldsymbol{\mathsf{x}}).
\end{equation}
\end{proposition}
\begin{proof}
For every $\boldsymbol{\mathsf{x}}\in \HH^m$, we have the following simple equivalences:
\begin{align}
&\begin{cases}
\displaystyle 0 \in  \sum_{i=1}^m \mathsf{A}_i\mathsf{x}_i+\mathsf{C}_i\mathsf{x}_i+\mathsf{M}_{i}^*(\mathsf{B}_i\infconv\mathsf{D}_i)(\mathsf{M}_{i}\mathsf{x}_i)\\
\boldsymbol{\mathsf{x}} \in \Lambda_m
\end{cases}\nonumber\\
\Leftrightarrow\quad &\begin{cases}
(\forall i \in \{1,\ldots,m\})\quad 0 \in  \mathsf{A}_i\mathsf{x}_i+\mathsf{C}_i\mathsf{x}_i+\mathsf{M}_{i}^*(\mathsf{B}_i\infconv\mathsf{D}_i)(\mathsf{M}_{i}\mathsf{x}_i)+\mathsf{u}_i\\
\boldsymbol{\mathsf{x}} \in \Lambda_m\\
\displaystyle \sum_{i=1}^m \mathsf{u}_i = 0
\end{cases}\nonumber\\
\Leftrightarrow\quad &\begin{cases}
\boldsymbol{\mathsf{0}} \in  \boldsymbol{\mathsf{A}}\boldsymbol{\mathsf{x}}+\boldsymbol{\mathsf{C}}\boldsymbol{\mathsf{x}}+\boldsymbol{\mathsf{M}}^*(\boldsymbol{\mathsf{B}}\infconv\boldsymbol{\mathsf{D}})(\boldsymbol{\mathsf{M}}\boldsymbol{\mathsf{x}})+\boldsymbol{\mathsf{u}}\\
\boldsymbol{\mathsf{x}} \in \Lambda_m\\
\boldsymbol{\mathsf{u}} \in \Lambda_m^\perp
\end{cases}\nonumber\\
\Leftrightarrow\quad  & \boldsymbol{\mathsf{0}} \in  \boldsymbol{\mathsf{A}}\boldsymbol{\mathsf{x}}+\boldsymbol{\mathsf{C}}\boldsymbol{\mathsf{x}}+\boldsymbol{\mathsf{M}}^*(\boldsymbol{\mathsf{B}}\infconv\boldsymbol{\mathsf{D}})(\boldsymbol{\mathsf{M}}\boldsymbol{\mathsf{x}})+
\boldsymbol{\mathsf{S}}^*\boldsymbol{\mathsf{N}}_{\boldsymbol{\Lambda}}(\boldsymbol{\mathsf{S}}\boldsymbol{\mathsf{x}}),
\end{align}
\newpage
\noindent where we have used the fact that $\boldsymbol{\mathsf{N}}_{\Lambda_m} = \partial \iota_{\Lambda_m} = 
\partial(\iota_{\boldsymbol{\Lambda}}\circ \boldsymbol{\mathsf{S}}) = \boldsymbol{\mathsf{S}}^* \partial \iota_{\boldsymbol{\Lambda}} \boldsymbol{\mathsf{S}}
= \boldsymbol{\mathsf{S}}^*\boldsymbol{\mathsf{N}}_{\boldsymbol{\Lambda}} \boldsymbol{\mathsf{S}}$ since
$\boldsymbol{\Lambda} + \ran(\boldsymbol{\mathsf{S}})$ is a closed subspace of $\HHH$ \cite[Propositions 6.19 \& 16.42]{Bauschke_H_2011_book_con_amo}.
\end{proof}

\begin{remark}\label{re:formparallelvectdis}
If we now reexpress \eqref{e:formparallelvect} in terms of the notation used in Problem \ref{prob:main},
we see that the equivalence of Problem \ref{pr:distmon} with  Problem \ref{prob:main} is obtained by setting $p = m$, $q = m+r$, 
$\HH_1=\ldots=\HH_m = \HH$,
$(\forall \ell \in \{1,\ldots,r\})$ $\GG_{m+\ell} = \HH^{\kappa_\ell}$, $\mathsf{B}_{m+\ell} = \mathsf{N}_{\Lambda_{\kappa_\ell}}$,
$\mathsf{D}_{m+\ell} = \mathsf{N}_{\{0\}}$, and
\begin{align}
&\big(\forall (k,i) \in \{1,\ldots,m\}^2)\qquad \mathsf{L}_{k,i} =
\begin{cases}
\mathsf{M}_i & \mbox{if $k=i$}\\
\mathsf{0} & \mbox{otherwise,}
\end{cases}\\
&(\forall \ell \in \{1,\ldots,r\})(\forall \boldsymbol{\mathsf{x}}\in \HH^m)\qquad
\sum_{i=1}^m\mathsf{L}_{m+\ell,i}\,\mathsf{x}_i = \mathsf{S}_\ell \boldsymbol{\mathsf{x}}
\end{align}
(hence, \eqref{e:defLk} and \eqref{e:defLjs} are satisfied).
\end{remark}

Our goal now is to develop asynchronous distributed algorithms
for solving Problem~\ref{pr:distmon} in the sense that,  at each iteration of these algorithms,
a limited number of operators $(\mathsf{A}_i)_{1\le i \le m}$, $(\mathsf{B}_i)_{1\le i \le m}$, $(\mathsf{C}_i)_{1\le i \le m}$,
and $(\mathsf{D}_i)_{1\le i \le m}$ are activated in a random manner.
Based on the above remark, the following convergence result can be deduced from
Proposition~\ref{p:algopdmon1sym}.
\begin{proposition}\label{p:algopdmon1dist}
Let $(\theta_\ell)_{1\le \ell \le r} \in \RPP^r$.
For every $i\in \{1,\ldots,m\}$, let $\mathsf{W}_i$
be a strongly positive self-adjoint operator in $\BL(\HH)$
such that $\mathsf{W}_i^{1/2} \mathsf{C}_i \mathsf{W}_i^{1/2}$
is $\mu_i$-cocoercive with $\mu_i \in \RPP$, let $\mathsf{U}_i$
be a strongly positive self-adjoint operator in $\BL(\GG_i)$
such that $\mathsf{U}_i^{1/2} \mathsf{D}_i^{-1} \mathsf{U}_i^{1/2}$
is $\nu_i$-cocoercive with $\nu_i \in \RPP$, and let
\begin{equation}
\overline{\theta}_i= \sum_{\ell\in \{\ell'\in\{1,\ldots,r\}\mid i\in \mathbb{V}_{\ell'}\}} \theta_\ell.
\end{equation}
Suppose that
\begin{align}
&(\exists \alpha\in \RPP)\qquad  
(1-\chi)
\min\{\mu (1+\alpha \sqrt{\chi})^{-1},
\nu (1+\alpha^{-1} \sqrt{\chi})^{-1}\}> \frac12 \label{e:condvarthetaalphadist}
\end{align}
where 
\begin{equation}\label{e:defchi}
\chi = \max_{i\in\{1,\ldots,m\}}  \|\mathsf{U}_i^{1/2} \mathsf{M}_i \mathsf{W}_i^{1/2}\|^2+  \overline{\theta}_i \|\mathsf{W}_i\| ,
\end{equation}
$\mu = \min\{\mu_1,\ldots,\mu_m\}$, and $\nu=\min\{\nu_1,\ldots,\nu_m\}$.
Let $(\lambda_n)_{n\in\NN}$  be a sequence in $\left]0,1\right]$ such that 
$\inf_{n\in\NN}\lambda_n>0$,
let $\boldsymbol{x}_0$, 
$(\boldsymbol{a}_n)_{n\in\NN}$, and 
$(\boldsymbol{c}_n)_{n\in\NN}$ be $\HH^m$-valued random variables,
let $(v_{i,0})_{1\le i \le m}$, $(\boldsymbol{b}_n)_{n\in\NN}$, and 
$(\boldsymbol{d}_n)_{n\in\NN}$ be $\GGG$-valued random 
variables, for every $\ell \in \{1,\ldots,r\}$ let 
$v_{m+\ell,0} = (v_{m+\ell,j,0})_{1\le j \le \kappa_\ell}$ be a $\HH^{\kappa_\ell}$-valued random variable,
and let $(\boldsymbol{\varepsilon}_n)_{n\in\NN}$ be identically distributed 
$\mathbb{D}_{2m+r}$-valued random variables. 
Set $(\forall \ell \in \{1,\ldots,r\})$ $\overline{x}_{\ell,0} = \kappa_{\ell}^{-1}\sum_{i\in \mathbb{V}_\ell} x_{i,0}\,$, iterate
\begin{equation}\label{e:PDcoord1symdister}
\begin{array}{l}
\text{for}\;n=0,1,\ldots\\
\left\lfloor
\begin{array}{l}
\text{for}\;\ell=1,\ldots,r\\
\left\lfloor
\begin{array}{l}
\displaystyle \overline{u}_{\ell,n} = \varepsilon_{2m+\ell,n}\Big(\frac{1}{\kappa_\ell}
\sum_{j=1}^{\kappa_\ell} v_{m+\ell,j,n}+ \theta_\ell\, \overline{x}_{\ell,n}\Big)\\
\text{for}\;j=1,\ldots,\kappa_\ell\\
\left\lfloor
\begin{array}{l}
\displaystyle w_{\ell,j,n} = \varepsilon_{2m+\ell,n}\big(2(\theta_\ell\, x_{\mathsf{i}(\ell,j),n}-\overline{u}_{\ell,n})+v_{m+\ell,j,n})\\
\end{array}
\right.\\[2mm]
\end{array}
\right.\\
\text{for}\;i=1,\ldots,m\\
\left\lfloor
\begin{array}{l}
\displaystyle u_{i,n} = \varepsilon_{m+i,n}\Big(\mathsf{J}_{\mathsf{U}_i\mathsf{B}_i^{-1}}\big(v_{i,n}+\mathsf{U}_i\big(\mathsf{M}_i
x_{i,n}- \mathsf{D}_i^{-1}v_{i,n}+d_{i,n})\big)+b_{i,n}\Big)\\
\displaystyle y_{i,n} =
\varepsilon_{i,n}\Big(\mathsf{J}_{\mathsf{W}_i\mathsf{A}_i}\big(x_{i,n}-\mathsf{W}_i(\mathsf{M}_i^* (2u_{i,n}-v_{i,n})+\sum_{(\ell,j) \in \mathbb{V}_i^*}w_{\ell,j,n}+ \mathsf{C}_i x_{i,n} +c_{i,n})\big)+a_{i,n}\Big)\\
v_{i,n+1} = v_{i,n}+\lambda_n \varepsilon_{m+i,n} (u_{i,n}-v_{i,n})\\
x_{i,n+1} = x_{i,n}+\lambda_n \varepsilon_{i,n}\,  (y_{i,n}- x_{i,n})
\end{array}
\right.\\
\text{for}\;\ell=1,\ldots,r\\
\left\lfloor
\begin{array}{l}
\displaystyle v_{m+\ell,n+1} = v_{m+\ell,n}+\frac{\lambda_n}{2} \varepsilon_{2m+\ell,n} (w_{\ell,n}-v_{m+\ell,n})\\
 \eta_{\ell,n} = \max\menge{\varepsilon_{i,n}}{i\in \mathbb{V}_\ell}\\
\displaystyle \overline{x}_{\ell,n+1} = \overline{x}_{\ell,n}+\eta_{\ell,n}\Big(\frac{1}{\kappa_\ell}
\sum_{i\in \mathbb{V}_\ell} x_{i,n+1}-\overline{x}_{\ell,n}\Big),
\vspace*{0.1cm}
\end{array}
\right.
\vspace*{0.1cm}
\end{array}
\right.\\
\end{array}
\end{equation}
and set $(\forall n\in\NN)$ $\EEE_n=\sigma(\boldsymbol{\varepsilon}_n)$ and $\boldsymbol{\XX}_n=
\sigma(\boldsymbol{x}_{n'},\boldsymbol{v}_{n'})_{0\leq n'\leq n}$.
In addition, assume that the following hold:
\begin{enumerate}
\item
\label{c:PDcoord1idist}
$\sum_{n\in\NN}\sqrt{\EC{\|\boldsymbol{a}_n\|^2}
{\boldsymbol{\XX}_n}}<\pinf$,
$\sum_{n\in\NN}\sqrt{\EC{\|\boldsymbol{b}_n\|^2}
{\boldsymbol{\XX}_n}}<\pinf$,
$\sum_{n\in\NN}\sqrt{\EC{\|\boldsymbol{c}_n\|^2}
{\boldsymbol{\XX}_n}}<\pinf$, and
$\sum_{n\in\NN}\sqrt{\EC{\|\boldsymbol{d}_n\|^2}
{\boldsymbol{\XX}_n}}<\pinf$ $\as$
\item \label{c:PDcoord1iidist}
For every $n\in\NN$, $\EEE_n$ and $\XXX_n$ are independent, and
$(\forall i\in\{1,\ldots,m\})$ $\PP[\varepsilon_{i,0}=1] > 0$. 
\item \label{c:PDcoord1iiidist} For every $n\in \NN$,
\begin{equation}\label{incluicPDcoord2iii}
(\forall i\in \{1,\ldots,m\})\qquad \menge{\omega\in \Omega}{\varepsilon_{i,n}(\omega) = 1}
\subset \menge{\omega\in \Omega}{\varepsilon_{m+i,n}(\omega) = 1}
\end{equation}
and
\begin{equation}\label{incluicPDcoord2iiii}
(\forall \ell \in \{1,\ldots,r\})\;\; \bigcup_{i\in \mathbb{V}_\ell} \menge{\omega\in \Omega}{\varepsilon_{i,n}(\omega) = 1}
\subset \menge{\omega\in \Omega}{\varepsilon_{2m+\ell,n}(\omega) = 1}.
\end{equation}
\end{enumerate}
Then, under Assumption \ref{a:distrib}, for every $i\in \{1,\ldots,m\}$, $(x_{i,n})_{n\in\NN}$ converges weakly $\as$ to a 
$\widehat{\mathsf{F}}$-valued random variable $\widehat{x}$ and, for every $\ell\in \{1,\ldots,r\}$, $(\overline{x}_{\ell,n})_{n\in\NN}$ converges weakly $\as$
to $\widehat{x}$.
\end{proposition}
\begin{proof}
By using Proposition \ref{co:formparallelvectdis}, Remark \ref{re:formparallelvectdis}, \eqref{e:defSell}, \eqref{e:defSs1}-\eqref{e:defSs2}, 
setting
\begin{align}\label{e:Utheta}
(\forall \ell \in \{1,\ldots,r\}) \quad &\mathsf{U}_{m+\ell} = \theta_\ell \Id\\
&(\forall n \in \NN)\;\; 
\;\;  b_{m+\ell,n} = d_{m+\ell,n} = 0, \label{e:errbdzero}
\end{align}
and noticing that $\mathsf{J}_{\mathsf{U}_{m+\ell} \mathsf{N}_{\Lambda_{\kappa_\ell}}^{-1}} = \Id-\theta_\ell \Pi_{\Lambda_{\kappa_\ell}}(\cdot/\theta_\ell)
= \Id-\Pi_{\Lambda_{\kappa_\ell}}$ (see \eqref{e:moreaugen}), 
Algorithm \eqref{e:PDcoord1sym} for solving Problem \eqref{e:formparallel} reads
\begin{equation}\label{e:revalgodistproj}
\begin{array}{l}
\text{for}\;n=0,1,\ldots\\
\left\lfloor
\begin{array}{l}
\text{for}\;i=1,\ldots,m\\
\left\lfloor
\begin{array}{l}
\displaystyle u_{i,n} = \varepsilon_{m+i,n}\Big(\mathsf{J}_{\mathsf{U}_i\mathsf{B}_i^{-1}}\big(v_{i,n}+\mathsf{U}_i\big(\mathsf{M}_i
x_{i,n}- \mathsf{D}_i^{-1}v_{i,n}+d_{i,n})\big)+b_{i,n}\Big)\\
v_{i,n+1} = v_{i,n}+\lambda_n \varepsilon_{m+i,n} (u_{i,n}-v_{i,n})
\end{array}
\right.\\
\text{for}\;\ell=1,\ldots,r\\
\left\lfloor
\begin{array}{l}
\displaystyle u_{m+\ell,n} = \varepsilon_{2m+\ell,n}\big(v_{m+\ell,n}+\theta_\ell\, (x_{i,n})_{i\in \mathbb{V}_\ell}-\Pi_{\Lambda_{\kappa_\ell}}(v_{m+\ell,n}+\theta_\ell\, (x_{i,n})_{i\in \mathbb{V}_\ell})\big)\\
v_{m+\ell,n+1} = v_{m+\ell,n}+\lambda_n \varepsilon_{2m+\ell,n} \big(u_{m+\ell,n}-v_{m+\ell,n}\big)
\end{array}
\right.\\
\text{for}\;i=1,\ldots,m\\
\left\lfloor
\begin{array}{l}
\displaystyle y_{i,n} =
\varepsilon_{i,n}\Big(\mathsf{J}_{\mathsf{W}_i\mathsf{A}_i}\big(x_{i,n}-\mathsf{W}_i(\mathsf{M}_i^* (2u_{i,n}-v_{i,n})+\sum_{(\ell,j) \in \mathbb{V}_i^*}( 2u_{m+\ell,j,n}-v_{m+\ell,j,n})\\
\qquad\qquad\qquad\quad + \mathsf{C}_i x_{i,n} +c_{i,n})\big)+a_{i,n}\Big)\\
x_{i,n+1} = x_{i,n}+\lambda_n \varepsilon_{i,n}\,  (y_{i,n}- x_{i,n}).
\end{array}
\right.\\
\end{array}
\right.
\end{array}
\end{equation} 
Making explicit the form of the projections onto the vector spaces $(\Lambda_{\kappa_\ell})_{1\le \ell \le r}$ leads to 
\begin{equation}\label{e:PDcoord1symdist}
\begin{array}{l}
\text{for}\;n=0,1,\ldots\\
\left\lfloor
\begin{array}{l}
\text{for}\;i=1,\ldots,m\\
\left\lfloor
\begin{array}{l}
\displaystyle u_{i,n} = \varepsilon_{m+i,n}\Big(\mathsf{J}_{\mathsf{U}_i\mathsf{B}_i^{-1}}\big(v_{i,n}+\mathsf{U}_i\big(\mathsf{M}_i
x_{i,n}- \mathsf{D}_i^{-1}v_{i,n}+d_{i,n})\big)+b_{i,n}\Big)\\
v_{i,n+1} = v_{i,n}+\lambda_n \varepsilon_{m+i,n} (u_{i,n}-v_{i,n})\\
\end{array}
\right.\\
\text{for}\;\ell=1,\ldots,r\\
\left\lfloor
\begin{array}{l}
\displaystyle \overline{u}_{\ell,n} = \varepsilon_{2m+\ell,n}\kappa_\ell^{-1}
\Big(\sum_{j=1}^{\kappa_\ell} v_{m+\ell,j,n}+ \theta_\ell \sum_{i\in \mathbb{V}_\ell} x_{i,n}\Big)\\
\text{for}\;j=1,\ldots,\kappa_\ell\\
\left\lfloor
\begin{array}{l}
\displaystyle u_{m+\ell,j,n} = \varepsilon_{2m+\ell,n}\big(v_{m+\ell,j,n}+\theta_\ell\, x_{\mathsf{i}(\ell,j),n}-\overline{u}_{\ell,n}\big)\\
\end{array}
\right.\\
v_{m+\ell,n+1} = v_{m+\ell,n}+\lambda_n \varepsilon_{2m+\ell,n} \big(u_{m+\ell,n}-v_{m+\ell,n}\big)
\end{array}
\right.\\
\text{for}\;i=1,\ldots,m\\
\left\lfloor
\begin{array}{l}
\displaystyle y_{i,n} =
\varepsilon_{i,n}\Big(\mathsf{J}_{\mathsf{W}_i\mathsf{A}_i}\big(x_{i,n}-\mathsf{W}_i(\mathsf{M}_i^* (2u_{i,n}-v_{i,n})+\sum_{(\ell,j) \in \mathbb{V}_i^*}( 2u_{m+\ell,j,n}-v_{m+\ell,j,n})\\
\qquad\qquad\qquad\quad + \mathsf{C}_i x_{i,n} +c_{i,n})\big)+a_{i,n}\Big)\\
x_{i,n+1} = x_{i,n}+\lambda_n \varepsilon_{i,n}\,  (y_{i,n}- x_{i,n}).
\end{array}
\right.\\
\end{array}
\right.
\end{array}
\end{equation}
By defining now, for every $n\in \NN$ and $\ell\in \{1,\ldots,r\}$,
\begin{align}
&\overline{x}_{\ell,n} = \frac{1}{\kappa_{\ell}}\sum_{i\in \mathbb{V}_\ell} x_{i,n}, \label{e:defxbar}\\
&w_{\ell,n} = \varepsilon_{2m+\ell,n}(2u_{m+\ell,n}-v_{m+\ell,n}), \\
&\eta_{\ell,n} = \max\menge{\varepsilon_{i,n}}{i\in \mathbb{V}_\ell},\label{e:incluepsilondist1}
\end{align}
and using \eqref{incluicPDcoord2iiii} and the update equation
\begin{equation}
\displaystyle \overline{x}_{\ell,n+1} = \overline{x}_{\ell,n}+\eta_{\ell,n}\Big(\frac{1}{\kappa_\ell}
\sum_{i\in \mathbb{V}_\ell} x_{i,n+1}-\overline{x}_{\ell,n}\Big),
\end{equation}
\eqref{e:PDcoord1symdister} is obtained after
reordering the computation steps in  \eqref{e:PDcoord1symdist}.

In order to apply Proposition~\ref{p:algopdmon1sym}, we shall now show that Condition \eqref{e:condvarthetaalpha}
where $\vartheta_\alpha$ is defined by \eqref{e:varthetalpha} is fulfilled. Let 
\begin{equation}
\boldsymbol{\mathsf{W}}\colon \HH^m \to \HH^m\colon \boldsymbol{\mathsf{x}}\mapsto (\mathsf{W}_i \mathsf{x}_i)_{1\le i \le m}\\
\quad\text{and}\quad  \boldsymbol{\mathsf{U}}\colon \GGG\oplus \HHH \to \GGG\oplus \HHH\colon \boldsymbol{\mathsf{v}}\mapsto (\mathsf{U}_k \mathsf{v}_k)_{1\le k \le m+r}.
\end{equation}
According to Remark \ref{re:formparallelvectdis} and \eqref{e:Utheta}, we have
\begin{equation}
(\forall \boldsymbol{\mathsf{x}} \in \HH^m) \qquad 
\boldsymbol{\mathsf{U}}^{1/2} \boldsymbol{\mathsf{L}} \boldsymbol{\mathsf{W}}^{1/2}\boldsymbol{\mathsf{x}}  = 
({\boldsymbol{\mathsf{U}}}_1^{1/2} \boldsymbol{\mathsf{M}} \boldsymbol{\mathsf{W}}^{1/2}\boldsymbol{\mathsf{x}},
{\boldsymbol{\mathsf{U}}}_2^{1/2} \boldsymbol{\mathsf{S}} \boldsymbol{\mathsf{W}}^{1/2}\boldsymbol{\mathsf{x}})
\end{equation}
where ${\boldsymbol{\mathsf{U}}}_1\colon \GGG \to \GGG\colon (\mathsf{v}_i)_{1\le i \le m}\mapsto (\mathsf{U}_i \mathsf{v}_i)_{1\le i \le m}$
and ${\boldsymbol{\mathsf{U}}}_2\colon \HHH \to \HHH\colon (\mathsf{v}_{m+\ell})_{1\le \ell \le r}\mapsto (\theta_\ell \mathsf{v}_{m+\ell})_{1\le \ell \le r}$.
This allows us to deduce that
\begin{align}
\|\boldsymbol{\mathsf{U}}^{1/2} \boldsymbol{\mathsf{L}} \boldsymbol{\mathsf{W}}^{1/2}\boldsymbol{\mathsf{x}}\|^2
&= \|{\boldsymbol{\mathsf{U}}}_1^{1/2} \boldsymbol{\mathsf{M}} \boldsymbol{\mathsf{W}}^{1/2}\boldsymbol{\mathsf{x}}\|^2
+  \|{\boldsymbol{\mathsf{U}}}_2^{1/2}\boldsymbol{\mathsf{S}} \boldsymbol{\mathsf{W}}^{1/2}\boldsymbol{\mathsf{x}}\|^2\nonumber\\
&= \sum_{i=1}^m \|\mathsf{U}_i^{1/2} \mathsf{M}_i \mathsf{W}_i^{1/2}\mathsf{x}_i\|^2+  
\scal{\boldsymbol{\mathsf{W}}^{1/2}\boldsymbol{\mathsf{x}}}{\boldsymbol{\mathsf{S}}^*{\boldsymbol{\mathsf{U}}}_2\boldsymbol{\mathsf{S}}\boldsymbol{\mathsf{W}}^{1/2}\boldsymbol{\mathsf{x}}}.
\end{align}
By using \eqref{e:defSell} and \eqref{e:defSs1}-\eqref{e:defSs2}, it can be further noticed that 
\begin{equation}\label{e:LsLdis}
\boldsymbol{\mathsf{S}}^*{\boldsymbol{\mathsf{U}}}_2\boldsymbol{\mathsf{S}}\colon \HH^m \to \HH^m \colon (\mathsf{x}_i)_{1\le i \le m} \mapsto
(\overline{\theta}_i \mathsf{x}_i)_{1\le i \le m}
\end{equation}
which yields
\begin{align}
\|\boldsymbol{\mathsf{U}}^{1/2} \boldsymbol{\mathsf{L}} \boldsymbol{\mathsf{W}}^{1/2}\boldsymbol{\mathsf{x}}\|^2
&= \sum_{i=1}^m \|\mathsf{U}_i^{1/2} \mathsf{M}_i \mathsf{W}_i^{1/2}\mathsf{x}_i\|^2+ 
\sum_{i=1}^m \overline{\theta}_i \|\mathsf{W}_i^{1/2}\mathsf{x}_i\|^2\nonumber\\
& \le \sum_{i=1}^m (\|\mathsf{U}_i^{1/2} \mathsf{M}_i \mathsf{W}_i^{1/2}\|^2+ \overline{\theta}_i\|\mathsf{W}_i\|)\|\mathsf{x}_i\|^2
\le \chi \|\boldsymbol{\mathsf{x}}\|^2,
\end{align}
so leading to
$\|\boldsymbol{\mathsf{U}}^{1/2} \boldsymbol{\mathsf{L}} \boldsymbol{\mathsf{W}}^{1/2}\|^2
\le \chi$.
This shows that \eqref{e:condvarthetaalphadist} implies \eqref{e:condvarthetaalpha}.

In addition, 
Condition \ref{c:PDcoord1iiisym} in Proposition \ref{p:algopdmon1sym} translates into 
Condition \ref{c:PDcoord1iiidist} in the present proposition.
It then follows from Propositions \ref{p:algopdmon1sym} and \ref{co:formparallelvectdis}
that, for every $i\in \{1,\ldots,m\}$, $(x_{i,n})_{n\in \NN}$ converges weakly $\as$ to a 
$\widehat{\mathsf{F}}$-valued random variable $\widehat{x}$. 
As a straightforward consequence of \eqref{e:defxbar}, for every $\ell \in \{1,\ldots,r\}$,
$(\overline{x}_{\ell,n})_{n\in \NN}$ also converges weakly $\as$ to $\widehat{x}$.
\end{proof}

\begin{remark}\ \label{re:aldist1}
\begin{enumerate}
\item 
The $n$-th iteration of Algorithm \eqref{e:PDcoord1symdister} basically consists of two kind of operations: the first ones update some of the 
variables $(x_{i,n})_{1\le i \le m}$ and $(v_{i,n})_{1\le i \le m}$ using the operators $(\mathsf{J}_{\mathsf{W}_i\mathsf{A}_i})_{1\le i \le m}$,
$(\mathsf{J}_{\mathsf{U}_i\mathsf{B}_i^{-1}})_{1\le i \le m}$, $(\mathsf{C}_i)_{1\le i \le m}$, and $(\mathsf{D}_i^{-1})_{1\le i \le m}$, while the second
ones can be viewed as merging steps performed on the sets $(\mathbb{V}_\ell)_{1\le \ell \le r}$. 
In this context, a simple choice for the Boolean random variables $(\varepsilon_{k,n})_{m+1\le k \le 2m+r}$ to satisfy Condition \ref{c:PDcoord1iiidist} is:
for every $n\in \NN$,
\begin{align}
&(\forall i\in \{1,\ldots,m\})\qquad \varepsilon_{m+i,n} = \varepsilon_{i,n},\\
&(\forall \ell \in \{1,\ldots,r\})\qquad
\varepsilon_{2m+\ell,n} = \eta_{\ell,n} = \max\menge{\varepsilon_{i,n}}{i\in \mathbb{V}_\ell} \label{e:contepseta}.
\end{align}
\item \label{re:aldist1ii} From \eqref{e:revalgodistproj}, it can be noticed that, for every $n\in \NN$ and $\ell\in \{1,\ldots,r\}$,
$\Pi_{\Lambda_{\kappa_\ell}} u_{m+\ell,n} = 0$, which implies that the following recursive relation holds:
\begin{equation}
\sum_{j=1}^{\kappa_\ell} v_{m+\ell,j,n+1} = (1-\lambda_n \varepsilon_{2m+\ell,n})\sum_{j=1}^{\kappa_\ell} v_{m+\ell,j,n}.
\end{equation}
In particular, if the initial values $(v_{m+\ell,0})_{1\le \ell \le r}$ are chosen such that
\begin{equation}\label{e:initzero}
(\forall \ell \in \{1,\ldots,r\})\qquad \sum_{j=1}^{\kappa_\ell} v_{m+\ell,j,0} = 0,
\end{equation}
then Algorithm \eqref{e:PDcoord1symdister} simplifies to
\begin{equation}\label{e:PDcoord1symdistersimp}
\begin{array}{l}
\text{for}\;n=0,1,\ldots\\
\left\lfloor
\begin{array}{l}
\text{for}\;\ell=1,\ldots,r\\
\left\lfloor
\begin{array}{l}
\text{for}\;j=1,\ldots,\kappa_\ell\\
\left\lfloor
\begin{array}{l}
\displaystyle w_{\ell,j,n} = \varepsilon_{2m+\ell,n}\big(2\theta_\ell\,(x_{\mathsf{i}(\ell,j),n}-\overline{x}_{\ell,n})+v_{m+\ell,j,n})\\
\end{array}
\right.\\[2mm]
\end{array}
\right.\\
\text{for}\;i=1,\ldots,m\\
\left\lfloor
\begin{array}{l}
\displaystyle u_{i,n} = \varepsilon_{m+i,n}\Big(\mathsf{J}_{\mathsf{U}_i\mathsf{B}_i^{-1}}\big(v_{i,n}+\mathsf{U}_i\big(\mathsf{M}_i
x_{i,n}- \mathsf{D}_i^{-1}v_{i,n}+d_{i,n})\big)+b_{i,n}\Big)\\
\displaystyle y_{i,n} =
\varepsilon_{i,n}\Big(\mathsf{J}_{\mathsf{W}_i\mathsf{A}_i}\big(x_{i,n}-\mathsf{W}_i(\mathsf{M}_i^* (2u_{i,n}-v_{i,n})+\sum_{(\ell,j) \in \mathbb{V}_i^*}w_{\ell,j,n}+ \mathsf{C}_i x_{i,n} +c_{i,n})\big)+a_{i,n}\Big)\\
v_{i,n+1} = v_{i,n}+\lambda_n \varepsilon_{m+i,n} (u_{i,n}-v_{i,n})\\
x_{i,n+1} = x_{i,n}+\lambda_n \varepsilon_{i,n}\,  (y_{i,n}- x_{i,n})
\end{array}
\right.\\
\text{for}\;\ell=1,\ldots,r\\
\left\lfloor
\begin{array}{l}
\displaystyle v_{m+\ell,n+1} = v_{m+\ell,n}+\frac{\lambda_n}{2} \varepsilon_{2m+\ell,n} (w_{\ell,n}-v_{m+\ell,n})\\
 \eta_{\ell,n} = \max\menge{\varepsilon_{i,n}}{i\in \mathbb{V}_\ell}\\
\displaystyle \overline{x}_{\ell,n+1} = \overline{x}_{\ell,n}+\eta_{\ell,n}\Big(\frac{1}{\kappa_\ell}
\sum_{i\in \mathbb{V}_\ell} x_{i,n+1}-\overline{x}_{\ell,n}\Big).
\vspace*{0.1cm}
\end{array}
\right.
\vspace*{0.1cm}
\end{array}
\right.\\
\end{array}
\end{equation}
\item Similarly to Remark \ref{re:mon1}\ref{re:mon1iii}, a sufficient condition for \eqref{e:condvarthetaalphadist} to be satisfied
is obtained by setting $\alpha =1$:
\begin{equation}
(1-\sqrt{\chi}) \min\{\mu,\nu\}> \frac12.
\end{equation}
\item When, for every $i\in \{1,\ldots,m\}$, $\mathsf{D}_i^{-1} = \mathsf{0}$, a looser condition is
\begin{equation}
(1-\chi) \mu> \frac12.
\end{equation}
In addition, if $(\forall i\in\{1,\ldots,m\})$ $\mathsf{B}_i = \mathsf{0}$, $(\|\mathsf{M}_i\|)_{1\le i \le m}$
can be chosen as small as desired, so that we can set 
$\chi = \max_{i\in\{1,\ldots,m\}}    \overline{\theta}_i \|\mathsf{W}_i\|$.
In this case, Algorithm \eqref{e:PDcoord1symdistersimp} can be simplified, by noting that, for every $n\in \NN$, the computation of variables $(u_{i,n})_{1 \le i \le m}$ and $(v_{i,n})_{1 \le i \le m}$ becomes useless.
By imposing \eqref{e:contepseta} and \eqref{e:initzero}, and by setting 
\begin{equation}
(\forall n \in \NN)
(\forall \ell \in \{1,\ldots,m\}) \quad \widetilde{v}_{\ell,n} = v_{m+\ell,n},
\end{equation}
we get
\begin{equation}
\begin{array}{l}
\text{for}\;n=0,1,\ldots\\
\left\lfloor
\begin{array}{l}
\text{for}\;\ell=1,\ldots,r\\
\left\lfloor
\begin{array}{l}
 \eta_{\ell,n} = \max\menge{\varepsilon_{i,n}}{i\in \mathbb{V}_\ell}\\
\text{for}\;j=1,\ldots,\kappa_\ell\\
\left\lfloor
\begin{array}{l}
\displaystyle w_{\ell,j,n} = \eta_{\ell,n}\big(2\theta_\ell(x_{\mathsf{i}(\ell,j),n}-\overline{x}_{\ell,n})+\widetilde{v}_{\ell,j,n})\\
\end{array}
\right.\\[2mm]
\end{array}
\right.\\
\text{for}\;i=1,\ldots,m\\
\left\lfloor
\begin{array}{l}
\displaystyle y_{i,n} =
\varepsilon_{i,n}\Big(\mathsf{J}_{\mathsf{W}_i\mathsf{A}_i}\big(x_{i,n}-\mathsf{W}_i(\sum_{(\ell,j) \in \mathbb{V}_i^*}w_{\ell,j,n}+ \mathsf{C}_i x_{i,n} +c_{i,n})\big)+a_{i,n}\Big)\\
x_{i,n+1} = x_{i,n}+\lambda_n \varepsilon_{i,n}\,  (y_{i,n}- x_{i,n})\\
\end{array}
\right.\\
\text{for}\;\ell=1,\ldots,r\\
\left\lfloor
\begin{array}{l}
\displaystyle \widetilde{v}_{\ell,n+1} = \widetilde{v}_{\ell,n}+\frac{\lambda_n}{2} \eta_{\ell,n} (w_{\ell,n}-\widetilde{v}_{\ell,n})\\
\displaystyle \overline{x}_{\ell,n+1} = \overline{x}_{\ell,n}+\eta_{\ell,n}\Big(\frac{1}{\kappa_\ell}
\sum_{i\in \mathbb{V}_\ell} x_{i,n+1}-\overline{x}_{\ell,n}\Big).
\vspace*{0.1cm}
\end{array}
\right.
\vspace*{0.1cm}
\end{array}
\right.\\
\end{array}
\end{equation}
\item An alternative distributed algorithm can be deduced from Proposition~\ref{p:algopdmon1},
which however necessitates, at each iteration $n\in \NN$, to update all the variables 
$(x_{i,n})_{i\in \mathbb{V}_\ell}$ corresponding to the sets $\mathbb{V}_\ell$ with $\ell \in \{1,\ldots,r\}$ which are 
randomly activated.
\end{enumerate}
\end{remark}

\newpage
As an offspring of Proposition \ref{p:algopdmon2}, another form of distributed algorithm is obtained:
\begin{proposition}\label{p:algopdmon2dist}
 Let $(\theta_\ell)_{1\le \ell \le r}$, $(\mathsf{W}_i)_{1\le i \le m}$, $(\mathsf{U}_i)_{1\le i \le m}$, 
$\mu$, $\nu$, and $\chi$ be defined as in Proposition~\ref{p:algopdmon1dist}.
Suppose that 
\begin{equation}\label{e:condvartheta2dist}
\min\big\{\mu,\nu(1-\chi)\big\} > \frac12.
\end{equation}
Let $(\lambda_n)_{n\in\NN}$  be a sequence in $\left]0,1\right]$ such that 
$\inf_{n\in\NN}\lambda_n>0$,
let $\boldsymbol{x}_0$ and 
$(\boldsymbol{c}_n)_{n\in\NN}$ be $\HH^m$-valued random variables,
let $(v_{i,0})_{1\le i \le m}$, $(\boldsymbol{b}_n)_{n\in\NN}$, and 
$(\boldsymbol{d}_n)_{n\in\NN}$ be $\GGG$-valued random 
variables, 
let $(v_{m+\ell,0})_{1\le \ell \le r}$ be a $\HHH$-valued random variable satisfying
\eqref{e:initzero},
and let $(\boldsymbol{\varepsilon}_n)_{n\in\NN}$ be identically distributed 
$\mathbb{D}_{2m+r}$-valued random variables. Iterate
\begin{equation}\label{e:PDcoord2dist}
\begin{array}{l}
\text{for}\;n=0,1,\ldots\\
\left\lfloor
\begin{array}{l}
\text{for}\;i=1,\ldots,m\\
\left\lfloor
\begin{array}{l}
\eta_{i,n} = \max\big\{\varepsilon_{m+i,n},(\varepsilon_{2m+\ell,n})_{\ell\in \{\ell'\in \{1,\ldots,r\}\mid i\in \mathbb{V}_{\ell'}\}}\big\}\\
\displaystyle w_{i,n} =
\eta_{i,n}\big(x_{i,n}-\mathsf{W}_i(\mathsf{C}_i x_{i,n} +c_{i,n})\big)\\
\displaystyle \widetilde{w}_{i,n} =
\eta_{i,n}\big(w_{i,n}-
\mathsf{W}_i(\mathsf{M}_i^*v_{i,n}+\sum_{(\ell,j) \in \mathbb{V}_i^*} v_{m+\ell,j,n})\big)\\
\displaystyle u_{i,n} = \varepsilon_{m+i,n}\Big(\mathsf{J}_{\mathsf{U}_i\mathsf{B}_i^{-1}}\big(v_{i,n}+\mathsf{U}_i(\mathsf{M}_i \widetilde{w}_{i,n}- \mathsf{D}_i^{-1}v_{i,n}+d_{i,n})\big)+b_{i,n}\Big)\\
v_{i,n+1} = v_{i,n}+\lambda_n \varepsilon_{m+i,n} (u_{i,n}-v_{i,n})
\end{array}
\right.\\
\text{for}\;\ell=1,\ldots,r\\
\left\lfloor
\begin{array}{l}
\displaystyle \overline{w}_{\ell,n} = \varepsilon_{2m+\ell,n}\,\frac{\theta_\ell}{\kappa_\ell} \sum_{i\in \mathbb{V}_\ell} \widetilde{w}_{i,n}\\
\text{for}\;j=1,\ldots,\kappa_\ell\\
\left\lfloor
\begin{array}{l}
\displaystyle u_{m+\ell,j,n} = \varepsilon_{2m+\ell,n}\Big(v_{m+\ell,j,n}+\theta_\ell\,
\widetilde{w}_{\mathsf{i}(j,\ell),n}-\overline{w}_{\ell,n}\Big)
\end{array}
\right.\\
v_{m+\ell,n+1} = v_{m+\ell,n}+\lambda_n \varepsilon_{2m+\ell,n} (u_{m+\ell,n}-v_{m+\ell,n})\\
\end{array}
\right.\\
\text{for}\;i=1,\ldots,m\\
\left\lfloor
\begin{array}{l}
\displaystyle  x_{i,n+1} = x_{i,n}+\lambda_n \varepsilon_{i,n}  \Big(w_{i,n}-\mathsf{W}_i \big(\mathsf{M}_i^* u_{i,n}+
\sum_{(\ell,j) \in \mathbb{V}_i^*} u_{m+\ell,j,n}\big)- x_{i,n}\Big),
\end{array}
\right.\vspace*{0.1cm}
\end{array}
\right.
\end{array}
\end{equation}
and set $(\forall n\in\NN)$ $\EEE_n=\sigma(\boldsymbol{\varepsilon}_n)$ and $\boldsymbol{\XX}_n=
\sigma(\boldsymbol{x}_{n'},\boldsymbol{v}_{n'})_{0\leq n'\leq n}$.
In addition, assume that
\begin{enumerate}
\item
\label{c:PDcoord2idist}
$\sum_{n\in\NN}\sqrt{\EC{\|\boldsymbol{b}_n\|^2}
{\boldsymbol{\XX}_n}}<\pinf$,
$\sum_{n\in\NN}\sqrt{\EC{\|\boldsymbol{c}_n\|^2}
{\boldsymbol{\XX}_n}}<\pinf$, and
$\sum_{n\in\NN}\sqrt{\EC{\|\boldsymbol{d}_n\|^2}
{\boldsymbol{\XX}_n}}<\pinf$ $\as$
\end{enumerate}
and Conditions \ref{c:PDcoord1iidist}-\ref{c:PDcoord1iiidist} in Proposition \ref{p:algopdmon1dist} hold.\\
If Assumption \ref{a:distrib} holds and, in Problem~\ref{pr:distmon}, $(\forall i \in \{1,\ldots,m\})$ $\mathsf{A}_i = 0$,
then, for every $i\in \{1,\ldots,m\}$, $(x_{i,n})_{n\in\NN}$ converges weakly $\as$ to a 
$\widehat{\mathsf{F}}$-valued random variable $\widehat{x}$.
\end{proposition}
\begin{proof}
By choosing $(\mathsf{U}_{m+\ell})_{1\le \ell \le r}$ as in \eqref{e:Utheta} and cancelling some error terms as in \eqref{e:errbdzero}, Algorithm~\eqref{e:PDcoord2} for solving Problem \eqref{e:formparallel}
can be expressed as
\begin{equation}
\begin{array}{l}
\text{for}\;n=0,1,\ldots\\
\left\lfloor
\begin{array}{l}
\text{for}\;i=1,\ldots,m\\
\left\lfloor
\begin{array}{l}
\eta_{i,n} = \max\big\{\varepsilon_{m+i,n},\big(\varepsilon_{2m+\ell,n}\big)_{\ell \in\{\ell'\in \{1,\ldots,r\}\mid i \in \mathbb{V}_{\ell'}\}}\big\}\\
\displaystyle w_{i,n} =
\eta_{i,n}\big(x_{i,n}-\mathsf{W}_i(\mathsf{C}_i x_{i,n} +c_{i,n})\big)\\
\displaystyle \widetilde{w}_{i,n} =
\eta_{i,n}\big(w_{i,n}-
\mathsf{W}_i(\mathsf{M}_i^*v_{i,n}+\sum_{(\ell,j) \in \mathbb{V}_i^*} v_{m+\ell,j,n})\big)\\
\displaystyle u_{i,n} = \varepsilon_{m+i,n}\Big(\mathsf{J}_{\mathsf{U}_i\mathsf{B}_i^{-1}}\big(v_{i,n}+\mathsf{U}_i(\mathsf{M}_i \widetilde{w}_{i,n}- \mathsf{D}_i^{-1}v_{i,n}+d_{i,n})\big)+b_{i,n}\Big)\\
v_{i,n+1} = v_{i,n}+\lambda_n \varepsilon_{m+i,n} (u_{i,n}-v_{i,n})
\end{array}
\right.\\
\text{for}\;\ell=1,\ldots,r\\
\left\lfloor
\begin{array}{l}
\displaystyle u_{m+\ell,n} = \varepsilon_{2m+\ell,n}\big(v_{m+\ell,n}+\theta_\ell\,
(\widetilde{w}_{i,n})_{i\in \mathbb{V}_\ell}-\Pi_{\Lambda_{\kappa_\ell}}(v_{m+\ell,n}+\theta_\ell\,
(\widetilde{w}_{i,n})_{i\in \mathbb{V}_\ell})\big)\\
v_{m+\ell,n+1} = v_{m+\ell,n}+\lambda_n \varepsilon_{2m+\ell,n} (u_{m+\ell,n}-v_{m+\ell,n})
\end{array}
\right.\\
\text{for}\;i=1,\ldots,m\\
\left\lfloor
\begin{array}{l}
\displaystyle  x_{i,n+1} = x_{i,n}+\lambda_n \varepsilon_{i,n}  \Big(w_{i,n}-\mathsf{W}_i \big(\mathsf{M}_i^* u_{i,n}+
\sum_{(\ell,j) \in \mathbb{V}_i^*} u_{m+\ell,j,n}\big)- x_{i,n}\Big).
\end{array}
\right.\vspace*{0.1cm}
\end{array}
\right.
\end{array}
\end{equation}
The rest of the proof is skipped due to its similarity with the proof of Proposition \ref{p:algopdmon1dist}.
\end{proof}
\begin{remark}
When $(\forall i \in \{1,\ldots,m\})$ $\mathsf{D}_i^{-1} = \mathsf{0}$, Condition \eqref{e:condvartheta2dist} can be rewritten as
\begin{equation}
(\forall i\in\{1,\ldots,m\})\qquad  \|\mathsf{U}_i^{1/2} \mathsf{M}_i \mathsf{W}_i^{1/2}\|^2+  \overline{\theta}_i \|\mathsf{W}_i\|<1 \;\;\;\text{and}\;\;\; \mu_i > 1/2.
\end{equation}
\end{remark}

As an illustration of the previous results in this section, let us consider variational problems which can be expressed as follows: 
\begin{problem}
\label{prob:mainfuncdist}
For every 
$i\in\{1,\ldots,m\}$, let 
$\mathsf{f}_i\in \Gamma_0(\HH)$,
let $\mathsf{h}_i\in \Gamma_0(\HH)$ be Lipschitz-differentiable,
let $\mathsf{g}_i\in \Gamma_0(\GG_i)$, 
let $\mathsf{l}_i\in \Gamma_0(\GG_i)$ be strongly convex,
and let $\mathsf{M}_i$ be a nonzero operator in $\BL(\HH,\GG_i)$. 
Suppose that there exists $\overline{\mathsf{x}}\in \HH$
such that 
\begin{equation}\label{e:qualifdist}
0\in
\sum_{i=1}^m \partial\mathsf{f}_i(\overline{\mathsf{x}})+\nabla \mathsf{h}_i(\overline{\mathsf{x}})
+\mathsf{M}_{i}^*(\partial\mathsf{g}_i\infconv\partial\mathsf{l}_i)\big(\mathsf{M}_i\overline{\mathsf{x}}\big).
\end{equation}
Let $\check{\mathsf{F}}$ be the set of solutions to the
problem
\begin{equation}\label{e:primoptdist}
\minimize{\mathsf{x}\in\HH}
{\sum_{i=1}^m \mathsf{f}_i(\mathsf{x})+\mathsf{h}_i(\mathsf{x})+
(\mathsf{g}_i\infconv\mathsf{l}_i)(\mathsf{M}_i\mathsf{x})}.
\end{equation}
Our objective is to find a $\check{\mathsf{F}}$-valued random variable $\widehat{x}$.
\end{problem}
A proximal algorithm for solving Problem \ref{prob:mainfuncdist} which results from Proposition \ref{p:algopdmon1dist} is described next:
\begin{proposition}\label{p:algopdfunc1dist}
Let $(\theta_\ell)_{1\le \ell \le r}$, $(\mathsf{W}_i)_{1\le i \le m}$, $(\mathsf{U}_i)_{1\le i \le m}$, 
and $\chi$
be defined as in Proposition~\ref{p:algopdmon1dist}.
For every $i\in \{1,\ldots,m\}$, let $\mu_i^{-1}\in \RPP$ be a Lipschitz constant of the gradient of
$\mathsf{h}_i \circ \mathsf{W}_i^{1/2}$ and
let $\nu_i^{-1}\in \RPP$ be a Lipschitz constant of the gradient of $\mathsf{l}_i^* \circ \mathsf{U}_i^{1/2}$.
Suppose that \eqref{e:condvarthetaalphadist} holds
where 
$\mu = \min\{\mu_1,\ldots,\mu_m\}$ and $\nu=\min\{\nu_1,\ldots,\nu_m\}$.
Let $(\lambda_n)_{n\in\NN}$  be a sequence in $\left]0,1\right]$ such that 
$\inf_{n\in\NN}\lambda_n>0$,
let $\boldsymbol{x}_0$, 
$(\boldsymbol{a}_n)_{n\in\NN}$, and 
$(\boldsymbol{c}_n)_{n\in\NN}$ be $\HH^m$-valued random variables,
let $(v_{i,0})_{1\le i \le m}$, $(\boldsymbol{b}_n)_{n\in\NN}$, and 
$(\boldsymbol{d}_n)_{n\in\NN}$ be $\GGG$-valued random 
variables, 
let $(v_{m+\ell,0})_{1\le \ell \le r}$ be a $\HHH$-valued random variable satisfying
\eqref{e:initzero},
and let $(\boldsymbol{\varepsilon}_n)_{n\in\NN}$ be identically distributed 
$\mathbb{D}_{2m+r}$-valued random variables. 
Iterate
\begin{equation}\label{e:betterPDdist}
\begin{array}{l}
\text{for}\;n=0,1,\ldots\\
\left\lfloor
\begin{array}{l}
\text{for}\;\ell=1,\ldots,r\\
\left\lfloor
\begin{array}{l}
\text{for}\;j=1,\ldots,\kappa_\ell\\
\left\lfloor
\begin{array}{l}
\displaystyle w_{\ell,j,n} = \varepsilon_{2m+\ell,n}\big(2\theta_\ell\, (x_{\mathsf{i}(\ell,j),n}-\overline{x}_{\ell,n})+v_{m+\ell,j,n})\\
\end{array}
\right.\\[2mm]
\end{array}
\right.\\
\text{for}\;i=1,\ldots,m\\
\left\lfloor
\begin{array}{l}
\displaystyle u_{i,n} = \varepsilon_{m+i,n}\Big(\prox^{\mathsf{U}_i^{-1}}_{\mathsf{g}_i^*}\big(v_{i,n}+\mathsf{U}_i\big(\mathsf{M}_i
x_{i,n}- \nabla \mathsf{l}_i^*(v_{i,n})+d_{i,n})\big)+b_{i,n}\Big)\\
\displaystyle y_{i,n} =
\varepsilon_{i,n}\Big(\prox^{\mathsf{W}_i^{-1}}_{\mathsf{f}_i}\big(x_{i,n}-\mathsf{W}_i(\mathsf{M}_i^* (2u_{i,n}-v_{i,n})+\sum_{(\ell,j) \in \mathbb{V}_i^*}w_{\ell,j,n}+ 
\nabla \mathsf{h}_i(x_{i,n}) +c_{i,n})\big)+a_{i,n}\Big)\\
v_{i,n+1} = v_{i,n}+\lambda_n \varepsilon_{m+i,n} (u_{i,n}-v_{i,n})\\
x_{i,n+1} = x_{i,n}+\lambda_n \varepsilon_{i,n}\,  (y_{i,n}- x_{i,n})
\end{array}
\right.\\
\text{for}\;\ell=1,\ldots,r\\
\left\lfloor
\begin{array}{l}
\displaystyle v_{m+\ell,n+1} = v_{m+\ell,n}+\frac{\lambda_n}{2} \varepsilon_{2m+\ell,n} (w_{\ell,n}-v_{m+\ell,n})\\
 \eta_{\ell,n} = \max\menge{\varepsilon_{i,n}}{i\in \mathbb{V}_\ell}\\
\displaystyle \overline{x}_{\ell,n+1} = \overline{x}_{\ell,n}+\eta_{\ell,n}\Big(\frac{1}{\kappa_\ell}
\sum_{i\in \mathbb{V}_\ell} x_{i,n+1}-\overline{x}_{\ell,n}\Big),
\vspace*{0.1cm}
\end{array}
\right.
\vspace*{0.1cm}
\end{array}
\right.\\
\end{array}
\end{equation}
where $(\overline{x}_{\ell,0})_{1\le \ell \le r}$
is initialized as in Proposition \ref{p:algopdmon1dist}.
In addition, assume that Conditions~\ref{c:PDcoord1idist}-\ref{c:PDcoord1iiidist} in Proposition \ref{p:algopdmon1dist} hold, where
$(\forall n\in\NN)$ $\EEE_n=\sigma(\boldsymbol{\varepsilon}_n)$ and $\boldsymbol{\XX}_n=
\sigma(\boldsymbol{x}_{n'},\boldsymbol{v}_{n'})_{0\leq n'\leq n}$.\\
Then, under Assumption \ref{a:distrib}, for every $i\in \{1,\ldots,m\}$, $(x_{i,n})_{n\in\NN}$ converges weakly $\as$ to a 
$\check{\mathsf{F}}$-valued random variable $\widehat{x}$ and, for every $\ell\in \{1,\ldots,r\}$, $(\overline{x}_{\ell,n})_{n\in\NN}$ converges weakly $\as$
to $\widehat{x}$.
\end{proposition}
\begin{proof}
For every $i\in \{1,\ldots,m\}$, set $\mathsf{A}_i = \partial \mathsf{f}_i$, $\mathsf{B}_i = \partial \mathsf{g}_i$,
$\mathsf{C}_i = \nabla \mathsf{h}_i$, and $\mathsf{D}_i^{-1} = \nabla \mathsf{l}_i^*$.
In view of \eqref{e:qualifdist} and \cite[Proposition 16.5]{Bauschke_H_2011_book_con_amo}, we have 
\begin{equation} 
0\in
\sum_{i=1}^m \mathsf{A}_i \overline{\mathsf{x}}+\mathsf{C}_i \overline{\mathsf{x}}
+\mathsf{M}_{i}^*(\mathsf{B}_i\infconv\mathsf{D}_i)\big(\mathsf{M}_i\overline{\mathsf{x}}\big)
\subset \partial 
\big(\sum_{i=1}^m \mathsf{f}_i+
\mathsf{h}_i+(\mathsf{g}_i\infconv\mathsf{l}_i)\circ \mathsf{M}_i\big)
(\mathsf{\overline{x}}),
\end{equation}
which shows that $\emp \neq \widehat{\mathsf{F}} \subset \check{\mathsf{F}}$.
This allows us to conclude by applying Proposition \ref{p:algopdmon1dist} and using Remark \ref{re:aldist1}\ref{re:aldist1ii}.
\end{proof}

\begin{remark}\
\begin{enumerate}
\item Alternatively, a second distributed convex optimization algorithm can be deduced from Proposition \ref{p:algopdmon2dist}.
\item If $(\forall i \in \{1,\ldots,m\})$ $\mathsf{g}_i = 0$ and $\mathsf{l}_i = \iota_{\{0\}}$,
$(\forall \ell \in \{1,\ldots,r\})$ $\kappa_\ell = 2$, and \eqref{e:contepseta} holds, then Algorithm~\eqref{e:betterPDdist}
reduces to
\begin{equation}\label{e:bianchi}
\begin{array}{l}
\text{for}\;n=0,1,\ldots\\
\left\lfloor
\begin{array}{l}
\text{for}\;\ell=1,\ldots,r\\
\left\lfloor
\begin{array}{l}
 \eta_{\ell,n} = \max\menge{\varepsilon_{i,n}}{i\in \mathbb{V}_\ell}\\
\displaystyle \widetilde{v}_{\ell,1,n+1} = \widetilde{v}_{\ell,1,n}+\frac{\lambda_n}{2} \eta_{\ell,n} \theta_\ell\, 
(x_{\mathsf{i}(\ell,1),n}-x_{\mathsf{i}(\ell,2),n})\\
\widetilde{v}_{\ell,2,n+1} = - \widetilde{v}_{\ell,1,n+1}\\
\end{array}
\right.\\
\text{for}\;i=1,\ldots,m\\
\left\lfloor
\begin{array}{l}
\displaystyle y_{i,n} =
\varepsilon_{i,n}\Big(\prox^{\mathsf{W}_i^{-1}}_{\mathsf{f}_i}\Big(\big(1-\mathsf{W}_i \overline{\theta}_i\big) x_{i,n}-\mathsf{W}_i\big(\sum_{(\ell,j) \in \mathbb{V}_i^*}(\widetilde{v}_{\ell,j,n}-\theta_\ell\, x_{\mathsf{i}(\ell,\overline{\jmath}),n})+ 
\nabla \mathsf{h}_i(x_{i,n}) +c_{i,n} \big)\Big)\\
\qquad\qquad\quad+a_{i,n}\Big)\\
x_{i,n+1} = x_{i,n}+\lambda_n \varepsilon_{i,n}\,  (y_{i,n}- x_{i,n}),
\end{array}
\right.
\vspace*{0.1cm}
\end{array}
\right.\\
\end{array}
\end{equation}
where we have set  $(\forall j \in \{1,2\})$ $\overline{\jmath} = 3-j$, $(\forall i \in \{1,\ldots,m\})$
$v_{i,0} = 0$, $(\forall n \in \NN)$ 
$\widetilde{\boldsymbol{v}}_n = (v_{m+\ell,n})_{1\le \ell \le r}$, and
$\boldsymbol{b}_n=\boldsymbol{0}$.
The particular case when $\HH$ is an Euclidean space, $(\forall n \in \NN)$ $\lambda_n = 1$, $(\forall \ell \in \{1,\ldots,r\})$ $\theta_\ell = \theta_1$,  
$(\forall i \in \{1,\ldots,m\})$ $\mathsf{W}_i = \tau_i \Id$ with $\tau_i \in \RPP$, and no error term is taken into account
appears to be similar to the distributed iterative scheme developed in \cite{Bianchi_P_2014_stochastic_cdpda}. Then, the sets $(\mathbb{V}_\ell)_{1\le \ell \le r}$ can be viewed as the edges of a connected undirected graph,
the nodes of which are indexed by $i\in \{1,\ldots,m\}$.
\end{enumerate}
\end{remark}

\bibliographystyle{plain}
{\small 
\renewcommand{\baselinestretch}{1.0}

}
\end{document}